\documentclass[reqno,centertags,draft]{amsart}
\usepackage{amsmath,amsthm,amscd,amssymb,latexsym,upref,stmaryrd}
\usepackage{color}

\usepackage[
centering,%includefoot,foot=0.6in
]{geometry}

\makeatletter
\def\theequation{\@arabic\c@equation}

\newcommand{\bbC}{{\mathbb{C}}}

\newcommand{\bbR}{{\mathbb{R}}}

\newcommand{\cB}{{\mathcal B}}

\newcommand{\cE}{{\mathcal E}}

\newcommand{\cH}{{\mathcal H}}
\newcommand{\cI}{{\mathcal I}}

\newcommand{\cL}{{\mathcal L}}

\newcommand{\cP}{{\mathcal P}}
\newcommand{\cQ}{{\mathcal Q}}

\newcommand{\cX}{{\mathcal X}}

\renewcommand{\ker}{\operatorname{ker}}

\renewcommand{\a}{\alpha}

\newcommand{\g}{\gamma}
\renewcommand{\d}{\delta}

\newcommand{\m}{\mu}
\DeclareMathOperator{\ran}{ran}
\DeclareMathOperator{\dom}{dom}

\DeclareMathOperator{\sgma}{Sp}

\renewcommand{\Re}{\text{\rm Re}}

\newcommand{\beq}{\begin{equation}}
\newcommand{\enq}{\end{equation}}

\newcommand{\no}{\notag}
\newcommand{\lb}{\label}

\renewcommand{\ge}{\geqslant}
\renewcommand{\le}{\leqslant}

\let\geq\geqslant
\let\leq\leqslant

\numberwithin{equation}{section}

\renewcommand{\det}{\operatorname{det}}

\newcommand{\ess}{\operatorname{ess}}

\newcommand{\Ran}{\operatorname{ran}}

\newcommand{\Spe}{\operatorname{Sp_{ess}}}
\renewcommand{\Re}{\operatorname{Re }}

\renewcommand{\ker}{\operatorname{ker}}

\newtheorem{theorem}{Theorem}[section]
\newtheorem{proposition}[theorem]{Proposition}
\newtheorem{lemma}[theorem]{Lemma}
\newtheorem{corollary}[theorem]{Corollary}
\newtheorem{definition}[theorem]{Definition}

\newtheorem{hypothesis}[theorem]{Hypothesis}

\theoremstyle{remark}
\newtheorem{remark}[theorem]{Remark}

\begin{document}

\allowdisplaybreaks

\title[Stability]{Stability of a planar front in a multidimensional reaction-diffusion system}
\author{ A. Ghazaryan}
       \address{Miami University, Oxford, OH, USA}
       \email{ghazarar@miamioh.edu}
       \author{ Y. Latushkin}
       \address{Department of Mathematics, University of Missouri, Columbia, MO, USA }
       \email{latushkiny@missouri.edu}
    \author{ X. Yang}
       \address{Department of Mathematical Sciences,  Xi'an Jiaotong-Liverpool University,  Suzhou, Jiangsu, 
P. R. China}
       \email{xinyao.yang@xjtlu.edu.cn}

%\today

\begin{abstract}
	We study the planar front solution for a class of reaction diffusion equations in multidimensional space  in the case when the essential spectrum of the linearization in the direction of the front touches the imaginary axis.  
	At the linear level, the spectrum is stabilized by using an exponential weight.  A-priori  estimates for the nonlinear terms of the equation governing the evolution of the perturbations of the front  are obtained  when perturbations belong to  the intersection of the exponentially weighted  space with the original space without a weight.
These estimates are then used to show that  in the original norm, initially small perturbations to the front  remain bounded,  while in the exponentially weighted norm, they  algebraically decay in time. 
\end{abstract}

\let\thefootnote\relax
  
\footnote{{\it Acknowledgements:}  YL was supported by the NSF grant  DMS-1710989, by the Research Board and Research Council of the University of Missouri,
and by the Simons Foundation.
AG was supported by the NSF grant  DMS-1311313. 
}

\maketitle

\section{Introduction}

Planar traveling fronts are solutions to  partial differential equations posed on multidimensional  infinite domains that move in a preferred direction with constant speed without changing their shape and that are asymptotic to spatially constant steady-state solutions. %The fronts are called planar if the  equations  are posed on multidimensional  domains.

Stability theory of the traveling fronts   in reaction-diffusion equations
is a vast subject that has a  long  history and is very active today, see, e.g.\ \cite{henry,KP,Sa02,VVV} and the literature cited in these books, as well as \cite{BGHL,BKSS,Kapitula2,KV,LX,liwu,rott2a,rott3,rott4,rmthesis,TZKS, Xin} 
and the bibliography therein. 

The cornerstone of the stability analysis of the fronts (or pulses), in general, is  to determine the location of the spectrum of the linearization of the underlying  system about the wave.   The spectrum may contain isolated eigenvalues of finite algebraic  multiplicity and the essential spectrum; the latter may  %is called essential and 
consist of curves and domains filled with spectrum, which is due to the dynamics near the asymptotic rest states of the  wave.   Presence of unstable discrete eigenvalues points to the absolute instability of the  wave when the perturbations to the wave  grow  exponentially  and eventually lead to unrecoverable distortion of the wave. Absence of unstable spectrum indicates the resilience of the wave to small perturbations, if the nonlinear effects are negligibly small compared to the linear dynamics.   In the case when   the only unstable spectrum   is a subset  of the  essential spectrum on the imaginary axis, the balance between linear growth and nonlinear effects becomes crucial.  We call  the essential spectrum marginally unstable if it extends up to the imaginary axis. Marginally unstable essential spectrum in reaction-diffusion systems as  well as discrete eigenvalues located on the imaginary axis are  indicative of an instability. In this paper we are interested in identifying the character of  instability of a planar front with marginally unstable spectrum.

Although a great deal of the literature is devoted to the multidimensional reaction-diffusion equations \cite{BKSS,LMNT, LX, LW, PSS, T, Xin, Kapitula2},  the theory in this case is still not as well developed as in the one-dimensional situation. 
One of the most high impact works is  the 1997 paper \cite{Kapitula2} by T.~Kapitula who  demonstrated,  under very general conditions, that  the stability  of a multidimensional planar front  is related to the stability of the associated one-dimensional front profile. More precisely,  T.~Kapitula  in  \cite{Kapitula2}
proved the algebraic decay of  perturbations to a planar front in a general reaction-diffusion system  in case when  the spectrum of the linearization along the associated one-dimensional front  is located in the stable half plane. The case of marginally unstable essential spectrum  has been open since 1990. In the current work  it is settled  blue for a special class of reaction-diffusion systems.

For the problems posed on one-dimensional space that exhibit  traveling waves  with  marginally unstable essential spectrum, there exists an important   technique for  stability analysis    based on applying exponential weights. 
It  goes back to the celebrated work \cite{pw94,sattinger} and   amounts to recalculating the spectrum of the operator obtained by linearizing the equation about the
wave in a  function space equipped with an exponential weight.    Since 
the exponential weights in some situations may stabilize  the system at the linear level by shifting the 
essential spectrum of the linearization  into the stable half-plane,  one can then exploit the decay of the related linear
semigroup to investigate whether the nonlinear effects in the underlying system are negligible in the introduced exponentially weighted norm. Instability of the essential spectrum in the original norm and stability of the wave in an exponentially weighted norm   point to the convective nature of  instability \cite{SS} which is the instability characterized by point-wise   decay of the perturbations.  
To the best of our knowledge,  the current paper is the first where this technique  is used   to analyze   the stability of multidimensional traveling waves with marginal or unstable essential spectrum.   %Moreover, it seems that prior to this paper there were no results that address stability of multidimensional fronts or pulses beyond spectral level. 
We mention, however, an important  paper \cite{BKSS},  where T. Brand, M. Kunze, G. Schneider, and  T. Seelbach successfully used combination of weights in a reaction-diffusion-convection system to investigate the nonlinear stability of the zero solution. The reaction terms in \cite{BKSS} are assumed to be exponentially localized unlike the reaction terms  considered  here.

In a recent series of papers  \cite{G,GLS1,GLS2,GLS3}  the method  of exponential weights was used for  a traveling front with marginally unstable essential spectrum in a class of reaction-diffusion systems posed in  one space dimension.  These equations  originate in combustion theory where the  combustion front   captures the propagation of   the highest temperature zone. We refer to  \cite{GLSR}, \cite{BGHL},\cite{BM}, \cite{SKMS1}, \cite{VaV} and references therein, for the results on  existence and spectral stability of the combustion fronts.
%This class of system support a traveling front with marginally unstable essential spectrum (we call essential spectrum marginally unstable if it extends to the imaginary axis).
   A recent paper \cite{GLSR}  contains an
overview of the exponential weights technique and further references. 
The results were formulated as the orbital stability of the traveling front in an exponentially weighted norm, against perturbations that belong to the space obtained by intersecting the original space (a Sobolev space or the space of bounded uniformly continuous functions) with the same space but equipped with the exponentially weighted norm.  The orbital stability of the wave in the  exponentially  weighted norm  may be interpreted as  convective instability.

In the current  paper, we investigate  nonlinear  stability of the planar front of a certain special class of systems of reaction-diffusion equations.   The front is assumed to have a special feature: for the associated one-dimensional front,  the instability is  related not to the presence of unstable eigenvalues, but to the  essential spectrum touching imaginary axis. More precisely, our objective  is to relate the type of the stability  of the  planar front  to the convective nature of the instability of the associated one-dimensional fronts.  
The  system considered in the current paper has  a certain ``product-triangular" structure in the reaction term similar to that of the equations studied in \cite{GLS1,GLS2,GLS3} for the one-dimensional case. Indeed,  our motivation comes from the following system,
\begin{equation}\label{MPr}
\begin{cases}u_{1t}(t,x)=\Delta_{x} u_1(t,x)+u_2(t,x) g(u_1(t,x)),\\
u_{2t}(t,x)=\epsilon\Delta_{x} u_2(t,x)-\kappa u_2(t,x) g(u_1(t,x)),\end{cases}
\end{equation}
with
\begin{equation}\label{deng}
g(u_1)=\begin{cases} e^{-\frac{1}{u_1}} &\mbox{ if } u_1>0 ;\\
0 &\mbox{ if } u_1\leq 0,\end{cases}
\end{equation}
where  $u_1$, $u_2\in\mathbb{R}$,  $t\in \mathbb{R}^+$, $x\in\mathbb{R}^d$ ($d\geq 2$),  and 
the parameters $\epsilon$ and $\kappa$ satisfy $0\leq\epsilon<1$ and $\kappa>0$.  

Ultimately, we would like to develop a technique to study  nonlinear stability in weighted spaces  for a marginally  unstable 
front in the general system of the type
 \begin{equation}\label{GEQ_D}
u_t(t,x)=D \Delta_{x}u(t,x)+ f(u(t,x)),\, u\in\mathbb{R}^n,\,\,x\in\mathbb{R}^d,\, \,d\geq 2,\,\, t\in \mathbb{R}^+,
\end{equation}
but here  we focus on the case when the diffusion matrix  $D$ is the identity matrix and  the reaction terms satisfy some additional assumptions that are described in Section~\ref{subs3.2.1}.    
The restriction imposed by choosing   the identity matrix to describe diffusion is not necessarily technical. Distinct  diffusion rates are known to be responsible for a variety of dynamical phenomena such as Turing instability or sideband instability.   Removing this condition  or identifying  situations when this condition can be removed is generally speaking  an  open problem.   A stronger   assumption  sufficient to guarantee asymptotic stability of a planar traveling wave against  localized in transverse direction perturbations  in the reaction-diffusion  systems with non-equal diffusion, in addition to   the spectral stability of  the associated one-dimensional front  includes  quadratic tangency of the dispersion relation near critical Fourier mode. We refer readers to \cite{HS} for the proof and the discussion.

To summarize,  for a class of reaction terms  in   the system \eqref{GEQ_D}  with an identity diffusion matrix, we  developed  a
   technique  that effectively combines  the approach introduced in \cite{Kapitula2}  with techniques from  \cite{GLS1,GLS2,GLS3} to   prove   nonlinear stability  in a weighted multidimensional  space for a  planar front  that has  unstable or marginally stable essential spectrum.

\section{The plan of the paper and notations}\lb{subs3.2.0}

The plan of the paper is as follows. In Section \ref{subs3.2.1} we list the assumptions  imposed on the  system. We study the spectrum of the operator obtained by linearizing the system about the planar front in Section \ref{subs3.2.2} and obtain the estimates of the semigroup generated by the linear operator in Section \ref{sec4}. In Section \ref{subs3.2.3} we derive a system of partial differential equations for the perturbations of the planar front to be studied, and in Section \ref{sec6} we estimate the nonlinear terms in the system. We complete the proof of  the stability of the front in Section \ref{subs3.2.4}. 

We consistently  use the same symbol to denote the space of scalar valued functions and the space of respective vector valued functions, whenever it is clear from the context, e.g.,  we use the same notation $H^k(\mathbb{R}^d)$  for the Sobolev space of scalar  functions and for the Sobolev space of vector   functions   $(H^k(\mathbb{R}^d))^n$ when $n>1$.

   In the physical space $\mathbb{R}^d$ we associate $z\in \mathbb{R}$ with the direction of the propagation of the planar front. The complement of $z$ we denote $y \in\mathbb{R}^{d-1}$. 
  
    For a fixed weight function $\gamma_{\alpha}(z)$ and $(z,y)\in\bbR^d$, we denote $H^k_{\a}(\bbR)=\{v:\gamma_{\a}v\in H^k(\bbR)\}$, and  $  
 H^k_{\a}(\bbR^{d}) 
 =\{ u: (z,y)\mapsto 
 \gamma_{\alpha}(z)u(z,y)
 %\gamma_{\alpha}\otimes I_{H^k(\mathbb{R}^{d-1})} u
 \in H^k(\mathbb{R}^d) \}$. The spaces are equipped 
 with the norm $\|u\|_{H^k_{\alpha}(\bbR^d)}=\|\gamma_{\alpha} u\|_{H^k(\mathbb{R}^{d})}$.
Also,  we use $\mathcal{H}$ to denote the intersection space
\begin{equation} \label{defcH}
\mathcal{H}:= H^k(\mathbb{R}^d)\cap {H^k_{\alpha}(\mathbb{R}^d)},\, \text{with}\,\|u\|_{\mathcal{H}}=\max\{ \|u\|_{H^k(\bbR^d)},\|u\|_{H^k_{\alpha}(\bbR^d)} \}.\end{equation}

Throughout, $\mathcal{B}(X,Y)$ denotes  the space of bounded operators from $X$ to $Y$, and  we abbreviate $\mathcal{B}(X)=\mathcal{B}(X, X)$. 
 We denote by $Sp(\mathcal{T}) $   and  $Sp_{ess}(\mathcal{T}) $ the spectrum and the essential spectrum  of the operator  $\mathcal{T}$,    and  by $ \ran \mathcal{T}$ and $ \ker \mathcal{T}$  its range and nullspace. 
 Throughout the paper, we   denote by $C$  a generic positive constant. % that are not necessarily the same.

\section{Assumptions}\lb{subs3.2.1}

In this section we  introduce the   reaction-diffusion system to be studied. In the  one-dimensional situation similar  % and describe  the planar front that we assume it supports. 
 assumptions on the system  were  originally developed in  \cite{GLS2,GLS3}. %  for  the case of a one-dimensional front.  %that is exhibited by a system of reaction-diffusion equations of a special  type.  
We consider the system of reaction-diffusion equations
\begin{equation}\label{eq3.1.1}
u_t(t,x)=\Delta_{x}u(t,x)+f(u(t,x)), 
\end{equation}  where $u(t,x)\in\mathbb{R}^n$, $x\in\mathbb{R}^{d}$, $t\in\mathbb{R}^{+}$, 
%$D=(d_1,...,d_n)$ with $d_i>0$ for $i=1,...,n$, 
and the function $f(\cdot):\bbR^n\rightarrow\bbR^n$ is smooth.

We assume that  this system has a planar  wave
 %Without loss of generality, we  can say that the planar front 
that moves in the direction of the vector $e=(1,0,...,0)\in\mathbb{S}^d$ with a certain speed $c>0$. %We then  make a change of variable $z=x_1-ct $ for some $c>0$ that represents the speed of the wave. 
In the co-moving frame  $z=x_1-ct $,  %with an abuse of the notation, we denote again $x=(z,x_2,...,x_d)$,  and 
 \eqref{eq3.1.1}  reads
\begin{equation}
\label{sysu}
u_t=\Delta u+cu_z+f(u),
\end{equation}
where $\Delta=\partial_{z}^2+\partial_{x_2}^2+\cdots+\partial_{x_n}^2$. 

  A traveling wave  $\phi(z)$ for system \eqref{eq3.1.1} is a time-independent function of $z\in\bbR$, such that  
 % \begin{equation}
%\label{sysu}
%u_t=  u_{zz}+cu_z+f(u).
%\end{equation}
%Moreover, it is s a time independent solution of \eqref{sysu}, therefore 
\begin{equation}\label{sys1dim}
0=\frac{ d^2\, \phi}{d \, z^2}+c\frac{d\, \phi}{ d \, z}+f(\phi).
\end{equation}
 We further assume that the wave converges to its rest states  $\phi_{\pm}\in\bbR^n$  exponentially. %so that $f(\phi_{\pm})=0$ and 
% i.e. there exist constants $K>0$ and $\omega_-<0<\omega_+$ such that
%\begin{equation}
%\label{EDC}
%\|\phi(z)-\phi_{-}\|_{\bbR^n}\leq Ke^{-\omega_- z}\ \text{for}\ z\leq 0,\quad \text{and}\  \|\phi(z)-\phi_{+}\|_{\mathbb{R}^n}\leq Ke^{-\omega_+ z}\ \text{for }\ z\geq 0.
%\end{equation}
The  wave is called  a front if $\phi_-\neq \phi_+$, or, otherwise, it is called a pulse. Without loss of generality, we  assume that $\phi_-=0$.% and study the stability of the planar front $\phi$ of \eqref{eq3.1.1}.

To study the stability of $\phi$, we first linearize \eqref{sysu} about $\phi$. We define the linear variable-coefficient differential expression  $L$ by
\begin{equation}\label{eq3.1.2}
L=\Delta+c\partial_z+df(\phi),
\end{equation}
where $df(\phi)$ is the differential of the function $f$ evaluated at $\phi(\cdot)$.
 The linear stability of the front is determined by the  spectral information of the  operator $\mathcal{L}$ associated with %the differential expression \eqref{eq3.1.2}
    $L$  and acting on the Sobolev space $H^k(\mathbb{R}^d)^n $ for $k>1$.  %which  is a space of complex-valued vector-functions with $n$ components  with each component  from $H^k(\mathbb{R}^d)$.  
 For $k>\left[\frac{d}{2}\right ]$  the spaces $H^k(\mathbb{R}^d)$  are, in fact, Banach algebras and thus
are  convenient for  the nonlinear stability analysis. %Here, $H^0(\bbR^d)=L^2(\bbR^d)$. 

%For a fixed $k$, we denote $ H^k(\mathbb{R}^d)=H^k(\mathbb{R}^d)$ with  the norm  $\|\cdot\|_0$.
Using the tensor product notation, we write $H^k(\mathbb{R}^d)=H^k(\mathbb{R})\otimes H^k(\mathbb{R}^{d-1})$, and note that for any $u\in H^k(\mathbb{R})$ and $v\in H^k(\mathbb{R}^{d-1})$ the function $(z,x_2,...,x_d)\mapsto u(z)v(x_2,...,x_d)$ belongs to $H^k(\mathbb{R}^d)$. From now on, we decompose $x\in\mathbb{R}^d$ as $x=(z,y)\in\mathbb{R}\otimes\mathbb{R}^{d-1}$, where $z=x_1-ct$ and $y=(x_2,...,x_d)$.
Thus we can use the decomposition of $\cL$ on $ H^k(\mathbb{R}^d)$ as follows, 
\begin{equation*}%\label{L} 
\mathcal{L}=\mathcal{L}_{1}\otimes I_{H^k(\mathbb{R}^{d-1})^n}+I_{H^k(\mathbb{R})^n}\otimes
%(\textcolor{red}{I_{H^k(\mathbb{R}^{d-1})^n}} 
\Delta_y,\end{equation*}
 where $\mathcal{L}_1$ is associated with the one-dimensional  differential expression
  \begin{equation}\label{eq3.2.1}
L_1=\partial_z^2+c\partial_z+df(\phi),
\end{equation}
that depends only on $z$, and
\begin{equation}\label{eq3.2.2}%H^k(\mathbb{R}^d)
\Delta_y=\partial_{x_2}^2+\cdots+\partial_{x_d}^2.
\end{equation}
%\textcolor{red}{Did you really mean to keep D in the formula above? It contradicts the formula after (2.11) but is consistent with what is done after. It is irrelevant starting from page 9 though. I would keep D out of $\Delta_y$  but introduce it in  $\mathcal{L}$}

We next introduce an exponential weight to counteract the marginally unstable essential spectrum.
%We will be concerned with the problem that on $ H^k(\mathbb{R}^d)^{n}$ the essential spectrum of the linear operator associated with $L_1$ as in \eqref{eq3.2.1} may touch the imaginary axis. To fix this, one introduces a class of weight functions of exponential type in order to stablize the spectrum when it is possible. 
 %Let $\alpha=(\alpha_-,\alpha_+)\in\mathbb{R}^2$. 
 We call  $\gamma_{\alpha}\in C^{k+3}(\mathbb{R})$ the  weight function of class $\alpha=(\alpha_-,\alpha_+)\in\mathbb{R}^2$  if $\gamma_{\alpha}(z)>0$ for all $z\in\bbR$, and
\begin{equation}\label{eq3.1.7}
\gamma_{\alpha}(z)=\begin{cases}
e^{\alpha_-z}\, ,\text{ for  $z$ negative, $|z|$ large},\\
e^{\alpha_+z}\, ,\text{ for $z$ positive and  large}.
\end{cases}
\end{equation}
%Following the setting in \cite{GLS2}, we will always assume that $0<\alpha_-<-\omega_-$ and $0\leq \alpha_+<\omega_+$, where $\omega_-$ and $\omega_+$ were introduced in \eqref{EDC}. The condition $\alpha_+<\omega_+$ ensures that $\phi'=\partial_z \phi $ is in the weighted space, and the conditions $\alpha_-<\omega_-$ and $0\leq \alpha_+$ ensure that $\gamma_{\alpha}^{-1}(z)\phi(z)$ is bounded.

For a fixed weight function $\gamma_{\alpha}$, let $H^k_{\a}(\bbR):=\{v:\gamma_{\a}v\in H^k(\bbR)\}$.  We then denote 
 $$%{H^k_{\alpha}(\mathbb{R}^d)} 
 H^k_{\a}(\bbR^{d}) =H_{\alpha}^k(\mathbb{R})\otimes H^k(\mathbb{R}^{d-1})=
 \{ u: (\gamma_{\alpha}\otimes I_{H^k(\mathbb{R}^{d-1})}) u\in H^k(\mathbb{R}^d) \},$$ 
 with the norm $\|u\|_{H^k_{\alpha}(\bbR^d)}=\|\gamma_{\alpha} u\|_{H^k(\mathbb{R}^{d})}$. Here, % for a function $u=u(z,y)$ we denote %by $\gamma_{\alpha}\otimes I_{H^k(\bbR^{d-1})} u$ the function of $(z,y)$ defined by 
$
(\gamma_{\alpha}\otimes I_{H^k(\bbR^{d-1})}) u(z,y):=\gamma_{\alpha}(z)u(z,y)$, $(z,y)\in\bbR^d$.
%Note that  then %by the definition of ${H^k_{\alpha}(\mathbb{R}^d)}$,
 %we can represent 
% ${H^k_{\alpha}(\mathbb{R}^d)}=H_{\alpha}^k(\mathbb{R})\otimes H^k(\mathbb{R}^{d-1})$, where $H^k_{\a}(\bbR):=\{v:\gamma_{\a}v\in H^k(\bbR)\}$. 

%Also, throughout the paper we use $\mathcal{H}$ to denote the following intersection space
%\begin{equation} \label{defcH}
%\mathcal{H}:= H^k(\mathbb{R}^d)\cap {H^k_{\alpha}(\mathbb{R}^d)},\, \text{with}\,\|u\|_{\mathcal{H}}=\max\{ \|u\|_{H^k(\bbR^d)},\|u\|_{H^k_{\alpha}(\bbR^d)} \}.\end{equation}

\begin{definition}\label{op} $\,$ % \footnote{ We use the same symbol to denote the space of scalar valued functions and the space of respective vector valued functions, wherever it is clear from the context, e.g.,  we use the same notation $H^k(\mathbb{R}^d)$  for the space $(H^k(\mathbb{R}^d))^n$ when $n>1$.}
	\begin{enumerate}
	\item   $\mathcal{L}: H^k(\mathbb{R}^d)^n\rightarrow  H^k(\mathbb{R}^d)^n$ is the linear operator given by the formula  $u\mapsto L u$, with $L$ as in \eqref{eq3.1.2} where $\dom{\mathcal{L}}=H^{k+2}(\mathbb{R}^d)^n\subset  H^k(\mathbb{R}^d)^n$, for $k=1,2,\dots$; 
	\item  $\mathcal{L}_1:H^k(\mathbb{R})^n\rightarrow H^k(\mathbb{R})^n$ is  the linear operator given by the formula  $u\mapsto L_1 u$ with $L_1$ as in \eqref{eq3.2.1}, where $\dom{\mathcal{L}_1}=H^{k+2}(\mathbb{R})^n\subset H^k(\mathbb{R})^n$; 
	\item  $\Delta_y:H^k(\mathbb{R}^{d-1})^n\rightarrow H^k(\mathbb{R}^{d-1})^n$ is the linear operator given by the formula  \eqref{eq3.2.2}, with the domain  $H^{k+2}(\mathbb{R}^{d-1})^n$; 
	\item  $\mathcal{L}_{\alpha}:{H^k_{\alpha}(\mathbb{R}^d)}^n\rightarrow {H^k_{\alpha}(\mathbb{R}^d)}^n$  is the operator given by the formula $u\mapsto L u$,  with ${L}$ as in \eqref{eq3.1.2}  and $\dom{\mathcal{L}_{\alpha}}=H^{k+2}_{\alpha}(\mathbb{R})\otimes H^{k+2}(\mathbb{R}^{d-1})$;
	\item   $\mathcal{L}_{1,\alpha}:H^k_{\alpha}(\mathbb{R})^n\rightarrow H^k_{\alpha}(\mathbb{R})^n$ is the operator given by $u\mapsto L_1 u$,
	 with $L_1$ as in \eqref{eq3.2.1}
	 and $\dom{\mathcal{L}_{1,\alpha}}=H^{k+2}_{\alpha}(\mathbb{R})^n\subset H^k_{\alpha}(\mathbb{R})^n$;
%	\item  \label{defcH}$\mathcal{H}:= H^k(\mathbb{R}^d)\cap {H^k_{\alpha}(\mathbb{R}^d)},\, \text{with}\,\|u\|_{\mathcal{H}}=\max\{ \|u\|_{H^k(\bbR^d)},\|u\|_{H^k_{\alpha}(\bbR^d)} \}$;
\item $\mathcal{L}_{\mathcal{H}}:\mathcal{H}^n\rightarrow\mathcal{H}^n$ is the linear operator generated by $u\rightarrow Lu$ with $L$ as in \eqref{eq3.1.2}, with  the domain $ H^{k+2}(\mathbb{R}^d)\cap {H^{k+2}_{\alpha}(\mathbb{R}^d)}$.%, where $\dom(\mathcal{L})$ and $\dom(\mathcal{L}_{\alpha})$ are respective domains defined above.
\end{enumerate}	
\end{definition}

We summarize the assumptions on the system \eqref{eq3.1.1} considered on $\mathcal{H}$  as follows.
%and assume that the following hypotheses hold. % are true for \eqref{eq3.1.1} in this paper. 
%Recall that $ H^k(\mathbb{R}^d)=H^k(\mathbb{R}^d)$.
\begin{hypothesis}\label{hypo3.2.1}
	The function $f:\bbR^n\to \bbR^n$ is in $C^{k+3}(\bbR^n)^n$.
\end{hypothesis}
\begin{hypothesis}\label{hypo3.2.2}
	The system \eqref{eq3.1.1} has a $C^{k+5}$-smooth  planar front $\phi(z)$, $z=x_1-ct$,  $\lim_{z\to \pm \infty} \phi(z) =\phi_{\pm}$,  for which there exist numbers $K>0$ and $\omega_-<0<\omega_+$ such that
	\begin{equation*}
%\label{EDC}
\|\phi(z)-\phi_{-}\|_{\bbR^n}\leq Ke^{-\omega_- z}\ \text{for}\ z\leq 0,\quad \text{and}\  \|\phi(z)-\phi_{+}\|_{\mathbb{R}^n}\leq Ke^{-\omega_+ z}\ \text{for }\ z\geq 0.
\end{equation*}
	%$$\|\phi(z)-\phi_-\|\leq Ke^{-\omega_- z}\ \text{for}\ z\leq 0\quad\text{and} \quad\|\phi(z)-\phi_+\|\leq Ke^{-\omega_+ z}\ \text{for }\ z\geq 0.$$
\end{hypothesis}

%Hypotheses \ref{hypo3.2.1} and \ref{hypo3.2.2} imply that $f(\phi_{\pm})=0$. XXXXXXXXXXXXXXXXXXXXXXX

%The following lemma  follows.
%\begin{lemma}\label{lem3.2.3}
%	There exists $K>0$ such that the following is true: $\|\phi^{(m)}(z)\|\leq Ke^{-\omega_-z}$ for $z\leq 0$, and $\|\phi^{(m)}(z)\|\leq Ke^{-\omega_+z}$ for $z\geq 0$, here $m=1,...,k+1$.
%\end{lemma}
%\begin{remark}
%	The $z$-derivative $\phi'$ of the planar front $\phi$ satisfies the equation $L\phi'=0$. %  this follows by taking $z$-derivative of the equation \eqref{sysu}.
%	\hfill$\Diamond$
%\end{remark}

Without loss of generality, in the rest of the paper we assume that  $\phi_-= 0$.

\begin{hypothesis}\label{hypo3.2.3}
	%In addition to Hypotheses \ref{hypo3.2.1} and \ref{hypo3.2.2}, we assume that 
	There exists $\alpha=(\alpha_-,\alpha_+)\in\mathbb{R}^2$ such that the following assertions hold:
	\begin{itemize}
		\item[(1)] $0<\alpha_-<-\omega_-$.
		\item[(2)] $0\leq \alpha_+<\omega_+$.
		\item[(3)] For the linear operator $\mathcal{L}_{1,\alpha}:H^k_{\alpha}(\mathbb{R})^n\rightarrow H^k_{\alpha}(\mathbb{R})^n$, there exists $\nu>0$ such that $$\sup\{ \Re \lambda:\lambda\in \Spe(\mathcal{L}_{1,\alpha}) \}<-\nu,$$ and the only element of $\sgma(\mathcal{L}_{1,\alpha})$ in $\{ \lambda: \Re\lambda\geq 0 \}$ is a simple eigenvalue $0$.
	\end{itemize}
\end{hypothesis}

Hypotheses \ref{hypo3.2.2} and  \ref{hypo3.2.3} imply the following lemma.
\begin{lemma}\label{lem3.2.6} 
	If  Hypotheses \ref{hypo3.2.2} and \ref{hypo3.2.3} hold, then
		\begin{itemize}
	\item[(1)] $\gamma_{\alpha}^{-1}\phi$ is a $ C^{k+5}(\mathbb{R})^n$ function that  approaches zero  exponentially  as $z\rightarrow \pm \infty$. %or $z\rightarrow-\infty$ together with $\gamma_{\alpha}^{-1}\phi^{(m)}$, $m=0$, $1$,$ 2$,...,$k$. % exponentially approach zero as $z\rightarrow+\infty$ or $z\rightarrow-\infty$ for $m=0,1,2,...,k$.
		\item[(2)] $\gamma_{\alpha}\phi$ is a $ C^{k+5}(\mathbb{R})^n$ function that  exponentially  approaches  infinity as $z\rightarrow \infty$ and zero as $z\rightarrow-\infty$,  while $\gamma_{\alpha}\phi^{(m)}$   approaches zero exponentially as $z\rightarrow \pm \infty$, for  any $m=1$, $2$, ..., $k+5$.	
	\end{itemize}
\end{lemma}
An addition, we assume that the nonlinearity in system \eqref{eq3.1.1}  satisfies the following hypothesis: 
\begin{hypothesis}\label{hypo3.2.5}
	There exists a splitting $u=(u_1,u_2) \in \mathbb{R}^n$, where  $u_1\in \mathbb{R}^{n_1}$  and  $u_2\in\mathbb{R}^{n_2}$, and $n_2+n_1=n$, such that  $f(u_1,0)=(A_1u_1,0)$ for some 
	  $n_1\times n_1$ constant matrix $A_1$.  %and  $u_1\in\mathbb{R}^{n_1}$, $n_1<n$,  such that $f(u_1,0)=(A_1u_1,0)$.
\end{hypothesis}

%If  $u=(u_1,u_2) \in \mathbb{R}^n$, where  $u_1\in \mathbb{R}^{n_1}$  and  $u_2\in\mathbb{R}^{n_2}$, and $n_2=n -n_1$, then 
From Hypotheses \ref{hypo3.2.1} and  \ref{hypo3.2.5}  we have the following representation of function $f$, 
\begin{equation}\label{eq3.13}
f(u)=\begin{pmatrix}
f_1(u_1,u_2)\\ f_2(u_1,u_2)
\end{pmatrix}= \begin{pmatrix}
A_1u_1+\tilde{f}_1(u_1,u_2)u_2\\\tilde{f}_2(u_1,u_2)u_2
\end{pmatrix},\,\,f_i:\mathbb{R}^{n_1}\times\mathbb{R}^{n_2}\rightarrow \mathbb{R}^{n_i}\,\, ,i=1,2,%,\,\, %D=\begin{pmatrix}D_1 & 0\\0 & D_2\end{pmatrix}.
\end{equation}
%where
%\begin{equation}\label{eq3.13}
%f(u_1,u_2)=\begin{pmatrix}
%f_1(u_1,u_2)\\ f_2(u_1,u_2)
%\end{pmatrix} = \begin{pmatrix}
%A_1u_1+\tilde{f}(u_1,u_2)u_2\\\tilde{f}_2(u_1,u_2)u_2
%\end{pmatrix},
%\end{equation}
where $\tilde{f}_1$ and $\tilde{f}_2$ are matrix-valued functions of size $n_1\times n_2$ and $n_2\times n_2$, respectively. 

The system  \eqref{eq3.1.1} in terms of $u_1$ and $u_2$ reads
\begin{align*}%\label{decomp}
&\partial_t u_1=%D_1 
\Delta_x u_1+ f_1(u_1,u_2),\no\\
&\partial_t u_2=%D_2 
\Delta_x u_2+ f_2(u_1,u_2),
\end{align*}
and the system \eqref{sysu} reads
\begin{align*}
&\partial_t u_1=%D_1
 (\partial_{zz}+\Delta_y) u_1+ c\partial_{z}u_1+f_1(u_1,u_2),\\
&\partial_t u_2=%D_2
 (\partial_{zz}+\Delta_y) u_2+ c\partial_{z}u_2+f_2(u_1,u_2).
\end{align*}
Similarly, we   write $\phi(z)=(\phi_1(z),\phi_2(z))$ and $ \phi_+=(\phi_{1,+},\phi_{2,+})$ and  the differential  expressions obtained by linearizing  \eqref{sys1dim} at $0$ and $\phi_+$, respectively, are given by the formulas
\begin{equation}\label{eq3.2.5}
L^-_1=%D
\partial_{zz}+c\partial_z+df(0),\quad L^+_1=%D
\partial_{zz}+c\partial_z+df(\phi_+).
\end{equation}
In relation to the linearization about  $0$, we denote
\begin{equation}\label{eq3.2.6}
L^{(1)}_1=%D_1
\partial_{zz}+c\partial_z+d_{u_1}f_1(0,0)=%D_1
\partial_{zz}+c\partial_{z}+A_1,
\end{equation}
\begin{equation}\label{eq3.2.7}
L^{(2)}_1=%D_2
\partial_{zz}+c\partial_z+d_{u_2}f_2(0,0).
\end{equation}
where $d_{u_i}f_i$ is the Jacobian of $f_i$ with respect to $u_i$, $i=1$, $2$.
%For $i=1,2$, obviously, $L^{(i)}_1$ is a constant coefficient linear differential expression on $\mathbb{R}^{n_i}$. 
From  \eqref{eq3.13} it follows then  that
\begin{equation}\label{eq3.2.9}
L^-_1=\begin{pmatrix}
L^{(1)}_1 & d_{u_2}f_1(0,0)\\
0 & L^{(2)}_1
\end{pmatrix}.
\end{equation}
 We  write $L_1$  defined in  \eqref{eq3.2.1} as follows,
%\begin{equation*}
%L_1=L^-_1+\left(df(\phi)-df(0)\right),
%\end{equation*}
%and then from \eqref{eq3.2.9}, we finally have
\begin{equation}\label{eq3.2.13}
L_1\begin{pmatrix}
u_1\\u_2
\end{pmatrix}=\begin{pmatrix}
L_1^{(1)} & d_{u_2}f_1(0,0)\\
0 & L^{(2)}_1
\end{pmatrix}\begin{pmatrix}
u_1 \\ u_2
\end{pmatrix}+\left(df(\phi)-df(0)\right)\begin{pmatrix}
u_1 \\ u_2
\end{pmatrix}.
\end{equation}

We  now define  the operators $\mathcal{L}_1^{(1)}$ and $\mathcal{L}_1^{(2)}$ as prescribed in item (2) of Definition~\ref{op}. The 
next hypothesis implies, in part, the stability of the end state $(0,0)$ located behind the front.

\begin{hypothesis}\label{hypo3.2.7} In addition to Hypotheses \ref{hypo3.2.3} and \ref{hypo3.2.5}, we assume that  the following is true. 
	\begin{itemize}
		\item[(1)] The analytic semigroup generated by the operator $\mathcal{L}_1^{(1)}$ on $H^k(\mathbb{R})^{n_1}$ induced by \eqref{eq3.2.6} in $H^k(\mathbb{R})$ is  bounded, that is,  there exists $K>0$ such that  $\|e^{t\mathcal{L}_1^{(1)}}\|_{\mathcal{B}(H^k(\mathbb{R}))}\leq K$ for  all $t\geq 0$;
		\item[(2)] The spectrum  $\sgma(\mathcal{L}^{(2)}_1)$ of the operator $\cL^{(2)}_1$ on $H^k(\mathbb{R})^{n_2}$ introduced by \eqref{eq3.2.7} is  located strictly to the left of the imaginary axis, that is,  $\sup\{ \Re\lambda:\lambda\in \sgma(\mathcal{L}^{(2)}_1) \}<0$. Therefore, 
		 there exist constants $\rho>0$ and $K>0$ such that   $\|e^{t\mathcal{L}_1^{(2)}}\|_{\mathcal{B}(H^k(\mathbb{R}))}\leq Ke^{-\rho t}$ for all $t\geq 0$.

	\end{itemize}
\end{hypothesis}

\begin{remark}\label{hypo3}

	 Hypothesis \ref{hypo3.2.7} implies that 
	%\begin{itemize}
		%\item[(a)]
		 (a) $\sup\{ \Re \lambda:\lambda\in \sgma(\mathcal{L}_1^{(1)}) \}\leq 0$;
		%\item[(b)] 
		(b) $\sup\{ \Re \lambda:\lambda\in \sgma(\mathcal{L}_1^{-}) \}\leq 0$.
%	\end{itemize}

\end{remark}

\section{Spectrum and projection operators}\lb{subs3.2.2}

In this section we  discuss the projection  operator on the central direction that corresponds to the isolated zero eigenvalue of the linear operator $\cL_{1,\alpha}$ associated with \eqref{eq3.2.1}, on the weighted space $H^k_{\alpha}(\bbR)^n$,  and describe the central projection  for  the operator $\cL_{\alpha}$ in $H^k_{\alpha}(\bbR^d)^n$. 

We recall that for a closed densely defined operator $\mathcal{T}$, the resolvent set $\rho(\mathcal{T})$ is the set of $\lambda\in\bbC$ such that $\mathcal{T}-\lambda I$ has a bounded inverse. The complement of $\rho(\mathcal{T})$ is the spectrum $\sgma(\mathcal{T})$. It includes the discrete spectrum, $\sgma_d(\mathcal{T})$, which is the set of isolated %points in $\sgma(\mathcal{T})$ that are 
 eigenvalues of $\mathcal{T}$ of finite algebraic multiplicity. The rest of the spectrum is called  the essential spectrum and denoted by $\sgma_{\ess}(\mathcal{T})$.

The spectrum of the linearization touching the imaginary axis complicates  the stability analysis of system \eqref{eq3.1.1} in multidimensional  space. In the one-dimensional case \cite{GLS2}, the authors have imposed the hypotheses under which the  front  is spectrally stable in $H_{\alpha}^1(\mathbb{R})^n$, i.e., the  linear operator associated with the one-dimensional differential expression $L_1=D\partial_z^2+c\partial_z+df(\phi)$
has only one simple, isolated  eigenvalue at $0$ while the rest of the spectrum is located to the left  of the imaginary axis. %, in the closed right half-plane  with all others points of the spectrum being located strictly to the left of the imaginary axis on the weighed space $H_{\alpha}^1(\bbR)^n$. %One can refer to the discussion in \cite[Section 3.1]{GLS2}. 
%In our case the  assumptions about the boundary of the essential spectrum being moved to the left of the imaginary axis  that concerning  the boundary of the  essential spectrum 
 More precisely, let $L^-_1$ and $L^+_1$ be defined as in \eqref{eq3.2.5}.  By \cite[Lemma 3.5]{GLS2}, the rightmost boundary of the corresponding $\Spe(\mathcal{L}_{1,\alpha})$ is  the rightmost boundary of the set $\sgma(\mathcal{L}^-_{1,\alpha})\cup \sgma(\mathcal{L}^+_{1,\alpha})$,  where
 %The spectrum of $\mathcal{L}^-_{1,\alpha}$ on $H^k_{\alpha}(\mathbb{R})^n$ is %equal to the set of $\lambda\in\mathbb{C}$ for which there exists $\theta\in\mathbb{R}$ such that 
\begin{eqnarray*} \sgma(\mathcal{L}^-_{1,\alpha})&=& \{ \lambda\in\mathbb{C}\,\,\big| \,\, \exists\,\,\, \theta\in\mathbb{R}: \det\left( -\theta^2+i\theta (c-2\alpha_-)I -\lambda I+(\alpha_-^2-c\alpha_-)I+df(0)\right)=0\},\\ \sgma(\mathcal{L}^+_{1,\alpha}) &=& \{ \lambda\in\mathbb{C}\,\,\big|\,\, \exists\, \,\, \theta\in\mathbb{R}:
\det\left( -\theta^2+i\theta (c-2\alpha_+)I -\lambda I+(\alpha_+^2-c\alpha_+)I+df(\phi_+)\right)=0\} .\end{eqnarray*}
It is assumed that the right most boundary is located  strictly to the left of the imaginary axis.
Thus, in the one-dimensional case the essential spectrum of the linearization in the exponentially weighted  space is located in the open left plane. 

% By \cite[Lemma 3.5]{GLS2}, the right-hand boundary of $\Spe(\mathcal{L}_{1,\alpha})$ is exactly the right-hand boundary of the set $\sgma(\mathcal{L}^-_{1,\alpha})\cup \sgma(\mathcal{L}^+_{1,\alpha})$.

For the multidimensional case % operator $\cL_{\alpha}$
 the situation is  by far more complicated.  %as shown on Figure \ref{fig2}. 
Here, one concern regarding the spectrum of $\mathcal{L}_{\alpha}$
 is that  the zero eigenvalue,  which is an isolated eigenvalue for a one-dimensional operator considered on $H^k_{\alpha}(\mathbb{R})$, %but, as it is shown on Figure~\ref{fig1}, in multidimensional  case  it
  is not  anymore an isolated point of the spectrum  of  $\mathcal{L}_{\alpha}$  in  ${H^k_{\alpha}(\mathbb{R}^d)}$.

In Figure~\ref{fig1}, we %also assume that there are no eigenvalues 
illustrate the influence of the exponential weight on the location of the essential spectrum, and the  issue arising in the multidimensional  system when the same method of passing to the exponentially weighted spaces is applied.%  on the Figure~\ref{fig1}. % based on the consideration of  the  example of  system \eqref{MPr}-\eqref{deng}, considered   on one-dimensional physical space in a frame moving with speed $c$.
%in the discussion of stability at the steady state solution in \cite{xythesis}, we showed for the model system 
%\begin{equation}\label{example} \begin{cases}u_{1t}(t,x)=\Delta_{x} u_1(t,x)+u_2(t,x) g(u_1(t,x)),\, u_1,u_2\in\mathbb{R},\\ u_{2t}(t,x)=\epsilon\Delta_{x} u_2(t,x)-\kappa u_2(t,x) g(u_1(t,x)), \, x\in\mathbb{R},\end{cases}\end{equation}
%
%The spectrum of  the operator obtained by linearizing this system  
%%system \eqref{example} 
%%about  the left end state $\phi_-=(0,0)$ %:
%%$$L_1^-=\begin{pmatrix} 1 & 0\\ 0& \epsilon \end{pmatrix}\partial_{zz}+c\partial_z+\begin{pmatrix} 0 & e^{-\kappa}\\ 0 &-\kappa e^{-\kappa}\end{pmatrix},$$
%touches the imaginary axis when the operator is considered in the  space $H^k(\mathbb{R}^d)^2$. The spectrum of the operator generated by the same differential equation but considered on
% spectrum of operator $\mathcal{L}^-_{1,\alpha}$ 
%on $H^k_{\alpha}(\mathbb{R})$ is %the union of two curves $-\theta^2+i\theta(c-2\alpha)+(\alpha^2-c\alpha)$ and $-\epsilon\theta^2+i\theta(c-2\epsilon\alpha)+(\epsilon\alpha^2-c\alpha)-\kappa e^{-\kappa}$ which are
%is located in the left half plane. 
 %The essential spectrum of the linearization of the system \eqref{MPr}-\eqref{deng}, considered   in a frame co-moving with  the  front,  in the weighted space $H^k_{\alpha}(\mathbb{R})$  compared to the  $H^k(\mathbb{R})$ then is moved to the left of the imaginary axis, exposing an eigenvalue at the origin, as shown in  Figure \ref{fig1}.
 %, and the two curves are also right-hand boundary of $\Spe(\mathcal{L}_{1,\alpha})$, see Figure \ref{fig1}.

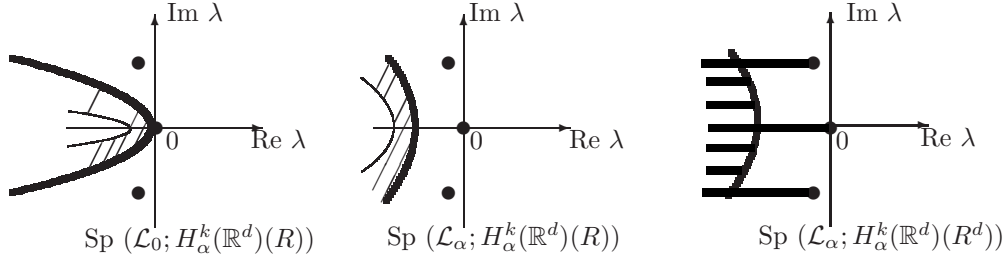
\begin{figure}[h]
	\setlength{\unitlength}{0.94 mm}
	\begin{picture}(160,50)(-225,5)\
	%%axis1
	\put(-206.408,23.0092){\vector(1,0){27.723}}
	\put(-193.962,8.88374){\vector(0,1){30.231}}
	%axis2
	\put(-162.975,23.0092){\vector(1,0){27.723}}
	\put(-150.302,8.88374){\vector(0,1){30.231}}
	%axis3
	\put(-112.254,23.3631){\vector(1,0){27.723}}
	\put(-98.3926,9.10567){\vector(0,1){30.231}}
	%dots1
	\put(-196.264,32.1603){\circle*{2}}
	\put(-193.791,23.0092){\circle*{2}}
	\put(-196.264,13.8581){\circle*{2}}
	%dots2
	\put(-152.435,32.1603){\circle*{2}}
	\put(-150.302,23.0092){\circle*{2}}
	\put(-152.435,13.8581){\circle*{2}}
	%dots3
	\put(-98.3926,23.0092){\circle*{2}}
	\put(-100.841,32.1603){\circle*{2}}
	\put(-100.841,13.8581){\circle*{2}}
	%text
	\put(-108.561,6.59741){Sp ($\mathcal{L}_{\alpha};{H^k_{\alpha}(\mathbb{R}^d)}(R^d))$ }
	\put(-135.384,20.2369){Re $\lambda$}
	\put(-149.51,38.1907){Im $\lambda$}
	\put(-85.6714,20.517){Re $\lambda$}
	\put(-96.1213,38.1242){Im $\lambda$}
	\put(-203.887,6.21084){Sp ($\mathcal{L}_{0};H^k_{\alpha}(\mathbb{R}^d)(R)$)}
	\put(-192.192,38.1242){Im $\lambda$}
	\put(-180.387,20.2019){Re $\lambda$}
	\put(-161.112,6.50018){Sp ($\mathcal{L}_{\alpha};{H^k_{\alpha}(\mathbb{R}^d)}(R)$)}
	\put(-192.43,20.0){0}
	\put(-149.43,20.0){0}
	\put(-97.6714,20){0}
	%curves
	
	\qbezier(-164.764,30.104) ( -155.561,23.0369) (-164.713,16.7833)
	\qbezier(-206.016,25.3619) (-188.676,22.8169 )( -206.2,20.4156)
	\linethickness{0.8mm}
	\qbezier(-160.863,32.9102)(-153.206,23.1412)(-161.259,12.7121)
	\qbezier(-214,32.8934) (-175,22.8934 )(-214,13.8934)
	\qbezier(-112.418,33.6442) ( -104.761,23.8752) ( -112.813,13.4462)
	%shadow2
	\put(-160.493,32.6969){\line(-1,-2){2.3}}
	\put(-159.669,31){\line(-1,-2){2}}
	\put(-158.9,29.2649){\line(-1,-2){1.8}}
	\put(-158.1,27.6871){\line(-1,-2){2}}
	\put(-157.6,25.2871){\line(-1,-2){5.5}}
	\put(-157.4,21.5871){\line(-1,-2){4.5}}
	%shadow1
	\put(-201.6,28.8){\line(-1,-2){2}}
	\put(-200.4,28.2){\line(-1,-2){1.8}}
	\put(-199,27.8){\line(-1,-2){1.8}}
	\put(-197.5,27){\line(-1,-2){1.77}}
	\put(-196,25.9){\line(-1,-2){4}}
	\put(-195,24){\line(-1,-2){2.3}}
	\put(-199.5,21.5){\line(-1,-2){2}}
	\put(-201.4,21){\line(-1,-2){2}}
	\color{black}
	\linethickness{1.mm}
	\put(-115.7,23.0){\line(1,0){17}}
	\put(-115.961,26.2393){\line(1,0){6.8}}
	\put(-115.961,29.5393){\line(1,0){6.3}}
	\put(-101.5,13.891){\line(-1,0){15}}
	\put(-101.6,32.1603){\line(-1,0){15}}
	\put(-116,17){\line(1,0){5}}
	\put(-116,20.2){\line(1,0){6.8}}
	\end{picture}
		\caption{The first panel: the rightmost boundary of the essential spectrum and the eigenvalue at the origin of  the linearization of \eqref{MPr}-\eqref{deng} about the front in the space with no exponential weight. The second panel: the rightmost boundary of the essential spectrum and the eigenvalue at the origin of  the linearization of \eqref{MPr}-\eqref{deng} about the front in the case of one-dimensional spacial variable in the exponentially weighted space. The essential spectrum in one-dimensional case is bounded away from the imaginary axis.  The third panel: the multidimensional case. The essential spectrum in the weighted space is not bounded away from the imaginary axis.
		% boundary of the essential spectrum and the eigenvalue at the origin of  the linearization of \eqref{MPr}-\eqref{deng} about the front in case of one-dimensional spacial variable.		
		%The rightmost boundary of the essential spectrum and the eigenvalue at the origin of  the linearization of \eqref{MPr}-\eqref{deng} about the front, in the exponentially weighted norm. 
		} 	\label{fig1} 
\end{figure}

Indeed, the following  proposition holds that
shows that the essential spectrum of $\mathcal{L}_{\alpha}$ is no longer bounded away from the imaginary axis on the weighted space ${H^k_{\alpha}(\mathbb{R}^d)}^n$ as it was for $d=1$.

% shows what kind of spectrum is created by isolated eigenvalues  of the operator $\mathcal{L}_{1,\alpha}$ when we consider the system on the multidimensional  weighted space ${H^k_{\alpha}(\mathbb{R}^d)}^n=H^k_{\alpha}(\mathbb{R})^n\otimes H^k(\mathbb{R}^{d-1})^n$. %The result is illustrated on Figure \ref{fig2}.

\begin{proposition}\label{spectrum}
	Let $d >1$,  the assumptions of Hypothesis \ref{hypo3.2.3} hold,   and the linear operators $\mathcal{L}_{\alpha}$  and $\mathcal{L}_{1,\alpha}$ be  the operators  defined according to Definition~\ref{op} associated with $L$ and $L_1$  introduced in \eqref{eq3.1.2} and \eqref{eq3.2.1}, respectively.   
	%Thus, the essential spectrum of $\mathcal{L}_{\alpha}$ is not bounded away from the imaginary axis on the weighted space ${H^k_{\alpha}(\mathbb{R}^d)}^n$ as it was for $d=1$. 
	Each point $\eta\in Sp(\mathcal{L}_{1,\alpha})$  generates a horizontal half-line $\{\lambda\in \mathbb C: \mathrm {Re}\, \lambda \leq \mathrm {Re}\, \eta, \,\, \mathrm {Im}\, \lambda =\mathrm {Im}\, \eta \}$ that  belongs to the essential spectrum $Sp_{ess}(\mathcal{L}_{\alpha})$. In particular,  the half-line  $\{\lambda\in \mathbb R: \mathrm {Re}\, \lambda \leq 0\}$ belongs  to the essential spectrum of $\mathcal{L}_{\alpha}$. %, together with  a semiline $\{\lambda: \mathcal Re \lambda \leq 0\}$.  %  generates a semiline  in the essential spectrum $\sgma_{\ess}(\cL_{\alpha})$.
\end{proposition}

%Thus, the essential spectrum of $\mathcal{L}_{\alpha}$ is no longer bounded away from the imaginary axis on the weighted space ${H^k_{\alpha}(\mathbb{R}^d)}^n$ as it was for $d=1$. 

\begin{proof}	
	The result   follows from \cite[Theorem XIII.34, Theorem XIII.35, and Corollary 1]{ReedSimon4}. Indeed, since $\mathcal{L}_{1,\alpha}$ and $I_n \Delta_y$ are the generators of bounded  analytic semigroups on Hilbert spaces $H^k_{\alpha}(\mathbb{R})$ and $H^k(\mathbb{R}^{d-1})$ respectively, we have
	\begin{equation*}%\label{tensorform}
	\sgma(\mathcal{L}_{1,\alpha}\otimes I_{H^k(\mathbb{R}^{d-1})}+I_{H^k_{\alpha}(\mathbb{R})}\otimes
	\Delta_y)=\sgma(\mathcal{L}_{1,\alpha})+\sgma(\Delta_y),
	\end{equation*}
	%(see \cite{ReedSimon4}.
which implies the conclusions of  the proposition. %In particular, we can conclude that every point on the boundary of spectum of $\mathcal{L}_{1,\alpha}$ will generate a semiline to $-\infty$, so that the essential spectrum of $\mathcal{L}_{\alpha}$ on $ {H^k_{\alpha}(\mathbb{R}^d)}^n$ will now touch the imaginary axis by the semiline generated by this eigenvalue $0$ of $\cL_{1,\alpha}$, in contrast to the $1$-dimensional case, where zero eigenvalue is an isolated point in $\sgma(\cL_{1,\alpha})$.
\end{proof}
%\textcolor{blue}{If this proof is so much shorter, shouldn't we keep it instead of the approximate spectrum? Anna}

%We will also impose the following standing assumption that will be used from this point on to the end of this section.
%\begin{hypothesis}\label{MH}
%	In addition to Hypotheses \ref{hypo3.2.1}, \ref{hypo3.2.2}, \ref{hypo3.2.3}, \ref{hypo3.2.5} and \ref{hypo3.2.7}, we assume that the diffusion matrix in \eqref{eq3.1.1} is identity, $D=I_{n}$.\end{hypothesis}

Since by Hypothesis \ref{hypo3.2.3},  $0$ is a simple,  isolated eigenvalue of $\mathcal{L}_{1,\alpha}$, we can define the Riesz spectral projection $P_{\alpha}$ of $\cL_{1,\alpha}$ on $H_{\alpha}^k(\mathbb{R})^n$ onto the 1-dimensional space $\ker(\mathcal{L}_{1,\alpha})$. The projection $P_{\alpha}$ commutes with $e^{t\mathcal{L}_{1,\alpha}}$ for all $t\geq0$. Since the operator $\mathcal{L}_{1,\alpha}$ is Fredholm of index zero, standard operator theory, see, e.g.,  %and $0$ is a simple eigenvalue of $\mathcal{L}_{1,\alpha}$,  
\cite[Lemma 2.13]{DL}, yields that 
$H_{\alpha}^k(\mathbb{R})^n=\ran \mathcal{L}_{1,\alpha} \oplus\ker \mathcal{L}_{1,\alpha}$ and $\ker P_{\alpha} =\ran \mathcal{L}_{1,\alpha}$.

Hypothesis \ref{hypo3.2.3} implies that $\ran P_{\alpha} =\ker \mathcal{L}_{1,\alpha} $ is spanned by $\phi'$. Reasoning as in \cite{Kapitula2} or as in the proof of 
Lemma~3.8 in \cite{GLS2}, that is, by invoking Palmer's Theorem \cite{Palm},  one can show that there exists a unique $H^k$-smooth function $\tilde{e}:\mathbb{R}\rightarrow\mathbb{R}^n$ such that the function $\gamma_{\alpha}^{-1}(\cdot)\tilde{e}(\cdot)$ is exponentially decaying, $\tilde{e}$ solves the adjoint equation $L_{1,\alpha}^*\tilde{e}=0$ and satisfies $\int_{\bbR}(\tilde{e}(s),\phi'(s))_{\bbR^n}ds=1$, where $(\ ,\ )_{\bbR^n}$ is the standard inner product in $\bbR^n$. 
Then for $V\in H^k_{\alpha}(\mathbb{R})^n$, the operator $P_{\alpha}$ can be written as follows,
\begin{equation*}
(P_{\alpha}V)(z)=\left(\int_{\mathbb{R}}\left( \tilde{e}(s),V(s) \right)_{\bbR^n}ds\right)\phi'(z),\, z\in\mathbb{R}.
\end{equation*}
Let $Q_{\alpha}=I-P_{\alpha}$ be the projection in $H^k_{\alpha}(\bbR)^n$ onto $\ran\mathcal{L}_{1,\alpha}$ with kernel $\ker(\mathcal{L}_{1,\alpha})$. The operator $Q_{\alpha}$ also commutes with $e^{t\mathcal{L}_{1,\alpha}}$ for all $t\geq 0$.
Next, for $U\in H^k_{\alpha}(\mathbb{R})^n\otimes H^k(\mathbb{R}^{d-1})^n$ we denote 
 \begin{equation}\label{DFNPP0}
 (\pi_{\alpha}U)(y)=\int_{\mathbb{R}}\left( \tilde{e}(s),U(s,y) \right)_{\bbR^n}ds, \end{equation}
and introduce  an operator on ${H^k_{\alpha}(\mathbb{R}^d)}^n=H^k_{\alpha}(\mathbb{R})^n\otimes H^k(\mathbb{R}^{d-1})$ defined by
\begin{equation*}%\label{DFNPP1}
\mathcal{P}U=\left(P_{\alpha}\otimes I_{H^k(\mathbb{R}^{d-1})} \right)U,
\end{equation*}
so that
\begin{equation*}%\label{DFNPP2}
(\mathcal{P}U)(z,y)=\left( \int_{\mathbb{R}}(\tilde{e}(s),U(s,y))_{\bbR^n}ds \right)\phi'(z)
=\left(\pi_{\alpha}U\right)(y)\phi'(z),\, (z,y)\in\mathbb{R}^d.
\end{equation*}

In what follows we frequently use the following lemma from \cite[page 299]{ReedSimon}:
\begin{lemma}\label{tensornorm}
	Let $A$ and $B$ be bounded operators on Hilbert spaces $H_1$ and $H_2$. Then $$\|A\otimes B\|_{\cB(H_1\otimes H_2)}=\|A\|_{\cB(H_1)}\|B\|_{\cB(H_2)}.$$
\end{lemma}

We now show that $\pi_{\alpha}$ and $\mathcal{P}$ have the following properties.	

\begin{lemma}\label{lem3.2.5}
	Let $k\geq[\frac{d+1}{2}]$ and  $\alpha=(\alpha_-,\alpha_+)\in\mathbb{R}^2_+$  be as in Hypothesis \ref{hypo3.2.3}. Then % linear operators $\mathcal{P}$ and $\pi_{\alpha}$ satisfy: 
	$$\mathcal{P}\in \cB\left({H^k_{\alpha}(\mathbb{R}^d)}\cap{H^k(\mathbb{R}^d)}\right)  \text{ and }    \pi_{\alpha}\in\mathcal{B}\left( H^k(\mathbb{R}^d)\cap{H^k_{\alpha}(\mathbb{R}^d)},H^k(\mathbb{R}^{d-1})\right).$$ 
	Moreover, $$\pi_{\alpha}\in\mathcal{B}\left(L^1_{\alpha}(\mathbb{R})\otimes L^1(\mathbb{R}^{d-1}),L^1(\mathbb{R}^{d-1})\right).$$
\end{lemma}
\begin{proof}
	Since $\|\gamma_{\alpha}^{-1}(z)\tilde{e}(z)\|_{\bbR^n}\rightarrow 0$ exponentially fast as $|z|\rightarrow\infty$, there exist $\zeta_-<0<\zeta_+$ and $K>0$ such that $\|\gamma_{\alpha}^{-1}(z)\tilde{e}(z)\|_{\bbR^n}\leq Ke^{-\zeta_- z}$ for $z\leq 0$, and $\|\gamma_{\alpha}^{-1}(z)\tilde{e}(z)\|_{\bbR^n}\leq Ke^{-\zeta_+ z}$ for $z\geq 0$.  We pick $U\in{H^k(\mathbb{R}^d)}^n\cap{H^k_{\alpha}(\mathbb{R}^d)}^n$, and first consider the $L^2$-norm, so that
	\begin{align*}
	\| \pi_{\alpha}U \|^2_{L^2(\mathbb{R}^{d-1})}&= % \int_{\mathbb{R}^{d-1}}|(\pi_{\alpha}U)(y)|^2dy
	%=
	\int_{\bbR^{d-1}}\left| \int_{\bbR}\left( \gamma_{\alpha}^{-1}(s) \tilde{e}(s), \gamma_{\alpha}(s)U(s,y) \right)_{\bbR^n} \, ds\right|^2\,dy\\
	%&\leq\int_{\mathbb{R}^{d-1}}\left(\int_{\mathbb{R}}\|\gamma_{\alpha}^{-1}(s) \tilde{e}(s)\|_{\bbR^n}  \|\gamma_{\alpha}(s)U(s,y) \|_{\bbR^n}\,ds\right)^2dy
	&\leq\int_{\mathbb{R}^{d-1}} \left( \int_{\mathbb{R}}\|\gamma_{\alpha}^{-1}(s)\tilde{e}(s)\|_{\bbR^n}^2 ds \right)\left( \int_{\mathbb{R}}\|\gamma_{\alpha}(s)U(s,y)\|_{\bbR^n}^2 ds \right) dy
	\end{align*}
	by H\"{o}lder's inequality. Since 
		\begin{equation*}%\label{eq3.3.6}
	\|\gamma_{\alpha}^{-1}(s)\tilde{e}(s)\|_{\bbR^n}\leq\begin{cases}
	Ke^{-\zeta_-s},\, \text{for }s\leq 0,\\
	Ke^{-\zeta_+s},\, \text{for }s\geq 0,
	\end{cases}
	\end{equation*}
	then 	
	\begin{equation}\label{gamma}
	\int_{\mathbb{R}}\|\gamma_{\alpha}^{-1}(s)\tilde{e}(s)\|_{\bbR^n}^2 ds \leq K\left(\int_{-\infty}^0e^{-2\zeta_- s}\, ds+\int_0^{\infty} e^{-2\zeta_+s}\,ds\right)\leq C
	\end{equation}
	for some constant $C>0$. Thus, 
	\begin{equation}\label{eq3.2.27}
	\| \pi_{\alpha}U \|^2_{L^2(\mathbb{R}^{d-1})}\leq %C\int_{\mathbb{R}^{d-1}}\int_{\mathbb{R}}\| \gamma_{\alpha}(s)U(s,y) \|^2 \,ds\,dy=
	C\| \gamma_{\alpha}U \|^2_{L^2(\mathbb{R}^d)} 
	\leq C\max\{ \|U\|^2_{L^2(\bbR^d)}, \|U\|^2_{L^2_{\alpha}(\bbR^d)} \}\leq C\|U\|^2_{{\mathcal{H}}}.
	\end{equation}
	
	For $H^k$-norms, we use the equivalent Sobolev norm (see, e.g., \cite[p.316]{NS}) given as follows: Let $x=(z,y)\in\bbR^d$ and $y=(x_2,...,x_d)\in\bbR^{d-1}$, then
	\begin{equation*}
	\|f\|_{H^k(\mathbb{R}^{d-1})}\sim\|f\|_{L^2(\mathbb{R}^{d-1})}+\sum_{a_2+\cdots+a_d=k}\left\|\frac{\partial^k}{\partial{x_2^{a_2}}\cdots \partial{x_d^{a_d}}}f\right\|_{L^2(\bbR^{d-1})},
	\end{equation*}
	where the sum extends over all $(d-1)$-tuples $(a_2,...,a_d)$ of non-negative integers with $\sum_{i=2}^d a_i=k$, and $\displaystyle{\frac{\partial^{a_i}}{\partial{x_i^{a_i}}}}$ is the $a_i$-th differentiation of functions with respect to $x_i$, $i=2$, ..., $d$. 
	
	We already have the estimates for $\|\pi_{\alpha}U\|_{L^2(\mathbb{R}^{d-1})}$ for $U\in{H^k_{\alpha}(\mathbb{R}^d)}^n\cap {H^k(\mathbb{R}^d)}^n$. From H\"{o}lder's inequality and \eqref{gamma} it follows  that
	\begin{align}\label{eq3.4.30}
	&\left\| \frac{\partial^k}{\partial{x_2^{a_2}}\cdots \partial{x_d^{a_d}}}\pi_{\alpha}U \right\|^2_{L^2(\mathbb{R}^{d-1})}%= \int_{\mathbb{R}^{d-1}}\left|\frac{\partial^k}{\partial{x_2^{a_2}}\cdots \partial{x_d^{a_d}}}(\pi_{\alpha}U)(y)\right|^2dy\no\\
	\leq\int_{\mathbb{R}^{d-1}}\left(\int_{\mathbb{R}}\|\gamma^{-1}_{\alpha}(s) \tilde{e}(s)\|_{\bbR^n} \left\|\gamma_{\alpha}(s)\frac{\partial^k}{\partial{x_2^{a_2}}\cdots \partial{x_d^{a_d}}}U(s,y) \right\|_{\bbR^n}ds\right)^2dy\no\\
	&\leq\int_{\mathbb{R}^{d-1}} \left( \int_{\mathbb{R}}\|\gamma_{\alpha}^{-1}(s)\tilde{e}(s)\|_{\bbR^n}^2 ds \right)\left( \int_{\mathbb{R}}\left\|\gamma_{\alpha}(s)\frac{\partial^k U(s,y)}{\partial{x_2^{a_2}}\cdots \partial{x_d^{a_d}}}\right\|_{\bbR^n}^2 ds \right) dy%\no\\
	%&
	%\qquad\leq C\int_{\mathbb{R}^{d-1}}\int_{\mathbb{R}}\left\| \gamma_{\alpha}(s)\frac{\partial^k}{\partial{x_2^{a_2}}\cdots \partial{x_d^{a_d}}}U(s,y) \right\|_{\bbR^n}^2 dsdy
	\leq C\| U \|^2_{\alpha}\leq C\| U \|^2_{{\mathcal{H}}} ,
	\end{align}
	thus implying  $\pi_{\alpha}\in\mathcal{B}\left({H^k_{\alpha}(\mathbb{R}^d)}^n\cap{H^k(\mathbb{R}^d)}^n,H^k(\mathbb{R}^{d-1})\right)$.
	
	For  the $L^1$-norm of $\pi_{\alpha}U$,  analogously,
	\begin{align}\label{eq3.2.32}
	\| \pi_{\alpha}U \|_{L^1(\mathbb{R}^{d-1})}
	%= \int_{\mathbb{R}^{d-1}}|(\pi_{\alpha}U)(y)|dy
	%=\int_{\mathbb{R}^{d-1}}\left|\int_{\mathbb{R}}\left( \tilde{e}(s),U(s,y) \right)_{\bbR^n}ds\right|dy\nonumber\\\nonumber
	%\leq\int_{\mathbb{R}^{d-1}}\left(\int_{\mathbb{R}} \|\gamma_{\alpha}^{-1}(s) \tilde{e}(s)\|_{\bbR^n}\|\gamma_{\alpha}(s)U(s,y) \|_{\bbR^n}ds\right)dy%\\\nonumber 
	%&\leq\int_{\mathbb{R}^{d-1}}  \left( \int_{\mathbb{R}}C\|\gamma_{\alpha}(s)U(s,y)\|_{\bbR^n} ds \right) dy\\\nonumber
	%&
	\leq C\|\gamma_{\alpha}U\|_{L^1(\mathbb{R}^{d})}
	\leq C\|U\|_{L^1_{\alpha}(\bbR)\otimes L^1(\bbR^{d-1})}.
	\end{align}
	
	We now  consider $\mathcal{P}U$ for $U\in{H^k_{\alpha}(\mathbb{R}^d)}^n$ noting that $\mathcal{P}=P_{\alpha}\otimes I_{H^k(\mathbb{R}^{d-1})}$. As shown in \cite[Section 3.3]{GLS2}, the projection $P_{\alpha}$ is a bounded operator from $H^k_{\alpha}(\bbR)\cap H^k(\bbR)$ to $H^k_{\alpha}(\bbR)\cap H^k(\bbR)$. Therefore, by Lemma \ref{tensornorm} we have: 
$
	\|\mathcal{P}\|_{\cB({\mathcal{H}})}= \|P_{\alpha}\|_{\cB(H^k_{\alpha}(\bbR)\cap H^k(\bbR))}\|I\|_{\cB(H^k(\bbR^{d-1}))}\leq C,
	$
	which completes the proof.
\end{proof}
\begin{lemma}\label{cpcq}  The operator $\cP$ is a bounded operator  (i) from ${H^k_{\alpha}(\mathbb{R}^d)}$ to ${H^k_{\alpha}(\mathbb{R}^d)}$; (ii)  from  ${\mathcal{H}}$ to ${H^k_{\alpha}(\mathbb{R}^d)}$.  %while the second inequality implies that $\cP$ is a bounded operator from
(iii)  from   $H^k_{\a}(\bbR^d)$ to $H^k(\bbR^d)$; (iv)  from  ${\mathcal{H}}$ to $H^k(\bbR^d)$.  The complementary projection  $\cQ=\cI-\cP$ is a bounded operator (i) from $H^k_{\a}(\bbR^d)$ to ${\mathcal{H}}$ and (ii) from ${\mathcal{H}}$ to ${\mathcal{H}}$. 
\end{lemma}
%This statement follows from the following inequalities 
%The two inequalities also yield that the complementary operator $\cQ=\cI-\cP$ is a bounded operator from ${H^k_{\alpha}(\mathbb{R}^d)}$ to ${\mathcal{H}}$ and from ${\mathcal{H}}$ to ${\mathcal{H}}$.
	\begin{proof}Indeed, Lemma \ref{lem3.2.5} and the definition $\cP U(z,y)=(\pi_{\alpha} U)(y) \phi'(z)$, $z,y\in\bbR^d$, imply that 
	\begin{align*}
	\|\cP U\|_{H^k_{\alpha}(\bbR^d)}=\|\pi_{\alpha} U\|_{H^k(\bbR^{d-1})}\|\phi'\|_{H^k_{\alpha}(\bbR)}\leq C\|U\|_{H^k_{\alpha}(\bbR^d)}\|\phi'\|_{H^k_{\alpha}(\bbR)}\leq C\|U\|_{{\mathcal{H}}}\|\phi'\|_{H^k_{\alpha}(\bbR)},\\
	\|\cP U\|_{H^k(\bbR^d)}=\|\pi_{\alpha} U\|_{H^k(\bbR^{d-1})}\|\phi'\|_{H^k(\bbR)}\leq C\|U\|_{H^k_{\alpha}(\bbR^d)}\|\phi'\|_{H^k(\bbR)}\leq C\|U\|_{{\mathcal{H}}}\|\phi'\|_{H^k(\bbR)},
	\end{align*}
	and  the statement above follows. \end{proof}
	%The first inequality implies that $\cP$ is a bounded operator from ${H^k_{\alpha}(\mathbb{R}^d)}$ to ${H^k_{\alpha}(\mathbb{R}^d)}$ and also from $\cE$ to ${H^k_{\alpha}(\mathbb{R}^d)}$ while the second inequality implies that $\cP$ is a bounded operator from ${H^k_{\alpha}(\mathbb{R}^d)}$ to $\cE_{0}$ and also from $\cE$ to ${H^k(\mathbb{R}^d)}$. % The two inequalities also yield that the complementary operator $\cQ=\cI-\cP$ is a bounded operator from ${H^k_{\alpha}(\mathbb{R}^d)}$ to $\cE$ and from $\cE$ to $\cE$.\hfill$\Diamond$
%\end{remark}

The projection $P_{\alpha}$ is initially defined as the Riesz projection for the operator $\mathcal{L}_{1,\alpha}$. To verify  that $\mathcal{P}\mathcal{L}_{\alpha}=\mathcal{L}_{\alpha}\mathcal{P}$, we recall that    $P_{\alpha}\mathcal{L}_{1,\alpha}=\mathcal{L}_{1,\alpha}P_{\alpha}$   which implies that $\mathcal{L}_{\alpha}$ and $\mathcal{P}$ commute since   $\mathcal{P}\mathcal{L}_{\alpha}
=%(P_{\alpha}\otimes I_{H^k(\mathbb{R}^{d-1})^n})(\mathcal{L}_{1,\alpha}\otimes I_{H^k(\mathbb{R}^{d-1})^n}+I_{H^k_{\alpha}(\mathbb{R})^n}\otimes\Delta_y)=
P_{\alpha}\mathcal{L}_{1,\alpha}\otimes I_{H^k(\mathbb{R}^{d-1})^n} +P_{\alpha} \otimes\Delta_y$  and   $\mathcal{L}_{\alpha}\mathcal{P}
=%(P_{\alpha}\otimes I_{H^k(\mathbb{R}^{d-1})^n})(\mathcal{L}_{1,\alpha}\otimes I_{H^k(\mathbb{R}^{d-1})^n}+I_{H^k_{\alpha}(\mathbb{R})^n}\otimes\Delta_y)=
\mathcal{L}_{1,\alpha}P_{\alpha}\otimes I_{H^k(\mathbb{R}^{d-1})^n} +P_{\alpha} \otimes\Delta_y$.

%Using $P_{\alpha}\mathcal{L}_{1,\alpha}=\mathcal{L}_{1,\alpha}P_{\alpha}$ yields
%\begin{align*}
%&\mathcal{P}\mathcal{L}_{\alpha}
%=(P_{\alpha}\otimes I_{H^k(\mathbb{R}^{d-1})^n})(\mathcal{L}_{1,\alpha}\otimes I_{H^k(\mathbb{R}^{d-1})^n}+I_{H^k_{\alpha}(\mathbb{R})^n}\otimes\Delta_y)=P_{\alpha}\mathcal{L}_{1,\alpha}\otimes I_{H^k(\mathbb{R}^{d-1})^n} +P_{\alpha} \otimes\Delta_y\\ 
%%&=\mathcal{L}_{1,\alpha}P_{\alpha}\otimes I_{H^k(\mathbb{R}^{d-1})}+P_{\alpha} \otimes\Delta_y \\
%&=\mathcal{L}_{1,\alpha}P_{\alpha}\otimes I_{H^k(\mathbb{R}^{d-1})^n} +P_{\alpha} \otimes\Delta_y =(\mathcal{L}_{1,\alpha}\otimes I_{H^k(\mathbb{R}^{d-1})^n}+I_{H^k_{\alpha}(\mathbb{R})^n}\otimes\Delta_y)(P_{\alpha}\otimes I_{H^k(\mathbb{R}^{d-1}){^n}}) 
%=\mathcal{L}_{\alpha}\mathcal{P}\end{align*}

\begin{remark}{ When the diffusion matrix $D$  in \eqref{GEQ_D} is not a multiple of an identity matrix, the relation $\cP\cL_{\alpha}=\cL_{\alpha}\cP$ does not hold in general.  Indeed, 
$
\mathcal{L}_{\alpha}=\mathcal{L}_{1,\alpha}\otimes I_{H^k(\mathbb{R}^{d-1})^n} +DI_{H^k_{\alpha}(\mathbb{R})^n}\otimes\Delta_y,
$
 and, in general,  $D$ doesn't commute with $P_{\alpha}$.}   This is the main obstacle that prevents us from dealing with non-scalar  diffusion matrices.

 \end{remark}
 %This explains the reason for assuming Hypothesis \ref{MH} instead of Hypothesis \ref{hypo3.2.7}.

%\end{document}

\section{The semigroup  estimates.}\label{sec4}

In this section we  provide estimates for the semigroups generated by the linear operators $\mathcal{L}_{\alpha}$, $\mathcal{L}_{\mathcal{H}}$, $\Delta_y$, and $\mathcal{L}^{(i)}$ for $i=1,2$, cf. \eqref{eq3.1.2}, \eqref{eq3.2.6} and \eqref{eq3.2.7}, see Lemmas~\ref{lalpha}, \ref{lem3.4.9}, \ref{lem3.7.19} and \ref{lem3.4.4} below.
%As is shown in  \cite[Lemma 3.13]{GLS2}, 
 Hypothesis \ref{hypo3.2.3} implies the following standard fact about analytic semigroups.
 \begin{lemma}\label{lalpha}
 	If $\nu>0$ is such that $\sup\{ \Re \lambda:\lambda\in \Spe(\mathcal{L}_{1,\alpha}) \}<-\nu$,  then there exists $K>0$ such that  $\|e^{t\mathcal{L}_{1,\alpha}}Q_{\alpha}\|_{\mathcal{B}(H^k_{\alpha}(\mathbb{R}))}\leq Ke^{-\nu t}$, for $t\geq 0$.
 \end{lemma}	

%Next, to obtain the estimate for the semigroup $\cL_{\a}$, we will frequently use the following lemma from \cite[page 299]{ReedSimon}.
%\begin{lemma}\label{tensornorm}
%	Let $A$ and $B$ be bounded operators on Hilbert spaces $H_1$ and $H_2$. Then $\|A\otimes B\|_{\cB(H_1\otimes H_2)}=\|A\|_{\cB(H_1)}\|B\|_{\cB(H_2)}$.
%\end{lemma}
Moreover,   the following lemma is true.
\begin{lemma}\label{lem3.4.9}
	Assume  Hypothesis \ref{hypo3.2.3}.  
	If  $\sup\{ \Re \lambda:\lambda\in \sgma(\mathcal{L}_{1,\alpha})\, \text{ and }\,\lambda\neq 0 \}<-\nu,$ for some $\nu>0$,   then there exists $K>0$ such that 
	$\|e^{t\mathcal{L}_{\alpha}}\mathcal{Q}\|_{\mathcal{B}({H^k_{\alpha}(\mathbb{R}^d)})}\leq Ke^{-\nu t}$, for all $t\geq 0$.
\end{lemma}
\begin{proof}
	Since $\mathcal{Q}=Q_{\alpha}\otimes I_{H^k(\mathbb{R}^{d-1})}$ and $\mathcal{L}_{\alpha}=\mathcal{L}_{1,\alpha}\otimes I_{H^k(\mathbb{R}^{d-1})}+I_{H_{\alpha}^k(\mathbb{R})}\otimes \Delta_y$, by the proof of \cite[Theorem XIII.35]{ReedSimon4} we have $e^{t\mathcal{L}_{\alpha}}\mathcal{Q}=e^{t\mathcal{L}_{1,\alpha}}Q_{\alpha}\otimes e^{t\Delta_y}I_{H^k(\mathbb{R}^{d-1})}$. The operators $\mathcal{L}_{1,\alpha}$ and $\Delta_y$ both generate bounded semigroups on $\ran Q_{\alpha}=\ran\cL_{1,\alpha}$ and $H^k(\mathbb{R}^{d-1})$, cf. Lemma \ref{lalpha} and Lemma \ref{lem3.4.4}.a), thus by Lemma \ref{tensornorm} we infer 
	$$\|e^{t\mathcal{L}_{1,\alpha}}Q_{\alpha}\otimes e^{t\Delta_y}I_{H^k(\bbR^{d-1})}\|_{\cB({H^k_{\alpha}(\mathbb{R}^d)})}=\|e^{t\cL_{1,\alpha}}Q_{\alpha}\|_{\cB(H^k_{\alpha}(\bbR))}\|e^{t\Delta_y}\|_{\cB(H^k(\bbR^{d-1}))},$$
	which completes  the proof.
\end{proof}
We consider the operator $\mathcal{L}^-$ on $ H^k(\mathbb{R}^d)$ associated with the differential expression
\begin{equation}
\label{Lneg}
L^-=L_1^-\otimes I_{H^k(\mathbb{R} ^{d-1})}+I_{H^k(\mathbb{R})}\otimes \Delta_y,
\end{equation}
where $L^-_1$ is defined in \eqref{eq3.2.9}, and let
\begin{align}\label{l12}
L^{(1)}&=\Delta_x+c\partial_z+A_1=L_1^{(1)}\otimes I_{H^k(\mathbb{R} ^{d-1})}+I_{H^k(\mathbb{R})}\otimes \Delta_y,\\\no
L^{(2)}&=\Delta_x+c\partial_z+d_{u_2}f_2(0)=L_1^{(2)}\otimes I_{H^k(\mathbb{R} ^{d-1})}+I_{H^k(\mathbb{R})}\otimes \Delta_y.
\end{align}
where $A_1$ is introduced in Hypothesis \ref{hypo3.2.5}, and $L_1^{(i)}$, $i=1$, $2$ are as in \eqref{eq3.2.6} and \eqref{eq3.2.7}.
Thus \begin{equation}\label{lnegl12}
L^-=\begin{pmatrix}
L^{(1)} & d_{u_2}f_1(0)\\
0 & L^{(2)}
\end{pmatrix},
\end{equation}
and the linearization \eqref{eq3.1.2} about the front is given by the formula 
\begin{equation}\lb{Ldecomp}
L=L^-+(df(\phi)-df(0))\otimes I_{H^k(\mathbb{R} ^{d-1})}.
\end{equation}

As in \cite[Lemma 8.2(1)]{GLS2}, the operator $df(\phi)-df(0)$ is a bounded operator from $H^k_{\alpha}(\bbR)$ into $H^k(\bbR)$. We therefore have 
\begin{equation}\lb{bddfi}
(df(\phi)-df(0))\otimes I_{H^k(\bbR^{d-1})}\in\cB(H^k_{\a}(\bbR^d),H^k(\bbR^d))
\end{equation}
\begin{lemma}\label{lem3.7.19}
	%Assume Hypothesis \ref{MH}. 
	Assume 	Hypotheses~\ref{hypo3.2.7}. Let $\mathcal{L}^{(i)}$, $i=1,2$ be the operators given by the differential expressions  \eqref{l12} on $ H^k(\mathbb{R}^d)$. If
	$\,\,\sup\{ \Re\lambda:\lambda\in \sgma(\mathcal{L}_{1}^{(2)}) \text{ and } \lambda\neq 0 \}<-\rho$, for some $\rho>0$, then
	 there exists  $K>0$ such that
	\begin{equation}\label{eq3.7.53}
	\|e^{t\mathcal{L}^{(1)}}\|_{\mathcal{B}( H^k(\mathbb{R}^d))}\leq K,\quad
	\|e^{t\mathcal{L}^{(2)}}\|_{\mathcal{B}( H^k(\mathbb{R}^d))}\leq Ke^{-\rho t},
	\end{equation}
	for all $t\geq 0$. Moreover,  the operator $\mathcal{L}^-$ given by the differential expression \eqref{Lneg} generates a bounded semigroup on $ H^k(\mathbb{R}^d)$, that is,
	\begin{equation}\label{lnegest}
	\|e^{t\mathcal{L}^-}\|_{\mathcal{B}( H^k(\mathbb{R}^d))}\leq K \text{ for all $t\ge0$}.
	\end{equation}
\end{lemma}
\begin{proof}
	We shall use the fact  \cite[Theorem XIII.35]{ReedSimon4} that 
	\begin{align*}
	e^{t\mathcal{L}^{(i)}}=e^{t(\mathcal{L}^{(i)}_1\otimes  I_{H^k(\mathbb{R}^{d-1})}+I_{H^k(\mathbb{R})}\otimes \Delta_y) }=e^{t\mathcal{L}^{(i)}_1}\otimes e^{t\Delta_y}, \text{ for $i=1$, $2$.}
	\end{align*}
	By Hypothesis \ref{hypo3.2.7}(1), the operator $\mathcal{L}_1^{(1)}$ generates a bounded semigroup on $H^k(\mathbb{R})$, thus,  by  Lemma \ref{tensornorm},
	$\|e^{t\mathcal{L}^{(1)}_1}\otimes e^{t\Delta_y}\|= \|e^{t\mathcal{L}^{(1)}_1}\|\|e^{t\Delta_y}\|<K$ for some $K>0$ and all $t\geq 0$.
	Similarly, from Hypothesis \ref{hypo3.2.7}(2) and Lemma \ref{tensornorm},
	$\|e^{t\mathcal{L}^{(2)}_1}\otimes e^{t\Delta_y}\|= \|e^{t\mathcal{L}^{(2)}_1}\|\|e^{t\Delta_y}\|<Ke^{-\rho t}$ for some $K>0$ and all $t\geq 0$.
	
	To prove  \eqref{lnegest}, %let $e^{t\mathcal{L}_-$ be the semigroup generated by the operator $\mathcal{L}^-$  (see  \eqref{lnegl12}) and let $\{S_i(t) \}_{t\geq 0}$, $i=1$, $2$ be the semigroups generated by the operators $\mathcal{L}^{(i)}$, $i=1$, $2$. 
	we notice that the triangular structure of the operator $\mathcal{L}^-$  yields the triangular structure of the semigroup $e^{t\mathcal{L}_-}$, that is
	\begin{equation}\label{lneg}
	e^{t\mathcal{L}_-}=\begin{pmatrix}
	e^{t\mathcal{L}^{(1)}}  & \int_0^t 
	e^{(t-s)\mathcal{L}^{(1)}}\partial_{u_2}f_1(0)e^{s\mathcal{L}^{(2)}}ds\\0 & e^{t\mathcal{L}^{(2)}}
	\end{pmatrix}.%\ \text{where}\ Q(t)=\int_0^t S_1(t-s)\partial_{u_2}f_1(0)S_2(s)ds.
	\end{equation}
		%\begin{equation}\label{lneg} S(t)=\begin{pmatrix} S_1(t) & Q(t)\\0 & S_2(t) \end{pmatrix},\ \text{where}\ Q(t)=\int_0^t S_1(t-s)\partial_{u_2}f_1(0)S_2(s)ds.	\end{equation}
Equation \eqref{lneg} and inequalities \eqref{eq3.7.53} imply \eqref{lnegest}.
%that \begin{equation*}	\|S(t)\|_{\mathcal{B}( H^k(\mathbb{R}^d))}=\|e^{t\mathcal{L}^-}\|_{\mathcal{B}( H^k(\mathbb{R}^d))}\leq K \end{equation*} for all $t\geq 0$ as required in \eqref{lnegest}.
\end{proof}
We next use Lemma \ref{lem3.4.9} and Lemma \ref{lem3.7.19}  to show  that the semigroup generated by the operator $\mathcal{L}$ on $\mathcal{H}$ is also bounded.
\begin{lemma}\label{lbeta}
	%Assume Hypothesis \ref{MH}. 
Assume 	Hypotheses~\ref{hypo3.2.7}. Let $\mathcal{L}_{\mathcal{H}}$ be the operator given by the differential expressions  \eqref{eq3.1.2} on  ${\mathcal{H}}=H^k(\bbR^d)\cap H^k_{\a}(\bbR^d)$. There exists $K>0$ such that  
	$\,\,\|e^{t\mathcal{L}_{\mathcal{H}}}\|_{\mathcal{B}(\mathcal{H})}\leq K \ 
	\text{ for all $t\geq 0$}.$
\end{lemma}
\begin{proof}
	%Let the operator $\mathcal{L}_{\mathcal{H}}$ be given by the same differential expression \eqref{eq3.1.2} on $\cE={H^k(\mathbb{R}^d)}\cap{H^k_{\alpha}(\mathbb{R}^d)}$, and l
	Let the operator $\cQ_{{\mathcal{H}}}$ be given by restricting $\cQ$ to ${\mathcal{H}}$, then by Lemma \ref{lem3.4.9} \begin{equation}\label{betaalpha}
	\|e^{t\mathcal{L}_{\mathcal{H}}}\mathcal{Q}_{\mathcal{H}}\|_{\mathcal{B}({H^k_{\alpha}(\mathbb{R}^d)})}\leq Ke^{-\nu t},
	\end{equation} 
	therefore, it remains to estimate $\|e^{t\mathcal{L}_{\mathcal{H}}}\mathcal{Q}_{{\mathcal{H}}}\|_{\cB({\mathcal{H}},H^k(\mathbb{R}^d))}$ and $\|e^{t\cL_{{\mathcal{H}}}}\cP_{{\mathcal{H}}}\|_{\cB({\mathcal{H}})}$.
	
	Since  $\ran\cQ_{{\mathcal{H}}}=\ran\mathcal{L}_{\alpha}\cap \mathcal{H}^n$ and $\mathcal{Q}_{\mathcal{H}}$ commutes with $\mathcal{L}_{\mathcal{H}}$ and $e^{t\mathcal{L}_{\mathcal{H}}}$,  the variation of constant formula and \eqref{Ldecomp} yield 
	$$e^{t\cL_{{\mathcal{H}}}}=e^{t\cL^-}+\int_0^t e^{(t-s)\cL^-}\big((df(\phi)-df(0))\otimes I_{H^k(\bbR^{d-1})}\big)e^{s\cL_{{\mathcal{H}}}}\, ds,$$ 
	%Multiplying by $\cQ_{{\mathcal{H}}}$, estimating the norm in $\cB({\mathcal{H}},{H^k(\mathbb{R}^d)})$ and using 
	from where, by \eqref{bddfi} and Lemma~\ref{cpcq}, as well as \eqref{lnegest} and \eqref{betaalpha},
		\begin{align*}
	&\|e^{t\mathcal{L}_{\mathcal{H}}}\mathcal{Q}_{\mathcal{H}}\|_{\cB({\mathcal{H}},H^k(\bbR^d))}
	\leq\|e^{t\mathcal{L}^{-}}\|_{\mathcal{B}( H^k(\mathbb{R}^d))}\|\cQ_{{\mathcal{H}}}\|_{\cB({\mathcal{H}},H^k(\bbR^d))}\\
	&+\int _0^t \|e^{(t-s)\cL^-}\|_{\cB(H^k(\bbR^d))}\|df(\phi)-df(0)\|_{\cB({H^k_{\alpha}(\mathbb{R}^d)},H^k(\bbR^d))}\|e^{s\cL_{{\mathcal{H}}}}\cQ_{{\mathcal{H}}}\|_{\cB({H^k_{\alpha}(\mathbb{R}^d)})}\|\cQ_{{\mathcal{H}}}\|_{\cB({\mathcal{H}},H^k_{\alpha}(\bbR^d))}ds \\&\qquad<K.
	\end{align*}
	%By using \eqref{lnegest} and \eqref{betaalpha}, we conclude that
	%\begin{equation*}
	%\|e^{t\mathcal{L}_{\mathcal{H}}}\cQ_{{\mathcal{H}}}\|_{\mathcal{B}({\mathcal{H}},\mathcal{H}_{0})}\leq K+\int _0^t K e^{-\nu s}ds\leq K,
	%\end{equation*}
	%for some $K>0$ and all $t\geq 0$. 
	Combined with \eqref{betaalpha} this shows that the semigroup $\{ e^{t\cL_{{\mathcal{H}}}}\cQ_{{\mathcal{H}}} \}_{t\geq 0}$ is bounded in $\ran \cQ_{{\mathcal{H}}}$. 
	
	We note that $\mathcal{H}=\ran\mathcal{P}_{\mathcal{H}}\oplus \ran\mathcal{Q}_{\mathcal{H}}$ and $e^{t\mathcal{L}_{\mathcal{H}}}=e^{t\mathcal{L}_{\mathcal{H}}}\mathcal{P}_{\mathcal{H}}\oplus e^{t\mathcal{L}_{\mathcal{H}}}\mathcal{Q}_{\mathcal{H}}$. In order to finish the proof of Lemma \ref{lbeta}, we will need to show that the seimigroup $\{e^{t\cL_{{\mathcal{H}}}} \cP_{{\mathcal{H}}}\}$ is bounded in $\ran \cP_{{\mathcal{H}}}$. %, that is,
	%\begin{equation*}\|e^{t\mathcal{L}_{\alpha}}\mathcal{P}_{\mathcal{H}}\|_{\mathcal{B}({H^k_{\alpha}(\mathbb{R}^d)})}\leq K,\, \,\,\|e^{t\mathcal{L}}\mathcal{P}_{\mathcal{H}}\|_{\mathcal{B}({\mathcal{H}},\mathcal{H}_{0})}\leq K \text{ for all $t\geq 0$}.\end{equation*}
	Since $\mathcal{P}_{\mathcal{H}}$ projects onto the kernels of $\mathcal{L}$  defined on $ H^k(\mathbb{R}^d)$ and $\mathcal{L}_{\alpha}$ defined  on ${H^k_{\alpha}(\mathbb{R}^d)}$, %of the generators of the semigroups $e^{t\cL}$ and $e^{t\cL_{\alpha}}$; 
	 then,  by Lemma~\ref{cpcq},
	$e^{t\cL_{\alpha}}\cP_{{\mathcal{H}}}=\cP_{{\mathcal{H}}}$ and $e^{t\cL}\cP_{{\mathcal{H}}}=\cP_{{\mathcal{H}}},$
	where $\cP_{{\mathcal{H}}}\in\cB({\mathcal{H}},H^k(\bbR^d))$ and $\cP_{{\mathcal{H}}}\in\cB({H^k_{\alpha}(\mathbb{R}^d)})$, and, therefore, for all $t\geq 0$,
	%This implies that
	 \begin{equation*}
	\|e^{t\mathcal{L}}\mathcal{P}_{\mathcal{H}}\|_{\mathcal{B}({\mathcal{H}}, H^k(\mathbb{R}^d))}=\|\mathcal{P}_{\mathcal{H}}\|_{\mathcal{B}(\cE, H^k(\mathbb{R}^d))}\leq K \, \text{ and } \, \|e^{t\mathcal{L}_{\alpha}}\mathcal{P}_{\mathcal{H}}\|_{\mathcal{B}({H^k_{\alpha}(\mathbb{R}^d)})}=\|\mathcal{P}_{\mathcal{H}}\|_{\mathcal{B}({H^k_{\alpha}(\mathbb{R}^d)})}\leq K.
	\end{equation*}
	
\end{proof}

We also recall the following standard estimates, see, e.g., \cite[Lemma 3.2]{Kapitula2}.
\begin{lemma}\label{lem3.4.4}
	The semigroup $S_{\Delta_y}(t)$ generated by the linear operator $\Delta_y$ for all $t>0$ satisfies the following decay estimates with some $\beta>0$:
	\begin{itemize}
		\item[(a)] $\|S_{\Delta_y}(t)u\|_{H^k(\mathbb{R}^{d-1})}\leq C\|u\|_{H^k(\mathbb{R}^{d-1})}$,
		\item[(b)] $\|S_{\Delta_y}(t)u\|_{H^k(\mathbb{R}^{d-1})}\leq C(1+t)^{-\frac{d-1}{4}}\|u\|_{L^1(\mathbb{R}^{d-1})}+Ce^{-\beta t}\|u\|_{H^k(\mathbb{R}^{d-1})}$,
		\item[(c)] $\|\nabla_y S_{\Delta_y}(t)u\|_{H^k(\mathbb{R}^{d-1})}\leq C t^{-1/2}\|u\|_{H^k(\mathbb{R}^{d-1})}$,
		\item[(d)] $\|\nabla_y S_{\Delta_y}(t)u\|_{H^k(\mathbb{R}^{d-1})}\leq  C(1+t)^{-\frac{d+1}{4}}\|u\|_{L^1(\mathbb{R}^{d-1})}+Ct^{-\frac{1}{2}}e^{-\beta t}\|u\|_{H^k(\mathbb{R}^{d-1})}.$\newline
	\end{itemize}
\end{lemma}

\section{The system of evolution equations}\lb{subs3.2.3}

In this section we derive the system of evolution equations \eqref{sys2}  governing the perturbation of the planar front, by following 
  \cite{Kapitula2} with modifications needed to accommodate the presence of the weight.
% and prove the well-posedness for it. % Throughout, we denote ${H^k(\mathbb{R}^d)}=H^k(\bbR^d)$, ${H^k_{\alpha}(\mathbb{R}^d)}=H_{\alpha}^k(\bbR)\otimes H^k(\bbR^{d-1})$ and $\cE={H^k(\mathbb{R}^d)}\cap\cE_{\a}$.

%For $\cP$ as defined in \eqref{DFNPP2}, 
We denote $\ran\mathcal{P}=\{U\in {H^k_{\alpha}(\mathbb{R}^d)}^n: U=\mathcal{P}U\}$ and $\ran\mathcal{Q}=\{U\in {H^k_{\alpha}(\mathbb{R}^d)}^n: U=\mathcal{Q}U\}$. In fact, if $U\in\ran\mathcal{Q}$, then $\pi_{\alpha}U=0$ because $\mathcal{P}U=0$.%, see formula \eqref{DFNPP0}. %Here, $\cP$ is the projection defined in \eqref{DFNPP2}.
 Hypothesis  \ref{hypo3.2.3} and Lemma \ref{lem3.2.6} imply that $\phi'\in\mathcal{H}^{n}$, therefore if $v\in\mathcal{H}^{n}\hookrightarrow{H^k_{\alpha}(\mathbb{R}^d)}^n$, then $\mathcal{P}v\in\mathcal{H}^{n}$, and then $\mathcal{Q}v=(I-\mathcal{P})v\in\mathcal{H}^{n}$. Hence we may define $\mathcal{P}_{\mathcal{H}}$ and $\mathcal{Q}_{\mathcal{H}}$ to be the restrictions of  $\mathcal{P}$ and $\mathcal{Q}$ to  $\mathcal{H}^{n}$. %and given by restricting $\mathcal{P}$ and $\mathcal{Q}$ to $\mathcal{H}^{n}$. 
 Since $\mathcal{H}^{n}\hookrightarrow {H^k_{\alpha}(\mathbb{R}^d)}^n$, the operators $\mathcal{P}_{\mathcal{H}}$ and $\mathcal{Q}_{\mathcal{H}}$ are also bounded. It follows from Lemmas~\ref{lem3.2.5} and \ref{cpcq}  that $\mathcal{H}^{n}=\ran\mathcal{P}_{\mathcal{H}}\oplus\ran\mathcal{Q}_{\mathcal{H}}$, where $\ran\mathcal{P}_{\mathcal{H}}=\ran\mathcal{L}_{\alpha}\cap \mathcal{H}^{n}$.%, see Lemma \ref{lem3.2.5} and Lemma~\ref{cpcq} above.

%In the space $\cE^n$ we consider a solution of equation \eqref{sysu} of the form 
%\begin{equation}\label{eq3.5.33}u(z,y,t)=\phi(z-q(y,t))+v(z,y,t),\ (z,y)\in\bbR^d,\end{equation}
%where $(v,q)\in\ran\mathcal{Q}_{\mathcal{H}}\otimes H^k(\bbR^{d-1}).$
%Here, we decompose the perturbation $u$ of the front $\phi$ into a spatial translation component, i.e., the component in the direction of the front $\phi(z-q(y,t))$, and a normal component $v$, so that $v=v(\cdot,y,t)$ belongs to %the range of $\mathcal{L}_{1,\alpha}$, $\Ran Q_{\alpha}=\Ran\cL_{1,\alpha}$, for each $(y,t)\in\bbR^{d-1}\times\bbR^+.$ 

The following lemma  shows % that the coordinate system $(v,q)$ is well defined. 
 that for any   sufficiently small $\tilde{v}\in \mathcal{H}^n$,  there exists a unique pair $(v,	q)\in\ran\mathcal{Q}_{\mathcal{H}}\times H^k(\mathbb{R}^{d-1})$ such that $\phi+\tilde{v}$ can be uniquely expressed by means of $(v,	q)$. %=\phi(\cdot-q)+v$.
\begin{lemma}\label{lem3.5.11}
	Assume  Hypothesis   \ref{hypo3.2.3} and   
  $k\geq[ \frac{d+1}{2}]$.  For any $\tilde{v}\in \mathcal{H}^n$ small enough,  there exists $(v,q)\in \ran\mathcal{Q}_{\mathcal{H}}\times H^k(\mathbb{R}^{d-1})$ such that 
	\begin{equation}\label{introtildev}
	\phi(z)+\tilde{v}(z,y)=\phi(z-q(y))+v(z,y),\, (z,y)\in\mathbb{R}^d.
	\end{equation}
\end{lemma}
\begin{proof}
	As in the proof of \cite[Lemma 2.2]{Kapitula2}, for any $q\in H^k(\mathbb{R}^{d-1})$ we write
	\begin{equation*}
	\phi(z-q(y))-\phi(z)=-q(y)\int_0^1\phi'(z-sq(y))\,ds.
	\end{equation*}
	%where $\phi'$ is the derivative of $\phi$ with respect to $z$. 
	Since $q\in H^k(\bbR^{d-1})\hookrightarrow L^{\infty}(\mathbb{R}^{d-1})$, we have 
		\begin{equation*}
	|\phi'(z-sq(y))|\leq Ke^{-\omega_{\pm}(z-sq(y))}%=Ke^{sq(y)}e^{-\omega_{\pm}z}
	\leq Ce^{-\omega_{\pm}z},
	\end{equation*}
	where $C$ is a constant that depends on $q$. 
	%We below  denote by $C$ positive constants that are not necessarily are the same.
	%and the fact that $q\in H^k(\bbR^{d-1})\hookrightarrow L^{\infty}(\mathbb{R}^{d-1})$. 
	By Hypothesis \ref{hypo3.2.3} then 
	$$\int\limits_{\bbR}|\phi'(z-sq(y))|^2\gamma_{\alpha}^2(z)\,dz \leq C\left(\int_0^{\infty}e^{-2\omega_+ z}e^{2\alpha_+z}dz+\int_{-\infty}^{0}e^{-2\omega_-z}e^{2\alpha_-z} dz\right), $$
	and, thus, 
	% right-hand side of the last equation is in $\cE^n$	
	%We claim that the right-hand side of the last equation is in $\cE^n$ since $\phi'$ and its derivatives decay to zero as $z\rightarrow\pm\infty$ faster than $\gamma_{\alpha}$ by Hypothesis \ref{hypo3.2.3}. Indeed, since $\omega_+>\alpha_+$ and $-\omega_->\alpha_-$ we have
%	\begin{align*}
	$\|\phi'(\cdot-sq(\cdot))q(\cdot)\|^2_{L^2_{\alpha}(\bbR^d)}%=\int\limits_{\mathbb{R}^{d-1}}|q(y)|^2\int\limits_{\bbR}|\phi'(z-sq(y))|^2\gamma_{\alpha}^2(z)\,dz\,dy\\
	\leq C\|q\|_{L^2(\bbR^{d-1})}$, so %\left(\int_0^{\infty}e^{-2\omega_+ z}e^{2\alpha_+z}dz+\int_{-\infty}^{0}e^{-2\omega_-z}e^{2\alpha_-z} dz\right)\\
	%\leq C\|q\|_{H^k(\bbR^{d-1})},
	%\end{align*}	
	%where we used the estimate 
	%\begin{equation*}
	%|\phi'(z-sq(y))|\leq Ke^{-\omega_{\pm}(z-sq(y))}=Ke^{sq(y)}e^{-\omega_{\pm}z}\leq Ce^{-\omega_{\pm}z}
	%\end{equation*}
	%and the fact that $q\in H^k(\bbR^{d-1})\hookrightarrow L^{\infty}(\mathbb{R}^{d-1})$. 
	 $\phi(\cdot-q)-\phi(\cdot) \in \mathcal{H}^n$  if  $q\in H^k(\mathbb{R}^{d-1})$. 
	 
	We then write \eqref{introtildev} as
\begin{equation}\label{eqtildev}
	\tilde{v}(z,y)=v(z,y)-q(y)\int_0^1\phi'(z-sq(y))\,ds
\end{equation}
	and apply $\pi_{\alpha}$ (see \eqref{DFNPP0}). %, to both sides of \eqref{eqtildev}.
	 Since $v\in\ran\cQ_{{\mathcal{H}}}=\ker\cP_{{\mathcal{H}}}$, % we have that $\pi_{\alpha} (v)=0$, and equation \eqref{eqtildev} becomes
	\begin{equation*}
	\pi_{\alpha}(\tilde{v}(z,y))=-q(y)\left(\int_0^1\pi_{\alpha}(\phi'(z-sq(y)))\,ds\right).
	\end{equation*}
	 We consider the mapping $(q,\tilde{v})\mapsto \mathcal{G}(q,\tilde{v})     $ defined by 
	\begin{equation*}%\label{mapg}
	\mathcal{G}(q(y),\tilde{v}(z,y))=\pi_{\alpha}(\tilde{v}(z,y))+q(y)\left(\int_0^1\pi_{\alpha}(\phi'(z-sq))\,ds\right)
	\end{equation*}
	as  %$\mathcal{G}(q,\tilde{v}) $ is 
	a mapping  from $ H^k(\mathbb{R}^{d-1})\times \mathcal{H}^n$ to $ H^k(\mathbb{R}^{d-1})$  such that  $\mathcal{G}(0,0)=0$ and $\frac{\partial\mathcal{G}}{\partial q}(0,0)=I$. For any $\tilde{v}$ near $\tilde{v}=0$, 
	 the  Implicit Function Theorem yields the existence of a unique $q$ as a function of $\tilde{v}$  so that  $\mathcal{G}(q,\tilde{v})=0$. 
	 
	So, given a  $\tilde v$, we first find   $q$   from the equation $\mathcal{G}(q,\tilde{v})=0$ and  then, to identify $v$  that corresponds to  that $q$  we  apply $\mathcal{Q}_{\mathcal{H}}$ to \eqref{eqtildev} and set   $v=\mathcal{Q}_{\mathcal{H}}v$, thus obtaining the following formula, 
	$$v%=\mathcal{Q}_{\mathcal{H}}v
	=\mathcal{Q}_{\mathcal{H}}\tilde{v}+\mathcal{Q}_{\mathcal{H}}\left( q \int_0^1 \phi'(\cdot-sq)\,ds \right).$$ 
	%Since $q(\tilde{v})$ has been found by solving \eqref{mapg}, 
	%Thus,  the  desired coordinate system  $(q,\tilde v)$ is well-defined.	
\end{proof}

Since   the coordinate system  $(v,q)\in\ran\mathcal{Q}_{\mathcal{H}}\otimes H^k(\bbR^{d-1})$ is well defined by Lemma \ref{lem3.5.11},
   we  can decompose solutions of  \eqref{sysu} that are close to the front $\phi$  as a sum of  a spatial translation component, i.e., the component in the direction of the front $\phi(z-q(y,t))$, and a normal component $v$, so that $v=v(\cdot,y,t)$ belongs to %the range of $\mathcal{L}_{1,\alpha}$, 
$\Ran Q_{\alpha}=\Ran\cL_{1,\alpha}$, for each $(y,t)\in\bbR^{d-1}\times\bbR^+$. In other words, we can  write  a solution $u$ of equation \eqref{sysu} in ${\mathcal{H}}^n$  as 
\begin{equation}\label{eq3.5.33}
u(z,y,t)=\phi(z-q(y,t))+v(z,y,t),\ (z,y)\in\bbR^d,
\end{equation}
where $(v,q)\in\ran\mathcal{Q}_{\mathcal{H}}\otimes H^k(\bbR^{d-1}).$ For convenience, in what  follows, we denote   $\phi_{q}(z)=\phi(z-q)$.
%Here, we decompose the perturbation $u$ of the front $\phi$ into a spatial translation component, i.e., the component in the direction of the front $\phi(z-q(y,t))$, and a normal component $v$, so that $v=v(\cdot,y,t)$ belongs to %the range of $\mathcal{L}_{1,\alpha}$,  $\Ran Q_{\alpha}=\Ran\cL_{1,\alpha}$, for each $(y,t)\in\bbR^{d-1}\times\bbR^+.$ 

We  substitute \eqref{eq3.5.33} into the equation \eqref{sysu}.  Repeating computations   from   \cite[Section 2]{Kapitula2}, we see that $v$  solves  the equation
\begin{equation}\label{eq3.2.16}
\partial_t v= Lv +\left(df(\phi_q)-df(\phi)\right)v+N(\phi_q,v)v+ \big(\partial_tq-\Delta_yq\big)\phi'_q+(\nabla_yq\cdot\nabla_y q)\phi''_q,
\end{equation}
where 
$L$ is the differential expression defined in \eqref{eq3.1.2}, $\nabla_y q = (\partial_{x_2}q, \cdots, \partial_{x_d}q)$, and  
\begin{equation}\label{eq3.2.12}
N(u,v)=\int_{0}^{1}df(u+sv)-df(u)ds,
\end{equation}
is an $n\times n$ matrix-valued function of $(u,v)$.

We assume that $v(\cdot,\cdot,t)\in \ran\cQ_{{\mathcal{H}}}\cap\mathcal{H}^n$ for every $t\geq 0$, that it, $\mathcal{P}_{{\mathcal{H}}}v=0$, and  apply the projection $\mathcal{P}_{{\mathcal{H}}}$  to \eqref{eq3.2.16}, thus obtaining  an equation for  $q$,
\begin{equation}\label{prop}
(-\pi_{\alpha}\phi'_q)\partial_t q=(\pi_{\alpha}\phi''_q)(\nabla_yq\cdot\nabla_y q)-(\pi_{\alpha}\phi'_q)\Delta_yq
+\pi_{\alpha}(\left(df(\phi_q)-df(\phi)\right)v+N(\phi_q,v)v).
\end{equation}

%It would be convinient to divide both sides of \eqref{prop} by $-\pi_{\alpha}(\phi_q')(y)$. For this we use t

The following result is proved in  \cite[Lemma 2.3]{Kapitula2}. It  %(one-dimensional case is discussed in  \cite[Lemma 4.1, one-dimensional case]{GLS2}) 
shows that $\pi_{\alpha}(\phi_q')(y)$ is not close to zero. 
\begin{lemma}\label{lem3.3.6}
	 There are constants $\delta_0$ and $C>0$ such that if $\|q\|_{L^{\infty}(\bbR^{d-1})}<\delta_0$, then  	for all $y\in\mathbb{R}^{d-1}$
	\begin{align*}
	&1-C\delta_0\leq 1-C\|q\|_{L^{\infty}(\bbR^{d-1})}\leq|\pi_{\alpha}(\phi_q')(y)|\leq 1+C\|q\|_{L^{\infty}(\bbR^{d-1})}\leq 1+C\delta_0,\\
	&C(1-\delta_0)\leq C(1-\|q\|_{L^{\infty}(\bbR^{d-1})})\leq|\pi_{\alpha}(\phi_q'')(y)|\leq C(1+\|q\|_{L^{\infty}(\bbR^{d-1})})\leq C(1+\delta_0).
	\end{align*}
%	for all $y\in\mathbb{R}^{d-1}$.%, where $\phi_q=\phi(\cdot -q)$.
\end{lemma}
For  $\delta_0$  as in Lemma \ref{lem3.3.6}, we assume that  $\|q\|_{L^{\infty}(\mathbb{R}^{d-1})}\leq \delta_0$ and, 
% for any $(v,q)\in \cE\times H^k(\bbR^{d-1})$, we 
denote
\begin{equation}\label{eq3.3.18}
G(v,q)=\left(df(\phi_q)-df(\phi)\right)v+N(\phi_q,v)v,\qquad
K_1(q)=-\displaystyle{\frac{\pi_{\alpha}\phi''_q}{\pi_{\alpha}\phi'_q}}, \qquad K_2(q)=-\displaystyle{\frac{1}{\pi_{\alpha}\phi'_q}}.
\end{equation}
Lemma \ref{lem3.3.6}  allows us to   divide both sides of \eqref{prop} by $\pi_{\alpha}\phi'_q$  and  obtain 
\begin{align}\label{eqq}
\partial_t q=\Delta_yq+K_1(q)(\nabla_yq)\cdot(\nabla_yq)+K_2(q)\pi_{\alpha}(G(v,q)).
\end{align}

The following lemma is  proved by minor  modifications of the argument leading to \cite[eq(2.23)]{Kapitula2}. It will be used to derive various estimates for nonlinearities in evolution equations studied below.
\begin{lemma}\label{k1k2}
	%Assume Hypothesis \ref{MH}. 
	Let the functions $K_1=K_1(q)(y)$ and $K_2=K_2(q)(y)$ for $q\in H^k(\bbR^{d-1})$ be defined as in \eqref{eq3.3.18}. There exist constants $\delta_0$ and $C>0$ such that for $\|q\|_{H^k(\bbR^{d-1})}\leq\delta_0$ we have
	\begin{equation}\label{eq3.2.21}
	\|K_i(q)\|_{L^{\infty}(\mathbb{R}^{d-1})}\leq C(1+\|q\|_{H^k(\mathbb{R}^{d-1})}),\, i=1,2.
	\end{equation}
	Moreover, the formulas for $K_i$, $i=1,2$, define locally Lipschitz mappings $q\mapsto K_i(q)$ from $H^k(\mathbb{R}^{d-1})$ to $L^{\infty}(\mathbb{R}^{d-1})$.
\end{lemma}
We  return to the task of deriving the evolution equation for the perturbation $\phi_q+v$. Applying the projection operator $\mathcal{Q}_{{\mathcal{H}}}$ to  the equation \eqref{eq3.2.16} yields the equation
\begin{equation}\label{eqv}
\partial_tv=Lv+\mathcal{Q}_{{\mathcal{H}}}\big(G(v,q)+\big(\partial_tq-\Delta_yq\big)\phi'_q+(\nabla_yq)\cdot(\nabla_yq)\phi''_q\big),
\end{equation}
where $G(v,q)$ is defined  in \eqref{eq3.3.18}.
Combining \eqref{eqv} and \eqref{eqq} we have the system
\begin{align}\label{sys1}
&\partial_tv=Lv+\mathcal{Q}_{{\mathcal{H}}}\big(G(v,q)+\big(\partial_tq-\Delta_yq\big)\phi'_q+(\nabla_yq)\cdot(\nabla_yq)\phi''_q\big)\no\\
&\partial_tq=\Delta_yq+K_1(q)(\nabla_yq)\cdot(\nabla_yq)+K_2(q)\pi_{\alpha}G(v,q).
\end{align}
We further denote %To make the system look nicer, define 
\begin{eqnarray}w(y)&=&\nabla_yq(y),\, y\in\mathbb{R}^{d-1}, \notag\\
\label{eqf1}
F_1(v,q,w)&=&G(v,q)+\big(\partial_tq-\Delta_yq\big)\phi'_q+(w\cdot w)\phi''_q, \\
\label{eqg1}
F_2(v,q,w)&=&K_1(q)(w\cdot w)+K_2(q)\pi_{\alpha}G(v,q).
\end{eqnarray}
Using %Because $\partial_tq-\Delta_yq=F_2(v,q,w)$ by 
\eqref{eqq} in \eqref{eqg1},  we obtain a relation between $F_1$ and $F_2$, 
\begin{equation}\label{eq3.2.25}
F_1(v,q,w)=G(v,q)+F_2(v,q,w)\phi'_q+(w\cdot w)\phi''_q.
\end{equation}
From \eqref{eq3.2.16}, using $v\in\ran \cQ_{{\mathcal{H}}}$ and $\phi_q'\in\ker P_{{\mathcal{H}}}$, we obtain
\begin{equation*}
\mathcal{P}_{{\mathcal{H}}}\left( G(v,q)+(\partial_t q-\Delta_y q)\phi'_q+(w\cdot w)\phi''_q \right)=0,
\end{equation*}
which implies that $F_1(v,q,w)=\cQ_{{\mathcal{H}}} F_1(v,q,w)%$ and $
%\mathcal{Q}_{{\mathcal{H}}}F_1(v,q,w)%\mathcal{Q}_{{\mathcal{H}}}\left( G(v,q)+(\partial_t q-\Delta_y q)\phi'_q+(w\cdot w)\phi_q'' \right)\\&
=G(v,q)+(\partial_t q-\Delta_y q)\phi'_q+(w\cdot w)\phi''_q.
$ 
Thus applying $\nabla_y$ to  \eqref{sys1} we finally arrive to the  system for $(v,q,w)\in\ran\mathcal{Q}_{{\mathcal{H}}}\times H^k(\mathbb{R}^{d-1})\times H^k(\mathbb{R}^{d-1})^{d-1}$
 that we shall study
\begin{equation}\label{sys2}
\begin{aligned}
& \partial_tv=Lv+F_1(v,q,w),\\
& \partial_tq=\Delta_y q+F_2(v,q,w),\\
& \partial_tw=\Delta_yw+\nabla_y\cdot F_2(v,q,w).
\end{aligned}
\end{equation}

%In what follows we denote
%\begin{eqnarray}\mathcal{X}= \ran\mathcal{Q}_{\mathcal{H}}\times H^k(\mathbb{R}^{d-1})\times H^k(\mathbb{R}^{d-1})^{d-1}. \label{x}\end{eqnarray}
%where the nonlinear terms $F_1$ is defined in \eqref{eqf1} and \eqref{eq3.2.25}, and $F_2$ is defined in \eqref{eqg1}.
%This is the main system of equations which we will study. 

% for $(v,q,w)\in\ran\mathcal{Q}_{\cE}\times H^k(\mathbb{R}^{d-1})\times H^k(\mathbb{R}^{d-1})^{d-1}$
%$\ran\cQ_{\cE}\times H^k(\bbR^{d-1})\times H^k(\bbR^{d-1})$

\section{Estimates  for the nonlinear terms}\label{sec6}

In this section we obtain estimates for  the nonlinear terms in \eqref{sys2}. %Recall that $\|\cdot\|_0=\|\cdot\|_{H^k(\mathbb{R}^d)}$ and $\|\cdot\|_{\alpha}=\|\cdot\|_{H^k_{\alpha}(\mathbb{R})\otimes H^k(\mathbb{R}^{d-1})}$. %, and, in addition,  \textcolor{red}{$\|\cdot\|_{H^k(\mathbb{R}^{d-1})}=\|\cdot\|_{H^k(\mathbb{R}^{d-1})}$ IS THIS REALLY A GOOD NOTATION?}.  
Below  we use  the fact    \cite[Theorem 4.39]{AF}, that,  for $2k>d$,   the Sobolev embedding   yields the inequality
\begin{equation}\label{eq2.3.3}
\|uv\|_{H^k(\bbR^d)}\leq C\|u\|_{H^k(\bbR^d)}\|v\|_{H^k(\bbR^d)}.
\end{equation}
%for $2k>d$ (see \cite[Theorem 4.39]{AF}). %We begin with some elementary facts.

\begin{lemma}\label{soblev}
	For  $k\geq[\frac{d+1}{2}]$,
	%and consider $ H^k(\mathbb{R}^d)=H^k(\mathbb{R}^d)$.
	 the following assertions hold.
		\begin{itemize}
	\item[(1)] If $u$, $v\in  H^k(\mathbb{R}^d)$, then $uv\in  H^k(\mathbb{R}^d)$. Moreover,   there exists a constant $C>0$ such that $\|uv\|_{H^k(\bbR^d)}\leq C\|u\|_{H^k(\bbR^d)}\|v\|_{H^k(\bbR^d)}$ for all $u$, $v\in  H^k(\mathbb{R}^d)$.
		\item[(2)] If $u$, $v\in \mathcal{H}$, then $uv\in {H^k_{\alpha}(\mathbb{R}^d)}$. Moreover, there exists a constant $C>0$ such that $\|uv\|_{H^k_{\alpha}(\bbR^d)}\leq C\|u\|_{H^k(\bbR^d)}\|v\|_{H^k_{\alpha}(\bbR^d)}$ for all  $u$, $v\in \mathcal{H}$.
		\item[(3)] If $u$, $v\in \mathcal{H}$, then $uv\in \mathcal{H}$. Moreover, there exists a constant $C>0$ such that $\|uv\|_{\mathcal{H}}\leq C\|u\|_{\mathcal{H}}\|v\|_{\mathcal{H}}$ for all   $u$, $v\in \mathcal{H}$.
	\end{itemize}
\end{lemma}
\begin{proof}
The proof  is  similar to the one-dimensional estimates in \cite[Proposition 7.1]{GLS2}.%, thus will be omitted from here.
\end{proof}

%To proof the main estimates, we will also frequently use the following lemma:
\begin{lemma}\label{estq}
	For  $k\geq[\frac{d+1}{2}]$,  if $q_1$, $q_2\in H^k(\mathbb{R}^{d-1})$ and $\psi \in  H^{k+2}(\mathbb{R})$ is such that $\psi'(z)\to 0 $ exponentially  as $z\rightarrow\pm\infty$, then  the function  $\sigma(z,y)=\psi'(z-q_1(y))q_2(y)$, $(z,y)\in\bbR^d$, satisfies 
	\begin{equation*}%\label{ests12}
	\|\sigma\|_{H^k(\mathbb{R}^d)}%=\|\psi'(\cdot-q_1(\cdot))q_2(\cdot)\|_{H^k(\mathbb{R}^d)}
	\leq C\|q_2\|_{H^k(\mathbb{R}^{d-1})},
	\end{equation*}
	where  $C=C(\|\psi\|_{L^{\infty}(\mathbb{R})}, \|\psi^{\prime} \|_{H^{k+1}(\mathbb{R})}, \|q_1\|_{H^k(\mathbb{R}^{d-1})})$ is bounded in each ball of the form 
	$
	\{ q_1:\|q_1\|_{H^k(\mathbb{R}^{d-1})}\leq K\}.
	$
\end{lemma}
\begin{proof}
The  derivatives of $\sigma$ are given by
	\begin{align}\lb{eq3.55.4}
	&\frac{\partial \sigma}{\partial z}=\psi''(z-q_1(y))q_2(y),\no\\
	&\frac{\partial \sigma}{\partial x_j}=\psi''(z-q(y))q_2(y)\frac{\partial q_1}{ \partial x_j}+\psi'(z-q_1(y))\frac{\partial q_2 }{\partial x_j}, \,\,\,j=2,...,d.
	\end{align}
	Since $\psi'$ is exponentially decaying to $0$, we have
	\begin{align}\label{eq3.55.2}
	\|\sigma\|^2_{L^2(\mathbb{R}^d)}= \int\limits_{\mathbb{R}^{d-1}}\left( \int\limits_{\mathbb{R}}|\psi'(z-q_1(y))|^2\, dz \right)|q_2(y)|^2\, dy
	%=\int\limits_{\mathbb{R}^{d-1}}\left( \int\limits_{\mathbb{R}}|\psi'(z)|^2\, dz \right)|q_2(y)|^2\, dy
	\leq C\|q_2\|_{H^k(\mathbb{R}^{d-1})}.
	\end{align}
	%since $\psi'$ is exponentially decaying to $0$. 
	Similarly,  
	\begin{align}\lb{eq3.55.3}
	\|\psi''(z-q(y))q_2(y)\dfrac{\partial q_1(y)}{\partial x_j}\|_{L^2(\bbR^d)}&%\leq C\left\|q_2\frac{\partial}{\partial x_j}q_1(\cdot)\right\|_{L^2(\bbR^{d-1})}
	\leq C\|q_2\|_{L^{\infty}(\bbR^{d-1})}\left\|\frac{\partial q_1}{\partial x_j}\right\|_{L^2(\bbR^{d-1})}\leq C\|q_2\|_{H^k(\bbR^{d-1})}\|q_1\|_{H^k(\bbR^{d-1})}.
	\end{align}
	The statement of the lemma then is  proved by a  calculation similar to the proof of Proposition~A.3  in the Appendix. %concludes the proof of the lemma.
	Indeed,  instead of equation \eqref{eq2.3.18} in the proof of Proposition A.3, we may   use  relations
	\begin{equation*}
	k-\frac{d-1}{2}=k-\frac{d}{2}+\frac{1}{2}>n_i-\frac{d}{p_i}+\frac{1}{p_i}=n_i-\frac{d-1}{p_i},
	\end{equation*}
	which  proves the embedding $H^k(\mathbb{R}^{d-1})\hookrightarrow W^{n_i,p_i}(\bbR^{d-1})$ by Lemma~A.1 in the Appendix.
\end{proof}
Using Lemma \ref{estq} we now prove the following estimates of  the $H^k(\mathbb{R}^{d})$-norm and the weighted norm of the nonlinear term $G(v,q)$ introduced in \eqref{eq3.3.18}.

\begin{proposition}\label{prop3.3.2}
	Assume Hypotheses \ref{hypo3.2.1} and  \ref{hypo3.2.2}. For $k\geq[\frac{d+1}{2}]$, %let ${H^k(\mathbb{R}^d)}=H^k(\bbR^d)$.%, and let ${H^k_{\alpha}(\mathbb{R}^d)}$ and $\cE$ be defined accordingly. T
	 the following assertions hold:
	\begin{itemize}
		\item[(1)] Formula $(v,q)\mapsto (df(\phi_q)-df(\phi))v$ defines a mapping from $ H^k(\mathbb{R}^d)^{n}\times H^{k}(\mathbb{R}^{d-1})$ to $ H^k(\mathbb{R}^d)^{n}$ that is locally Lipschitz on any set of the form $\{ (v,q):\|v\|_{H^k(\bbR^d)}+\|q\|_{H^k(\mathbb{R}^{d-1})}\leq K \}$. On such a set there is a constant $C_K$ depending on $K$ such that $$\|(df(\phi_q)-df(\phi))v\|_{H^k(\bbR^d)}\leq C_K\|q\|_{H^k(\mathbb{R}^{d-1})}\|v\|_{H^k(\bbR^d)}.$$
		\item[(2)] Formula $(v,q)\mapsto (df(\phi_q)-df(\phi))v$ defines a mapping from $\mathcal{H}^{n}\times H^k(\mathbb{R}^{d-1})$ to $\mathcal{H}^{n}$ that is locally Lipschitz on any set of the form $\{ (v,q):\|v\|_{\mathcal{H}}+\|q\|_{H^k(\mathbb{R}^{d-1})}\leq K \}$. On such a set there is a constant $C_K$ depending on $K$ such that 
		\begin{equation*}
		\|(df(\phi_q)-df(\phi))v\|_{H^k_{\alpha}(\bbR^d)}\leq C_K\|q\|_{H^k(\mathbb{R}^{d-1})}\|v\|_{H^k_{\alpha}(\bbR^d)}, \end{equation*} and, therefore, 
		\begin{equation*}\|(df(\phi_q)-df(\phi))v\|_{\mathcal{H}}\leq C_K\|q\|_{H^k(\mathbb{R}^{d-1})}\|v\|_{\mathcal{H}}.
		\end{equation*}
	\end{itemize}
\end{proposition}
\begin{proof}
	We define $p(q,v)\in H^k(\mathbb{R}^d)$ for $q\in H^k(\mathbb{R}^{d-1})$ and $v\in H^k(\mathbb{R}^d)$ by the formula 
	$$p(q,v)(z,y)=\bigg(df\big(\phi(z-q(y))\big)-df\big(\phi(z)\big)\bigg)v(x),$$  so that
	\begin{equation}\label{eq2.6.54}
	p(q,v)=\int_0^1\frac{d}{ds} df\big( \phi(\cdot-sq) \big)vds=-\int_0^1d^2f\big(\phi(\cdot-sq)\big)\big( \phi'(\cdot-sq)q,v \big)\,ds.
	\end{equation}
	Since $x\mapsto d^2f\big( \phi(z-sq(y)) \big)$ is a smooth function with bounded derivatives, using Lemma \ref{estq} we conclude that $p(q,v)\in H^k(\mathbb{R}^d)$ and
	$
	\|p(q,v)\|_{H^k(\mathbb{R}^d)}\leq C_K\|q\|_{H^k(\mathbb{R}^{d-1})}\|v\|_{H^k(\mathbb{R}^d)}.
	$
	We then multiply \eqref{eq2.6.54} by $\gamma_{\alpha}$ and infer
	$
	\|p(q,v)\|_{H^k_{\alpha}(\bbR^d)}\leq C_K\|q\|_{H^k(\mathbb{R}^{d-1})}\|v\|_{H^k_{\alpha}(\bbR^d)}$.

	To show the local Lipschitz estimates for $p(q,v)$ and $\gamma_{\alpha}p(q,v)$, we pass to components in the vector equation \eqref{eq2.6.54}. It is enough to show that  the map $(q_1,q_2,v) \mapsto    \tilde{l}(q_1,q_2,v)$ defined by 
	\begin{equation*}
	\tilde{l}(q_1,q_2,v)(z,y)=l(\psi(z-q_1(y)))\psi'(z-q_1(y))q_2(y)v(x),\,x=(z,y)\in\mathbb{R}^d
	\end{equation*}
	  is a locally Lipschitz  map  from $	H^k(\mathbb{R}^{d-1})\times H^k(\mathbb{R}^{d-1})\times H^k(\mathbb{R}^{d})$ into $ H^k(\mathbb{R}^{d})$. Here $\psi:\mathbb{R}\rightarrow\mathbb{R}$ is a function that is exponentially decaying to some constants $\psi_{\pm}$ as $z\rightarrow\pm\infty$,  the derivatives $\psi^{(m)}(z)\rightarrow0$, $m=1$, $2$, ..., $k+2$,  as $z\rightarrowtail\pm\infty$ exponentially, and $l:\mathbb{R}\rightarrow\mathbb{R}$ is a $C^{k+3}$ function with bounded derivatives.

	 Recall that $k\geq [\frac{d+1}{2}]$ and thus $H^k(\mathbb{R}^d)\hookrightarrow L^{\infty}(\mathbb{R}^d)$. The derivatives of $l$ are bounded, so by  Lemma \ref{estq} we have
	\begin{equation*}
	\|\tilde{l}(q_1,q_2,v)\|_{H^k(\mathbb{R}^d)}\leq C\|l\|_{C^{k+3}}\|q_2\|_{H^k(\mathbb{R}^{d-1})}\|v\|_{H^k(\mathbb{R}^d)},
	\end{equation*}
	and thus the map $\tilde{l}$ is well defined. 
	
	We will now proceed with the local Lipschitz estimates for $\tilde{l}$. To show the estimate for the variation in $q_1$, we  fix $q_2$, $v$ and write:
	\begin{align*}
	\tilde{l}(q_1,q_2,v)-\tilde{l}(\bar{q}_1,q_2,v)=&\bigg( l\big( \psi(\cdot-q_1) \big)-l\big( \psi(\cdot-\bar{q}_1) \big)  \bigg)\psi'(\cdot-q_1)q_2v\\
	&+l\big( \psi(\cdot-\bar{q}_1) \big)\big( \psi'(\cdot-q_1)-\psi'(\cdot-\bar{q}_1) \big)q_2 v,
	\end{align*}
	where
	\begin{align*}\nonumber
	l\big( \psi(z-q_1(y)) \big)-l\big( \psi(z-\bar{q}_1(y)) \big)&=\int_0^1\frac{d}{ds}l\bigg( \psi\big( z-q_1(y)-(s-1)(q_1(y)-\bar{q}_1(y)  ) \big) \bigg)\,ds\\%\label{tildel}
	&=- \int_0^1 l'(\psi(\cdot))\psi'(\cdot)(q_1(y)-\bar{q}_1(y))\,ds.
	\end{align*}
	Applying Lemma \ref{estq} again  we get
	\begin{equation*}
	\|l\big( \psi(\cdot-q_1) \big)-l\big( \psi(\cdot-\bar{q}_1) \big) \|_{H^k(\mathbb{R}^d)}\leq C_K\|q_1-\bar{q}_1\|_{H^k(\mathbb{R}^{d-1})}.
	\end{equation*}
	On the other hand, 
	\begin{align*}
	\psi'(z-q_1(y))-\psi'(z-\bar{q}_1(y))&=\int_0^1\frac{d}{ds}\psi'\big( z-q_1(y)-(s-1)(q_1(y)-\bar{q}_1(y))\big)\,ds\\
	&=-\int_0^1\psi''(\cdot)\big( q_1(y)-\bar{q}_1(y) \big)\, ds.
	\end{align*}
	Another application of Lemma \ref{estq} yields
	\begin{equation*}
	\|\tilde{l}(q_1,q_2,v)-\tilde{l}(\tilde{q}_1,q_2,v)\|_{H^k(\mathbb{R}^d)}\leq C_K\|q_1-\bar{q}_1\|_{H^k(\mathbb{R}^{d-1})}.
	\end{equation*}
	The estimate for the variation in $q_2$ are similar.
	The estimate for the variation in $v$ follows from Proposition A.3 in the Appendix by fixing $q_1$ and $q_2$.
	
	Multiplying $\tilde{l}$ by $\gamma_{\alpha}$ and working with $l(\psi(\cdot-q_1))\psi'(\cdot-q_1)q_2\gamma_{\alpha}v$ gives the local Lipschitz estimate of $p(q,v)$ in the  weighted norm.
\end{proof}
The next statement  concerns the nonlinearity $N$ defined in  \eqref{eq3.2.12}.
\begin{proposition}\label{prop3.4.3}
	Assume Hypotheses \ref{hypo3.2.1} and  \ref{hypo3.2.2}, and let  $k\geq [\frac{d+1}{2}]$. %, let ${H^k(\mathbb{R}^d)}=H^k(\bbR^d)$, and let ${H^k_{\alpha}(\mathbb{R}^d)}$, and $\cE$ be defined accordingly. We recall notation \eqref{eq3.2.12}.
	\begin{itemize}
		\item[(1)] Formula $(v,q)\mapsto N(\phi_q,v)$ defines a mapping from $ H^k(\mathbb{R}^d)^{n}\times H^k(\mathbb{R}^{d-1})$ to $ H^k(\mathbb{R}^d)^{n^2}$ that is locally Lipschitz and $\mathcal{O}(\|v\|_{H^k(\mathbb{R}^d)})$ as $\|v\|_{H^k(\mathbb{R}^d)}\rightarrow 0$ uniformly on any bounded neighborhood of $(0,0)$ in $ H^k(\mathbb{R}^d)^{n}\times H^k(\mathbb{R}^{d-1})$.
		\item[(2)] Formula $(v,q)\mapsto N(\phi_q,v)v$ defines a mapping from $ H^k(\mathbb{R}^d)^{n}\times H^k(\mathbb{R}^{d-1})$ to $ H^k(\mathbb{R}^d)^{n}$ that is locally Lipschitz on any bounded neighborhood of $(0,0)$ in $ H^k(\mathbb{R}^d)^{n}\times H^k(\mathbb{R}^{d-1})$.
	\end{itemize}
\end{proposition}
\begin{proof}
	To prove (1), we note first that 
	$$
	N(\phi_q,v)%&=\int_0^1\left( df(\phi_q+sv)-df(\phi_q)  \right)ds\\
	=\int_0^1\int_0^1\frac{d}{d\tau}\left( df(\phi_q+s\tau v) \right)d\tau \,ds=\int_0^1\int_0^1d^2f(\phi_q+s\tau v)sv \,d\tau\,ds.
	$$
	
	%If we consider  $d^2f(\cdot)v $ componentwise,  then 
	It is enough to show that the following   map $\tilde{l}:H^k(\mathbb{R}^{d-1})\times H^k(\mathbb{R}^d)\times H^k(\mathbb{R}^d)\rightarrow H^k(\mathbb{R}^d)$  is locally Lipschitz. We  define 
$$	\tilde{l}(q,u,v)(z,y)=l\left(\psi(z-q(y)) ,u(x)\right)v(x), \,\,\,x=(z,y)\in\mathbb{R}^d,$$
 where $\psi:\mathbb{R}\rightarrow\mathbb{R}$ is a function exponentially decaying to some constants $\psi_{\pm}$ as $z\rightarrow\pm\infty$ and $\psi^{(m)}(z)\rightarrow 0$ exponentially, for any $m=1$, $2$, ..., and $l:\mathbb{R}\times \mathbb{R}\rightarrow \mathbb{R}$ is a $C^{k+3}$-smooth bounded function with bounded derivatives.
	
	Again,  $k\geq [\frac{d+1}{2}]$, so that  $H^k(\mathbb{R}^d)\hookrightarrow L^{\infty}(\mathbb{R}^d)$, and then 
	$
	\|\tilde{l}(q,u,v)\|_{L^2(\mathbb{R}^d)}\leq \|l\|_{C^{k+3}}\|v\|_{L^2(\mathbb{R}^d)}$.
If  $l'_j$ denotes  the derivative with respect to the $j$-th variable,  then
	\begin{align*}
	&\frac{\partial \tilde{l}}{\partial z}=l'_1(\cdot)\psi'(z-q(y))v(x)+l_2'(\cdot)\frac{\partial u}{\partial z}v(x)+l(\cdot)\frac{\partial v}{\partial z},\\
	&\frac{\partial \tilde{l}}{\partial x_j}= - l'_1(\cdot)\psi'(z-q(y))v(x)\frac{\partial q}{ \partial x_j}+l_2'(\cdot)\frac{\partial u}{\partial x_j}v(x)+l(\cdot)\frac{\partial v}{\partial x_j},\, j=2,\,...,d.
	\end{align*}
	
	Since $l'_1$, $l'_2$, $\psi'$ and $\frac{\partial q}{\partial x_j}$ are bounded,  $\tilde{l}(q,u,v)\in H^1(\mathbb{R}^d)$. A similar calculation (cf. the proof of Proposition A.3) with $q(z,y)$ in the proof replaced by $\phi'(z-q(y))$) shows that $\tilde{l}(q,u,v)\in H^k(\mathbb{R}^d)$. Thus, the map $\tilde{l}$ is well defined.
	Next we proceed with the proof of the local Lipschitz property. 
	
 Variation in $q$ gives
	\begin{align*}
	\tilde{l}(q,u,v)-\tilde{l}(\bar{q},u,v)%&=\bigg(l\big(\psi(z-q),u\big)-l\big(\psi(z-\bar{q}),u\big)\bigg)v\\
	&=\int_0^1\frac{d}{ds}l\bigg(\psi\big(z-q-(s-1)(q-\bar{q})\big),u\bigg)v\, ds\\&=-\int_0^1 l_1'\big( \psi(\cdot),u \big)\psi'(\cdot)(q-\bar{q})v\,ds.
	\end{align*}
	Since $l_1'$ and its derivatives are bounded, the main part of the estimate
	\begin{equation*}
	\| \tilde{l}(q,u,v)-\tilde{l}(\bar{q},u,v) \|_{H^k(\mathbb{R}^d)}\leq C_K\|q-\bar{q}\|_{H^k(\mathbb{R}^{d-1})}
	\end{equation*}
	on sets of the form $\{ (q,u,v):\|q\|_{H^k(\mathbb{R}^{d-1})}+\|u\|_{H^k(\mathbb{R}^d)}+\|v\|_{H^k(\mathbb{R}^d)}\leq K \}$ is reduced to Lemma \ref{estq}.
	
For variation in $u$, the estimate
	\begin{align*}
	\|\tilde{l}(q,u,v)-\tilde{l}(q,\bar{u},v)\|_{H^k(\mathbb{R}^d)}&=\| \big( l(\psi(z-q),u)-l(\psi(z-q),\bar{u}) \big)v  \|_{H^k(\mathbb{R}^d)}\leq C_K\|  u-\bar{u} \|_{H^k(\mathbb{R}^d)}
	\end{align*}
	follows from Proposition A.3 % if one  considers  there function $m(u):=l(\psi(z-q),u)$  for a fixed $q$ (we fix $q$ first and use \cite[Proposition A.3]{XY} 
	considered for  the mapping  $u\mapsto l(\psi_q,u)$  from $H^k(\mathbb{R}^d)$ into $H^k(\mathbb{R}^d)$.
	
 The estimate  for the variation in $v$ also follows from Proposition A.3  for fixed $q$ and $u$.	This concludes the proof the first assertion in part (1) of Proposition \ref{prop3.4.3}. 
	
	Using the Lipschitz property and  the property  $N(\phi_q,0)=0$ we conclude that
	\begin{equation*}
	\|N(\phi_q,v)\|_{H^k(\mathbb{R}^d)}=\|N(\phi_q,v)-N(\phi_q,0)\|_{H^k(\mathbb{R}^d)}\leq C_K\|v\|_{H^k(\mathbb{R}^d)}
	\end{equation*}
	on any set of the form $\{ (v,q):\|q\|_{H^k(\mathbb{R}^{d-1})}+\|v\|_{H^k(\mathbb{R}^d)}\leq K \}$ as required.
	
	The proof of part (2) follows from part (1) since $H^k(\mathbb{R}^d)$ is an algebra, see \eqref{eq2.3.3}, for instance, the estimate of the variation in $q$ is\
	\begin{align*}
	\|N(\phi_q,v)v-N(\phi_{\bar{q}},v)v\|_{H^k(\mathbb{R}^d)}&=\| \big(\tilde{l}(q,u,v)-\tilde{l}(\bar{q},u,v) \big)v\|_{H^k(\mathbb{R}^d)}\\
	&\leq \|\tilde{l}(q,u,v)-\tilde{l}(\bar{q},u,v)\|_{H^k(\mathbb{R}^d)}\|v\|_{H^k(\mathbb{R}^d)}\\
	&\leq C_K\|q-\bar{q}\|_{H^k(\mathbb{R}^{d-1})}\|v\|_{H^k(\mathbb{R}^d)}.
	\end{align*}
	The estimate for variation in $v$ follows from  Proposition A.3 where we  fix $q$ and consider the map $v\rightarrow l(\psi_q, v)v$.
\end{proof}

\begin{proposition}\label{prop3.4.5}
	Assume   Hypotheses \ref{hypo3.2.1} and  \ref{hypo3.2.2}  and let $k\geq \frac{d+1}{2}$.%, let ${H^k(\mathbb{R}^d)}=H^k(\bbR^d)$, and ${H^k_{\alpha}(\mathbb{R}^d)}$ and $\cE$ be defined accordingly. We consider $N(\phi_q,v)$ defined as in \eqref{eq3.2.12}.
	\begin{itemize}
		\item[(1)] If $v\in\mathcal{H}^{n}$, then $N(\phi_q,v)v\in {H^k_{\alpha}(\mathbb{R}^d)}^n$, and for any ball of radius $K$ centered at $(0,0)$ in $\mathcal{H}^{n}\times H^k(\mathbb{R}^{d-1})$ there is a constant $C_K>0$ depending on $K$ such that for any $(v,q)$ in the ball one has
		$$\|N(\phi_q,v)v\|_{H^k_{\alpha}(\bbR^d)}\leq C_K\|v\|_{H^k(\bbR^d)} \|v\|_{H^k_{\alpha}(\bbR^d)}.$$
		\item[(2)] Formula $(v,q)\mapsto N(\phi_q,v)$ defines a mapping from $\mathcal{H}^{n}\times H^k(\mathbb{R}^{d-1})$ to $\mathcal{H}^{n^2}$ that is locally Lipschitz on any bounded neighborhood of $(0,0)$ in $\mathcal{H}^{n}\times H^k(\mathbb{R}^{d-1})$.
		\item[(3)] Formula $(v,q)\mapsto N(\phi_q,v)v$ defines a mapping from $\mathcal{H}^{n}\times H^k(\mathbb{R}^{d-1})$ to $\mathcal{H}^{n}$ that is locally Lipschitz on any bounded neighborhood of $(0,0)$ in $\mathcal{H}^{n}\times H^k(\mathbb{R}^{d-1})$.
	\end{itemize}
\end{proposition}
\begin{proof}
	(1) Using \eqref{eq2.3.3} and Proposition \ref{prop3.4.3} (1) we infer
	\begin{align*}
	\|N(\phi_q,v)v\|_{H^k_{\alpha}(\bbR^d)}
	&=\|N(\phi_q,v)\gamma_{\alpha}v\|_{{H^k(\bbR^d)}}
	\\
	&\leq C\|N(\phi_q,v)\|_{H^k(\bbR^d)}\|\gamma_{\alpha}v\|_{H^k(\bbR^d)} \leq C_K\|v\|_{H^k(\bbR^d)}\|v\|_{H^k_{\alpha}(\bbR^d)}.
	\end{align*}
	To show the local Lipschitz property in part (2) and (3) of the proposition, we note that 
	\begin{equation*}
	\gamma_{\alpha}N(\phi_q,v)=\int_0^1\int_0^1d^2f(\phi_q+s\tau v)s\gamma_{\alpha}v\,d\tau\, ds.
	\end{equation*}
	The Lipschitz assertion then is proved by repeating arguments from Proposition~\ref{prop3.4.3} (1) and (2) for  $\tilde{l}(q,u,\gamma_{\alpha}v)$.
\end{proof}

\begin{proposition}\label{prop3.4.6}
	Assume Hypotheses \ref{hypo3.2.1} and  \ref{hypo3.2.2}, and let  $k\geq [\frac{d+1}{2}]$. The formula $(q,w)\rightarrow (w\cdot w)\phi''_q$ defines a locally Lipschitz mapping from $H^k(\mathbb{R}^{d-1})\times H^k(\mathbb{R}^{d-1})^{d-1}$ to ${\mathcal{H}}^n$ on any bounded set of the form $\{ (q,w):\|q\|_{H^k(\mathbb{R}^{d-1})}+\|w\|_{H^k(\mathbb{R}^{d-1})}\leq K \}$, and the mapping satisfies
	\begin{equation*}
	\|(w\cdot w)\phi''_q\|_{H^k(\bbR^d)}\leq C_K\|w\|^2_{H^k(\mathbb{R}^{d-1})},\quad \text{and}\quad \|\big((w\cdot w)\phi''_q\big)\|_{H^k_{\alpha}(\bbR^d)}\leq C_K\|w\|^2_{H^k(\mathbb{R}^{d-1})}.
	\end{equation*}
\end{proposition}
\begin{proof}
	Recall that,  by Hypothesis~\ref{hypo3.2.3}, $\phi''$ and its derivatives are exponentially decaying to $0$ as $z\rightarrow \pm\infty$, and, by Lemma~\ref{lem3.2.6}, $\gamma_{\alpha}\phi^{(m)}$,  for  $m=1$, ..., $k+1$, is exponentially decaying to $0$ as $z\rightarrow\pm\infty$.
	
	For a fixed $q\in H^k(\mathbb{R}^{d-1})$, to show the local Lipschitz estimate in $w$, we use the Sobolev embedding $H^k(\mathbb{R}^{d})\hookrightarrow L^{\infty}(\mathbb{R}^{d})$, and inequality \eqref{eq2.3.3} for $H^k(\mathbb{R}^{d-1})$ with $k\geq[\frac{d+1}{2}]>\frac{(d-1)+1}{2}$ and observe that, using Lemma \ref{estq} with $q_2=w\cdot w-\bar{w}\cdot \bar{w}$, we have
	\begin{align*}
	\|(w\cdot w-\bar{w}\cdot\bar{w})\phi_q''\|_0&\leq C\|(w-\bar{w})\cdot(w+\bar{w})\|_{H^k(\mathbb{R}^{d-1})}\\
	&\leq C\|w+\bar{w}\|_{H^k(\mathbb{R}^{d-1})}\|w-\bar{w}\|_{H^k(\mathbb{R}^{d-1})}\leq C_K\|w-\bar{w}\|_{H^k(\mathbb{R}^{d-1})},\\
	\|(w\cdot w-\bar{w}\cdot\bar{w})\phi_q''\|_{H^k_{\alpha}(\bbR^d)}&\leq C\|(w-\bar{w})\cdot(w+\bar{w})\|_{H^k(\mathbb{R}^{d-1})}\\
	&\leq C\|w+\bar{w}\|_{H^k(\mathbb{R}^{d-1})}\|w-\bar{w}\|_{H^k(\mathbb{R}^{d-1})}\leq C_K\|w-\bar{w}\|_{H^k(\mathbb{R}^{d-1})},
	\end{align*}
	with some $C_K>0$ that depends on $K$. This completes the proof.
	
	The local Lipschitz estimate in $q$ is proved similarly to Proposition \ref{prop3.4.3} using Lemma \ref{estq}.
\end{proof}
\begin{proposition}\label{prop3.6.17}
	Assume Hypotheses \ref{hypo3.2.1} and  \ref{hypo3.2.2}, and let $k\geq[\frac{d+1}{2}]$. %Let ${H^k(\mathbb{R}^d)}=H^k(\bbR^d)$, and let ${H^k_{\alpha}(\mathbb{R}^d)}$ and $\cE$ be defined accordingly. 
	Formula \eqref{eq3.3.18} for $G(v,q)$, formula \eqref{eqf1} for $F_1(v,q,w)$, and formula \eqref{eqg1} for $F_2(v,q,w)$ define locally Lipschitz mappings from $\mathcal{H}^n\times H^k(\mathbb{R}^{d-1})\times H^k(\mathbb{R}^{d-1})^{d-1}$ to $\mathcal{H}^n$, $\mathcal{H}^n$, and $H^k(\mathbb{R}^{d-1})$ respectively, on any set of the form $\{ (v,q,w):\|v\|_{\mathcal{H}}+\|q\|_{H^k(\mathbb{R}^{d-1})}+\|w\|_{H^k(\mathbb{R}^{d-1})}<K \}$ with the Lipschitz constant denoted by $C_K$. Moreover, if $\|v\|_{\mathcal{H}}+\|q\|_{H^k(\mathbb{R}^{d-1})}+\|w\|_{H^k(\mathbb{R}^{d-1})}<K$, then for some $C_K>0$ depending on $K$ one has:
	\begin{itemize}
		\item[(a)]$
		\|G(v,q)\|_{H^k_{\alpha}(\bbR^d)}\leq C_K(\|v\|_{H^k(\bbR^d)}+\|q\|_{H^k(\mathbb{R}^{d-1})})\|v\|_{H^k_{\alpha}(\bbR^d)},
		$
		\item[(b)] $
		\|F_1(v,q,w)\|_{H^k_{\alpha}(\bbR^d)}\leq C_K\big(\|v\|_{H^k(\bbR^d)}\|v\|_{H^k_{\alpha}(\bbR^d)}+\|q\|_{H^k(\mathbb{R}^{d-1})}\|v\|_{H^k_{\alpha}(\bbR^d)}+\|w\|_{H^k(\mathbb{R}^{d-1})}^2\big),
		$
		\item[(b')]
		$\|F_1(v,q,w)\|_{H^k(\bbR^d)} \leq C_K ( \|v\|_{H^k(\bbR^d)}\|v\|_{H^k_{\alpha}(\bbR^d)}+\|v\|_{H^k(\bbR^d)}\|v_2\|_{H^k(\bbR^d)} \\ {}\qquad \qquad\qquad\qquad \qquad \qquad  +\|q\|_{H^k(\mathbb{R}^{d-1})}\|v\|_{H^k_{\alpha}(\bbR^d)}+\|w\|_{H^k(\mathbb{R}^{d-1})}^2)$,	
		\item[(c)] 
		$
		\|F_2(v,q,w)\|_{H^k(\mathbb{R}^{d-1})}\leq C_K\big(\|v\|_{H^k(\bbR^d)}\|v\|_{H^k_{\alpha}(\bbR^d)}+\|q\|_{H^k(\mathbb{R}^{d-1})}\|v\|_{H^k_{\alpha}(\bbR^d)}+\|w\|_{H^k(\mathbb{R}^{d-1})}^2\big),
		$
		\item[(d)] $
		\|F_2(v,q,w)\|_{L^1(\mathbb{R}^{d-1})}\leq C_K\big(\|v\|_{H^k(\bbR^d)}\|v\|_{H^k_{\alpha}(\bbR^d)}+\|q\|_{H^k(\mathbb{R}^{d-1})}\|v\|_{H^k_{\alpha}(\bbR^d)}+\|w\|_{H^k(\mathbb{R}^{d-1})}^2\big).
		$
	\end{itemize}
\end{proposition}

\begin{proof}
	The local Lipschitz property of $(df(\phi_q)-df(\phi))v$ on $\mathcal{H}^n$ has been proved in Proposition \ref{prop3.3.2}, and the local Lipschitz property of $N(\phi_q,v)v$ on $\mathcal{H}^n\times H^k(\mathbb{R}^{d-1})$ has been proved in Proposition \ref{prop3.4.5}(3). The local Lipschitz properties of these terms imply the locally Lipschitz property of $G(v,q)$ on $\mathcal{H}^n\times H^k(\mathbb{R}^{d-1})$. The proof of (a) follows from \eqref{eq3.3.18} and  Propositions~\ref{prop3.3.2} and \ref{prop3.4.5}.
	
	In formula \eqref{eqg1} for $F_2(v,q,w)$, we first consider the term $K_1(q)(w\cdot w)$. The Lipschitz estimate of the variation in $q$ follows from the triangular inequality, inequality \eqref{eq2.3.3} and Lemma \ref{k1k2} because
	\begin{align*}
	&\|K_1(q)(w\cdot w-K_1(\bar{q})(\bar{w}\cdot\bar{w})\|_{H^k(\mathbb{R}^{d-1})}\\
	&\quad \leq \|K_1(q)-K_1(\bar{q})\|_{H^k(\mathbb{R}^{d-1})}\|w\cdot w\|_{H^k(\mathbb{R}^{d-1})}+\|K_1(\bar{q})\|_{H^k(\mathbb{R}^{d-1})}\|(w\cdot w-\bar{w}\cdot\bar{w})\|_{H^k(\mathbb{R}^{d-1})}\\
	& \quad \leq C_K\left(\|q-\bar{q}\|_{H^k(\mathbb{R}^{d-1})}+\|w-\bar{w}\|_{H^k(\mathbb{R}^{d-1})}\right).
	\end{align*}
	We then consider the term $K_2(q)\pi_{\alpha}G(v,q)$. The Lipschitz estimate of the variation in $q$ follows  from the triangular inequality, \eqref{eq2.3.3}, Lemma \ref{lem3.2.5}, the Lipschitz property of $G(v,q)$ on $\mathcal{H}^n\times H^k(\mathbb{R}^{d-1})$, Lemma \ref{k1k2} and the estimates in part (a) because
	\begin{align*}
	&\|K_2(q)\pi_{\alpha}G(v,q)-K_2(\bar{q})\pi_{\alpha}G(v,\bar{q})\|_{H^k(\mathbb{R}^{d-1})}\\
	&\,\,\,\leq \|K_2(q)\pi_{\alpha}G(v,q)-K_2(q)\pi_{\alpha}G(v,\bar{q})\|_{H^k(\mathbb{R}^{d-1})}+\|K_2(q)\pi_{\alpha}G(v,\bar{q})-K_2(\bar{q})\pi_{\alpha}G(v,\bar{q})\|_{H^k(\mathbb{R}^{d-1})}\\
	&\,\,\, \leq \|K_2(q)\|_{L^{\infty}(\mathbb{R}^{d-1})}\|\pi_{\alpha}\|_{\mathcal{B}(\mathcal{H},H^k(\mathbb{R}^{d-1}))}\|G(v,q)-G(v,\bar{q})\|_{\mathcal{H}}\\
	&\,\,\,\,\,\qquad+\|K_2(q)-K_2(\bar{q})\|_{L^{\infty}(\mathbb{R}^{d-1})}\|\pi_{\alpha}\|_{\mathcal{B}(\mathcal{H},H^k(\mathbb{R}^{d-1}))}\|G(v,\bar{q})\|_{\mathcal{H}}
	\leq  C_K\|q-\bar{q}\|_{H^k(\mathbb{R}^{d-1})}.
	\end{align*}
	The estimates in part (c) follows from \eqref{eq2.3.3}, \eqref{eq3.2.27}, \eqref{eq3.4.30} and \eqref{eq3.2.21}. Indeed, 
	\begin{eqnarray}\label{eq3.3.4}
	&\|K_2(q)\pi_{\alpha}G(v,q)\|_{H^k(\mathbb{R}^{d-1})}&\leq\|K_2(q)\|_{L^{\infty}(\mathbb{R}^{d-1})}\|\pi_{\alpha}G(v,q)\|_{H^k(\mathbb{R}^{d-1})}\\\no
	&\quad \quad \leq C_K\|G(v,q)\|_{H^k_{\alpha}(\bbR^d)}&\leq C_K(\|v\|_{H^k(\bbR^d)}\|v\|_{H^k_{\alpha}(\bbR^d)}+\|v\|_{H^k_{\alpha}(\bbR^d)}\|q\|_{H^k(\mathbb{R}^{d-1})}),\\
	\label{eq3.6.57}
&	\|K_1(q)(w\cdot w)\|_{H^k(\mathbb{R}^{d-1})}&\leq \|K_1(q)\|_{L^{\infty}(\mathbb{R}^{d-1})}\|w\cdot w\|_{H^k(\mathbb{R}^{d-1})}\leq C_K\|w\|_{H^k(\mathbb{R}^{d-1})}^2.
\end{eqnarray}
	Combining estimates in part (a), \eqref{eq3.3.4} and \eqref{eq3.6.57} we have part (c).
	
	For part (b), in formula \eqref{eq3.2.25} of $F_1(v,q, w)$, we already have the Lipschitz property of $G(v,q)$ mapping into $\mathcal{H}^n$ by part (a) and the Lipschitz property of the term $(w\cdot w)\phi''(q)$ mapping into $\mathcal{H}^n$ by Proposition \ref{prop3.4.6}. To prove the Lipschitz estimate for $\phi'_qF_2(v,q,w)$ of the variation in $v$ and $w$, we apply the Lipschitz property of the map $(v,q)\mapsto F_2(v,q,w)$ for a fixed $q$. To prove the Lipschitz property of the variation in $q\mapsto \phi_q'F_2(v,q,w)$, we use the fact that $\phi'$ decays exponentially to $0$ and Lemma \ref{estq} with $q_2=F_2(v,q,w)$. We have the inequality
	\begin{equation*}
	\|F_1(v,q,w)\|_{H^k_{\alpha}(\bbR^d)}
	%=\|\gamma_{\alpha}F_1(v,q,w)\|_{H^k(\bbR^d)}
	\leq \|G(v,q)\|_{H^k_{\alpha}(\bbR^d)}+C\|F_2(v,q,w)\|_{H^k(\mathbb{R}^{d-1})}+\|\gamma_{\alpha}(w\cdot w)\phi''_q\|_{H^k(\mathbb{R}^{d})}.
	\end{equation*}
	We can now use the  estimates on of $G(v,q)$ in part (a) to deal with the first term, then use the estimates for $F_2(v,q,w)$ in part (c), while the estimates of the last term given by Proposition~\ref{prop3.4.6}.
	
	%The next lemma is an analogue of Proposition \ref{prop3.6.17} (b). %with $\|\cdot\|_{H^k_{\alpha}(\bbR^d)}$ in the left-hand side replaced by $\|\cdot\|_{H^k(\bbR^d)}$.
%\begin{lemma}\label{estf1f20}
%	Assume that  $k\geq \frac{d+1}{2}$. For each $K>0$ there is a constant $C_K>0$ such that if $(v,q,w)\in{H^k_{\alpha}(\mathbb{R}^d)}\times H^k(\bbR^{d-1})\times H^k(\bbR^{d-1})$ is the solution of \eqref{sys2} and $$\|v\|_{\mathcal{H}}+\|q\|_{H^k(\mathbb{R}^{d-1})}+\|w\|_{H^k(\mathbb{R}^{d-1})}\leq K,$$ then
%	\begin{align*}\|F_1(v,q,w)\|_{H^k(\bbR^d)} \leq C_K ( \|v\|_{H^k(\bbR^d)}\|v\|_{H^k_{\alpha}(\bbR^d)}&+\|v\|_{H^k(\bbR^d)}\|v_2\|_{H^k(\bbR^d)}	\\ &  +\|q\|_{H^k(\mathbb{R}^{d-1})}\|v\|_{H^k_{\alpha}(\bbR^d)}+\|w\|_{H^k(\mathbb{R}^{d-1})}^2).\end{align*}
%\end{lemma}
%\begin{proof}
%	
Similarly, for part (b'),  we  refer to  Proposition \ref{prop3.4.6}  and then use Lemma \ref{estq} to estimate  the second term in the formula \eqref{eq3.2.25} for $F_1$,
	$$\|\phi'_qF_2(v,q,w)\|_{H^k(\bbR^d)}\leq C(\|v\|_{H^k(\bbR^d)}\|v\|_{H^k_{\alpha}(\bbR^d)}+\|q\|_{H^k(\mathbb{R}^{d-1})}\|v\|_{H^k_{\alpha}(\bbR^d)}+\|w\|_{H^k(\mathbb{R}^{d-1})}^2).$$ 
	Similarly to the  proof of Proposition~\ref{prop3.6.17}(b), the third term in \eqref{eq3.2.25} is estimated as
	$$\|\phi''_q(w\cdot w)\|_{H^k(\bbR^d)}\leq C\|w\|_{H^k(\mathbb{R}^{d-1})}^2.$$
	 We  obtain  an  estimate for $G(v,q)=(df(\phi_q)-df(\phi))v+N(\phi_q,v)v$ by using From Lemma~\ref{lem3.5.1}~(2-3),
	\begin{align*}
	\|G(v,q)\|_{H^k(\bbR^d)}&\leq \|(df(\phi_q)-df(\phi))v\|_{H^k(\bbR^d)}+\|N(\phi_q,v)v\|_{H^k(\bbR^d)}\\
	& \leq K(\|q\|_{H^k(\mathbb{R}^{d-1})}\|v\|_{H^k_{\alpha}(\bbR^d)}+\|v\|_{H^k(\bbR^d)}\|v\|_{H^k_{\alpha}(\bbR^d)}+\|v\|_{H^k(\bbR^d)}\|v_2\|_{H^k(\bbR^d)}).
	\end{align*}
	Adding the above inequalities for the terms of \eqref{eq3.2.25} finishes the proof of  (b'). 
	%We recall notation \eqref{dfnek} and the definition of $T(\delta,\gamma)$ given in Definition \ref{DFNtmax}.
%\end{proof}

	To prove part (d), we use Cauchy-Schwarz inequality,    \eqref{eq3.2.32}, %, the fact that $\|\cdot\|_{L^2}\leq \|\cdot\|_{H^k(\mathbb{R}^{d-1})}$;
	 Proposition \ref{prop3.4.3}(1), and Lemma \ref{lem3.2.5},
	\begin{align*}
	\|\pi_{\alpha}G(v,q)\|_{L^1(\mathbb{R}^{d-1})}\leq &C\|\gamma_{\alpha} G(v,q)\|_{L^1(\mathbb{R}^{d})}\\
	\leq& C(\|\gamma_{\alpha}N(\phi_q,v)v\|_{L^1(\mathbb{R}^{d})}+\|\gamma_{\alpha}(df(\phi_q)-df(\phi))v\|_{L^1(\mathbb{R}^{d})})\\
	\leq &C(\|N(\phi_q,v)\|_{L^2(\mathbb{R}^{d})}\|\gamma_{\alpha}v\|_{L^2(\mathbb{R}^{d})}+\|df(\phi_q)-df(\phi)\|_{L^2(\mathbb{R}^{d})}\|\gamma_{\alpha}v\|_{L^2(\mathbb{R}^{d})})\\
	\leq& C_K(\|v\|_{H^k(\bbR^d)}\|v\|_{H^k_{\alpha}(\bbR^d)}+\|q\|_{H^k(\mathbb{R}^{d-1})}\|v\|_{H^k_{\alpha}(\bbR^d)}).
	\end{align*}
	Similarly, using Cauchy-Schwarz inequality, we infer
	\begin{equation*}
	\|w\cdot w\|_{L^1(\mathbb{R}^{d-1})}\leq \|w\|_{L^2(\mathbb{R}^{d-1})}\|w\|_{L^2(\mathbb{R}^{d-1})}\leq \|w\|_{H^k(\mathbb{R}^{d-1})}^2,
	\end{equation*}
	and thus we have
	\begin{align*}
	\|F_2(v,q,w)\|_{L^1(\mathbb{R}^{d-1})}&\leq C(\|\pi_{\alpha}G(v,q)\|_{L^1(\mathbb{R}^{d-1})}+\|w\cdot w\|_{L^1(\mathbb{R}^{d-1})})\\
	&\leq C_K(\|v\|_{H^k(\bbR^d)}\|v\|_{H^k_{\alpha}(\bbR^d)}+\|q\|_{H^k(\mathbb{R}^{d-1})}\|v\|_{H^k_{\alpha}(\bbR^d)}+\|w\|_{H^k(\mathbb{R}^{d-1})}^2).
	\end{align*}
	This finishes the proof of the required inequalities in part (d).
\end{proof}

We next formulate  lemmae, whose proofs resemble the proofs in \cite[Lemma 8.1, Lemma 8.2 and Lemma 8.3]{GLS2}.  We will use them later to prove boundedness of the components of the solutions in $H^k(\bbR^d)$-norm. 
\begin{lemma}\label{lemgls}
	Assume Hypotheses \ref{hypo3.2.1} and  \ref{hypo3.2.2}, and let   $k\geq \frac{[d+1]}{2}$. Then the entries of the matrix-valued function 
	$\big(  df(\phi)-df(0) \big)\gamma_{\alpha}^{-1}$ belong to $H^k(\mathbb{R})$.
\end{lemma}
\begin{proof}
	This follows from the formula
	$
	  df(\phi)-df(0) =\phi \int_0^1d^2f(s\phi)ds,
$
	where $f(\cdot)$ is a $C^{k+3}$ smooth function by Hypothesis \ref{hypo3.2.1} and from the fact that $\phi \gamma_{\alpha}^{-1}\in H^k(\mathbb{R})$ using Lemma \ref{lem3.2.6}(1).
\end{proof}
We will now use Lemma \ref{estq} to prove an analogue of Proposition \ref{prop3.3.2}(1) with $\|v\|_{H^k(\bbR^d)}$ in the right-hand side replaced by $\|v\|_{H^k_{\alpha}(\bbR^d)}$ and Proposition \ref{prop3.4.5}(1) with $\|\cdot\|_{H^k_{\alpha}(\bbR^d)}$ in the left-hand side replaced by $\|\cdot\|_{H^k(\bbR^d)}$. We recall that $\phi=\phi(z)$ and that the function $(z,y)\mapsto \big( df\big(\phi(z-q(y))\big)-df(0) \big)v(y)$ is in $%{H^k(\mathbb{R}^d)}=
H^k(\bbR^d)$.

\begin{lemma}\label{lem3.5.1}
	Assume   Hypotheses \ref{hypo3.2.1} and  \ref{hypo3.2.2}, and let   $k\geq \frac{d+1}{2}$. For each $k>0$, there is a constant $C_K>0$ such that if  $q\in H^k(\bbR^{d-1})$ and
	 $v\in \cH$
	%$ {H^k_{\alpha}(\mathbb{R}^d)}$ 
	 satisfy $\|v\|_{\cH}+\|q\|_{H^k(\mathbb{R}^{d-1})}\leq K$, then
	\begin{itemize}
		\item[(1)] $\|(df(\phi)-df(0))v\|_{H^k(\bbR^d)} \leq C_K\|v\|_{H^k_{\alpha}(\bbR^d)}$;
		\item[(2)] $\|(df(\phi_q)-df(\phi))v\|_{H^k(\bbR^d)} \leq C_K\|q\|_{H^k(\mathbb{R}^{d-1})}\|v\|_{H^k_{\alpha}(\bbR^d)}$.
		\item[(3)] For $(v,q)$ in a bounded neighborhood of $(0,0)$ in $\mathcal{H}^n\times H^k(\mathbb{R}^{d-1})$, and $v=(v_1,v_2)^T$ with $v_i\in\mathcal{H}^{n_i}$, $i=1,2$, one has for  $N(\cdot,\cdot)$  defined in equation \eqref{eq3.2.12}, 
		 $$\|N(\phi_q,v)v\|_{H^k(\bbR^d)} \leq C_K\|v\|_{H^k(\bbR^d)} (\|v\|_{H^k_{\alpha}(\bbR^d)}+\|v_2\|_{H^k(\bbR^d)}).$$
	\end{itemize}
\end{lemma}

\begin{proof}
	Lemma \ref{lemgls} and \eqref{eq2.3.3} yield (1) since
	\begin{align*}
	\|(df(\phi)-df(0))v\|_{H^k(\bbR^d)}\leq \|\big(  df(\phi)-df(0) \big)\gamma_{\alpha}^{-1}\|_{H^k(\bbR)}\|\gamma_{\alpha} v\|_{H^k(\bbR^d)}
	\leq C_K \|v\|_{H^k_{\alpha}(\bbR^d)}.
	\end{align*}
	To prove (2), we write, as in \eqref{eq2.6.54},
	\begin{align}\label{p60.28}
	(df(\phi_q)-df(\phi))v&%=-\int_{0}^{1}d^2f(\phi(z-sq(y)))(\phi'(z-sq(y))q,v)\ ds\\\no
	%&
	=-\int_{0}^{1}d^2f(\phi(z-sq(y)))(\gamma_{\alpha}^{-1}\phi'(z-sq(y))q,\gamma_{\alpha}v)\ ds.
	\end{align}
	We  next use an argument similar to the one in  Lemma \ref{estq} to prove that
	\begin{equation}\label{claimp60}
	\|(df(\phi_q)-df(\phi))v\|_{H^k(\bbR^d)} \leq C_K\|q\|_{H^k(\mathbb{R}^{d-1})}\|v\|_{H^k_{\alpha}(\bbR^d)}.
	\end{equation}
	Indeed,  the main steps in the proof of \eqref{claimp60} are as follows.
	
	We consider \eqref{p60.28} component-wise.
	 %to the components in the right-hand side of the vector equation \eqref{p60.28}, 
	The proof of  \eqref{claimp60} is  then reduced to proof of  the inequality
	\begin{equation}\label{ins1.1}
	\|\sigma\|_{H^k(\bbR^d)}\leq C\|q\|_{H^k(\bbR^{d-1})},
	\end{equation}
	where $\sigma(z,y)=\gamma^{-1}_{\alpha}(z)\psi'(z-q(y))q(y)$, $x=(z,y)\in\bbR^d$, and $\psi$  as required in Lemma \ref{estq} has exponentially decaying derivatives. Indeed, as soon as $\eqref{ins1.1}$ is proved, the inequality $$\|\phi'(\cdot-sq(\cdot))q(\cdot)v(\cdot)\|_{H^k(\bbR^d)}\leq \|\sigma\|_{H^k(\bbR^d)}\|v\|_{H^k_{\alpha}(\bbR^d)}$$ yields \eqref{claimp60} from  \eqref{p60.28}.

%	\textcolor{red}{To prove  \eqref{ins1.1}, we denote $m(x)=\gamma_{\alpha}(z-q(y))$ so that $\sigma(x)=m(y)(\gamma^{-1}_{\alpha}\psi')(z-q(y))q(y)$.  We note that $\psi_1(z)=\gamma^{-1}_{\alpha}(z)\psi'(z)$ exponentially decays at $z\rightarrow\pm\infty$. Using $q\in H^k(\bbR^{d-1})\hookrightarrow L^{\infty}(\bbR^{d-1})$, and formula \eqref{eq3.1.7} for $\gamma_{\alpha}(z)$, we conclude that  $m(y)=e^{-\alpha_- q(y)}$ for $z\leq -r$ and $m(y)=e^{-\alpha_+q(y)}$ for $z\geq r$ for some large $r>0$ uniformly in $y\in\bbR^{d-1}$; moreover, $m\in L^{\infty}(\bbR^d)$. Similarly to the calculation in \eqref{eq3.55.2}, the $L^2(\bbR^d)$-norm of $\sigma$ can be estimated as }
	
	To prove  \eqref{ins1.1}, we
	denote $m(z,y)=\gamma_{\alpha}(z-q(y))\gamma_{\alpha}^{-1}(z)$ so that 
	$$\sigma(z,y)=m(z,y) (\gamma^{-1}_{\alpha}\psi')(z-q(y))q(y).$$ We note that $\gamma^{-1}_{\alpha}(z)\psi'(z)$ exponentially decays at $z\rightarrow\pm\infty$. Using $q\in H^k(\bbR^{d-1})\hookrightarrow L^{\infty}(\bbR^{d-1})$ and formula \eqref{eq3.1.7} for $\gamma_{\alpha}(z)$, 
	we conclude that 
  $m(z,y)=e^{-\alpha_- q(y)}$ for $z\leq -r$ and $m(z,y)=e^{-\alpha_+q(y)}$ for $z\geq r$ for some large $r>0$ uniformly in $y\in\bbR^{d-1}$; moreover,
	 $\gamma_{\alpha}(-q(\cdot)) \in {L^{\infty}(\bbR^{d-1})}$, with the norm bounded by a constant that depends on $K$. 
	 Similarly to the calculation in \eqref{eq3.55.2}, the $L^2(\bbR^d)$-norm of $\sigma$ can be estimated as 
	\begin{equation}\label{ins1.2}
	\|\sigma\|_{L^2(\bbR^d)}\leq \| m  \|_{L^{\infty}(\bbR^{d-1})}, \quad
	 \|(\gamma_{\alpha}^{-1}\psi')(\cdot-q(\cdot))q(\cdot)\|_{L^2(\bbR^{d-1})}\leq C_K\|q\|_{H^k(\bbR^{d-1})}.
	\end{equation}
	We now show how to estimate the $L^2(\bbR^d)$-norm of the derivatives of $\sigma$. The $z-$derivative, 
	\begin{equation*}
	\frac{\partial\sigma}{\partial z}=(\gamma_{\alpha}^{-1})'(z)\psi'(z-q(y))q(y)+\gamma_{\alpha}^{-1}(z)\psi''(z-q(y))q(y),
	\end{equation*}
	is the sum of two terms that can handled similarly to \eqref{ins1.2}. Taking derivatives with respect to $x_j$, $j=2,...,d$, yields, as in \eqref{eq3.55.4},
	\begin{equation}\label{ins1.3}
	\frac{\partial\sigma}{\partial x_j}=\gamma_{\alpha}^{-1}\psi''(z-q(y))\frac{\partial q}{\partial x_j}q(y)+\gamma_{\alpha}^{-1}(z)\psi'(z-q(y))\frac{\partial q}{\partial x_j}.
	\end{equation}
	The $L^2(\bbR^d)$-norm of the first term can be estimated as in \eqref{eq3.55.3} and \eqref{ins1.2}, that is, by a calculation similar to \eqref{eq3.55.2}, we have
	\begin{align*}
	\|\gamma_{\alpha}^{-1}(\cdot)\psi''(\cdot-q(\cdot))\frac{\partial q}{\partial x_j}q\|_{L^2(\bbR^d)}&\leq \|\m\|_{L^{\infty}(\bbR^{d-1})}  \|q\|_{L^{\infty}(\bbR^{d-1})} \|(\gamma_{\alpha}^{-1}\phi'')(\cdot-q(\cdot))\frac{\partial q}{\partial x_j}\|_{L^2(\bbR^d)}%\leq C_K\|\frac{\partial q}{\partial x_j}q\|_{L^2(\bbR^{d-1})} 
	\\&\leq 
	\|\gamma_{\alpha}(-q(\cdot))\|_{L^{\infty}(\bbR^{d-1})} \|q\|_{L^{\infty}(\bbR^{d-1})} \|(\gamma_{\alpha}^{-1}\phi'')(\cdot-q(\cdot))\frac{\partial q}{\partial x_j}\|_{L^2(\bbR^d)}	\\&
	%\leq C_K\|\frac{\partial q}{\partial x_j}\|_{L^2(\bbR^{d-1})}
	\leq C_K\|q\|_{H^k(\bbR^{d-1})},%\|q\|_{H^k(\bbR^{d-1})}
	\end{align*}
where $C_K$ is a $K$-dependent constant  different from the constant in \eqref{ins1.2}. 
A similar calculation works for the second term in \eqref{ins1.3}. This proves assertion \eqref{ins1.1} for $k=1$. We conclude  the proof of assertion (2) by pointing out that higher order derivatives are handled as described in the proof of Proposition A.3.  
	
	To prove (3), we  recall the following  representation of the nonlinearity $v=(v_1,v_2)^T\mapsto N(\phi_q,v)v$ borrowed from the proof of \cite[Lemma 8.3]{GLS2},
	\begin{equation*}
	N(\phi_q,v)v=I_1(v)+I_2(v)+I_3(v)+I_4(v)+I_5(v),
	\end{equation*}
	where $\phi_q=(\phi_1(z-q),\phi_2(z-q))^T=(\phi_{1,q},\phi_{2,q})^T$, $v=(v_1,v_2)^T$,
	\begin{align*}%\label{5int}
	I_1(v)&=\int_0^1\left( \partial_{u_1}r(\phi_q+tv)-\partial_{u_1}r(\phi_q) \right)v_1\phi_{2,q}dt,\\\no
	I_2(v)&=\int_{0}^{1}\left(\partial_{u_1}r(\phi_q+tv)v_1\right)tv_2dt, \qquad 
	I_3(v)=\int_0^1\left( \partial_{u_2}r(\phi_q+tv)-\partial_{u_2}r(\phi_q) \right)v_2\phi_{2,q}dt,\\\no
	I_4(v)&=\int_0^1\left( \partial_{u_2}r(\phi_q+tv) v_2\right)tv_2dt,\qquad
	I_5(v)=\int_0^1\left( r(\phi_q+tv)-r(\phi_q) \right)v_2dt,
	\end{align*}
	and the $n\times n$ matrix-valued $C^{k}$ function $r=r(u_1,u_2)$ is given by 
	\begin{equation*}
	r(u_1,u_2)=\int_0^1\partial_{u_2}f(u_1,su_2)ds.
	\end{equation*}
	The proof of the required estimates for each $I_j$, $j=1$, $2$,..., $5$,  is similar to the proof of assertion (2) above and uses Lemma~\ref{lem3.2.6}. For instance, for $j=1$, passing in the integral to the third derivative of $f$ (which is a $C^k$-bounded function by Hypothesis \ref{hypo3.2.1}), we  reduce the problem to obtaining  an estimate for  % $H^k(\bbR^d)$-norm of the function 
	$\|vv_1\phi_{2,q}\|_{H^k(\bbR^d)}$. If we write $v_1\phi_{2,q}=(\gamma_{\alpha}v_1)(\gamma_{\alpha}^{-1}\phi_{2,q})$ and use that $H^k(\bbR^d)$ is an algebra, then,  in order to prove that 
	\begin{equation}\label{esti1}
	\|I_1(v)\|_{H^k(\bbR^d)}\leq C\|v\|_{H^k(\bbR^d)}\|v_1\|_{H^k_{\alpha}(\bbR^d)},
	\end{equation}
	it suffices  to show that the $H^k(\bbR^d)$-norm of $\sigma(z,y)=\gamma_{\alpha}^{-1}(z)\phi_2(z-q(y))w(x)$ with $w=\gamma_{\alpha}v_1$ is bounded by $C\|w\|_{H^k(\bbR^d)}$. This follows because $q\in H^k(\bbR^{d-1}) \hookrightarrow L^{\infty}(\bbR^{d-1})$ yields the existence of a large $r>0$ such that, uniformly for $y\in\bbR^{d-1}$, we have
	\begin{equation*}
	|\gamma_{\alpha}^{-1}(z)\phi_2(z-q(y))|\leq \begin{cases}
	K e ^{-\alpha_- z}e^{-\omega_-(z-q(y))},\quad z\leq -r,\\
	K e ^{-\alpha_+ z}(|\phi_2|+e^{-\omega_+(z-q(y))}),\quad z\leq -r.
	\end{cases}
	\end{equation*}
	Using $e^q\in L^{\infty}(\bbR^{d-1})$ and Hypothesis \ref{hypo3.2.3} we conclude that $\gamma_{\alpha}^{-1}(\cdot)\phi_2(\cdot-q(\cdot))$ is bounded. A similar argument, as in the proof of (2) above, applies for the derivatives of $\sigma$. This completes the proof of \eqref{esti1}. For $j=2$, ..., $5$, the estimates $\|I_j(v)\|_{H^k(\bbR^d)}\leq C\|v\|_{H^k(\bbR^d)} \|v_2\|_{H^k(\bbR^d)}$ are straightforward since each integral has a factor $v_2$ and both derivatives of $f$ and $\phi_{2,q}$ are $k$-smooth with bounded derivatives. Combining the estimates for $j=1$, ..., $k$, yields assertion (3). \newline
\end{proof}

\section{Nonlinear stability}\lb{subs3.2.4}

\subsection{Local in time existence and bounds.}
In this section we  analyze  the system \eqref{sys2}. %We recall formula \eqref{eqg1} and \eqref{eq3.2.25} for the nonlinear term $F_1, \ F_2$ and the notation 
We denote
\begin{eqnarray}\mathcal{X}= \ran\mathcal{Q}_{\mathcal{H}}\times H^k(\mathbb{R}^{d-1})\times H^k(\mathbb{R}^{d-1})^{d-1}. \label{x}\end{eqnarray}
%Recall that in \eqref{x} we denoted $$\mathcal{X}=\ran\mathcal{Q}_{\mathcal{H}}\times H^k(\mathbb{R}^{d-1})\times H^k(\mathbb{R}^{d-1})^{d-1}.$$
We also will assume as before that 
$k\geq\Big[\frac{d+1}{2}\Big].$
 Let $S_{\cL_{{\mathcal{H}}}}(t)=e^{t\cL_{{\mathcal{H}}}}$ be the semigroup generated by the operator $\cL_{{\mathcal{H}}}$ (see Definition~\ref{op} (6)). %  in associated with linearization \eqref{eq3.1.2} about the front.  
Let  $(v^0,q^0,w^0)$ be the initial perturbation to the front.  Since $S_{\Delta_y}(t)\nabla F_2=\nabla_y S_{\Delta_y}(t)F_2$,
 the variation of constants formula implies that the mild solution to \eqref{sys2} on  $\mathcal{X}$
%$\ran\mathcal{Q}_{\mathcal{H}}\times H^k(\mathbb{R}^{d-1})\times H^k(\mathbb{R}^{d-1})$ 
satisfy the equations
\begin{align}\label{eq3.4.4}
&v(t)=S_{\mathcal{L}_{\mathcal{H}}}(t)v^0+\int_{0}^{t}S_{\mathcal{L}_{\mathcal{H}}}(t-s)F_1(v(s),q(s),w(s))ds,\no\\
&q(t)=S_{\Delta_y}(t)q^0+\int_{0}^{t}S_{\Delta_y}(t-s)F_2(v(s),q(s),w(s))ds,\no\\
&w(t)=S_{\Delta_y}(t)w^0+\int_{0}^{t}\nabla_{y}S_{\Delta_y}(t-s)F_2(v(s),q(s),w(s))ds.
\end{align}
%where    $(v^0,q^0,w^0)$ is the initial perturbation to the front. Here we  used that $S_{\Delta_y}(t)\nabla F_2=\nabla S_{\Delta_y}(t)F_2$.
%Note that we used the fact that $S_{\Delta_y}(t)\nabla F_2=\nabla S_{\Delta_y}(t)F_2$ in \eqref{eq3.4.4}.

Next we formulate   a statement  that shows  the existence and uniqueness of the mild solutions of  \eqref{eq3.4.4}.% in the following proposition,   the proof  of which is the same as  the proof of \cite[Lemma 3.4]{Kapitula2}, but here we provide more detailed version of it. %, which is an application of  the semigroup theory (see \cite[Theorem6.1.4]{Pazy}). 
\begin{proposition}\label{prop3.8.21}
	%Assume $k\geq[\frac{d+1}{2}]$. 
	For any initial data $(v^0,q^0,w^0)\in\mathcal{X}$ %\ran\cQ_{{\mathcal{H}}}\times H^k(\bbR^{d-1})\times H^k(\bbR^{d-1}),$$
	system \eqref{sys2} has a unique mild solution (that is, a solution of \eqref{eq3.4.4}) 
	$(v(t),q(t),w(t))\in \mathcal{X}$ %\ran\mathcal{Q}_{\mathcal{H}}\times H^k(\mathbb{R}^{d-1})\times H^k(\mathbb{R}^{d-1})^{d-1}$$ 
	in the maximal interval $0\leq t<t_{\max}$, where $0<t_{\max}\leq\infty$.
\end{proposition}

%\begin{proof}
The proof can be found in  \cite[Lemma 3.4]{Kapitula2}. We just mention that the proof only uses the fact that $\mathcal{L}_{\mathcal{H}}$ generates a strongly continuous semigroup, even though we know that $\mathcal{L}_{\mathcal{H}}$ generates a bounded strongly continuous semigroup. Indeed, since the operator $\mathcal{L}_{\mathcal{H}}$ generates a strongly continuous semigroup and the nonlinearities $F_1$ and $F_2$ are locally Lipschitz with Lipschitz constant $C_K$ 
on the set  $\{ (v,q,w):\|v\|_{\mathcal{H}}+\|q\|_{H^k(\mathbb{R}^{d-1})}+\|w\|_{H^k(\mathbb{R}^{d-1})}<K \}$, the  estimate from Lemma \ref{lem3.4.4} ({c}), which is integrable at $t=0$, yields the statement of Proposition~\ref{prop3.8.21}.
%QQQQQQQQQQ
%	For any $t_0\geq0$, let $V^0\in \mathcal{X}$ %:= \ran\mathcal{Q}_{\mathcal{H}}\times H^k(\mathbb{R}^{d-1})\times H^k(\mathbb{R}^{d-1})^{d-1}$, 
	%and define a mapping $\mathcal{J}:C([t_0, t_1];\mathcal{X})\rightarrow C([t_0, t_1];\mathcal{X})$,  where $t_1$ will be specified later,  by the formula
%	\begin{equation}\label{eq4.3}(\mathcal{J}V)(t)=T(t-t_0)V^0+\int\limits_{t_0}^{t}BT(t-s)F(V(s))ds \mbox{ ,  } t_0\leq t\leq t_1,\end{equation}
	%where
% $$V(t)=(v(t),q(t),w(t))\in\mathcal{X},\,  T(t)=S_{\mathcal{L}_{\mathcal{H}}}(t)\oplus S_{\Delta_y}(t)\oplus S_{\Delta_y}(t), \,  B=I\oplus I\oplus \nabla_y,$$
%	with $\ran T(t)\subset \dom(B)$ and \begin{equation}\label{dfnofF}
%	F(V(s))=(F_1(v,q,w),F_2(v,q,w),F_2(v,q,w)).
%	\end{equation}
%	Here  $F_1$ and $F_2$ defined as in \eqref{sys2}, cf., \eqref{eq3.3.18}, \eqref{eqf1} and \eqref{eqg1}.
%The operator $\mathcal{L}_{\mathcal{H}}$ generates a strongly continuous semigroup on $\mathcal{H}:= H^k(\mathbb{R}^d)\cap{H^k_{\alpha}(\mathbb{R}^d)}$. By Lemma \ref{lem3.4.4} (a)  the semigroup $S_{\Delta_y}$ generated by $\Delta_y$ is bounded on $H^k(\mathbb{R}^{d-1})$. 
%	Thus,  \cite[Theorem 6.1.4]{Pazy} yields both the existence and the uniqueness of $V(t)$ and the Lipschitz continuity of the map $V^0\rightarrow V(t)$ for $0\leq t<t_{\max}$. 
%\end{proof}
%We then consider
%QQQQQQQ 

For  \eqref{sys2} on $\mathcal{X}$  from \eqref{x} 
%$\ran\mathcal{Q}_{\mathcal{H}}\times H^k(\mathbb{R}^{d-1})\times  H^k(\mathbb{R}^{d-1})$ 
 we combine Proposition~\ref{prop3.8.21} and \cite[Theorem 64.2]{Sell} to obtain the next lemma.
\begin{lemma}\label{lem3.4.7}
	For each $\delta>0$, if $0<\gamma<\delta$, there exists $T$  ($0<T\le\infty$) such that the following is true: if $(v^0,q^0,w^0)\in \mathcal{X}$ %\ran\mathcal{Q}_{\mathcal{H}}\times H^{k}(\mathbb{R}^{d-1})\times H^{k}(\mathbb{R}^{d-1})$
	 satisfies 
	\begin{equation}\label{eq3.4.6}
	\|v^0\|_{\mathcal{H}}+\|q^0\|_{H^k(\mathbb{R}^{d-1})}+\|w^0\|_{H^k(\mathbb{R}^{d-1})}\leq \gamma
	\end{equation}
	and $0\leq t< T$, then the solution $(v(t),q(t),w(t))\in  \mathcal{X}$ %\ran\mathcal{Q}_{\mathcal{H}}\times H^{k}(\mathbb{R}^{d-1})\times H^{k}(\mathbb{R}^{d-1})$ 
	of \eqref{eq3.4.4} with the initial data $(v^0,q^0,w^0)$ is defined and satisfies
	\begin{equation}\label{eq3.4.7}
	\|v(t)\|_{\mathcal{H}}+\|q(t)\|_{H^k(\mathbb{R}^{d-1})}+\|w(t)\|_{H^k(\mathbb{R}^{d-1})}\leq \delta.
	\end{equation} 
\end{lemma}
\begin{definition}\label{DFNtmax}
	Let $T(\delta,\gamma)$ denote the supremum of all $T$ such that \eqref{eq3.4.7} holds for all $0\leq t<T$ whenever \eqref{eq3.4.6} is satisfied.
\end{definition}

Having established the local in time existence of the solution of \eqref{sys2}, we  show  next  the algebraic decay and boundedness of the solution. % in Propositions~\ref{pro3.4.13} and \ref {prop3.5.2}. 
%We begin by proving that it suffices to verify the estimates in Proposition \ref{pro3.4.13} and \ref{prop3.5.2} only for large values of $t$, that is,  to  show  that, given an initial condition $V^0=(v^0,q^0,w^0)$, in a small time period the mild solution $V=(v,q,w)(t,v^0,q^0,w^0)\in\mathcal{X}%:=\ran\mathcal{Q}_{\mathcal{H}}\times H^k(\mathbb{R}^d-1)\times H^k(\mathbb{R}^{d-1})^{d-1}$
 %satisfies estimate \eqref{estVt} in the next corollary.
 %Given $K>0$, we continue to denote by $C_K$  the Lipschitz constant from Proposition \ref{prop3.6.17} for 
%$F(V(s))=(F_1(v,q,w),F_2(v,q,w),F_2(v,q,w))	$ on the ball $\{ V(t)\in\mathcal{X}:\|V(t)\|_{\mathcal{X}}%:=\|v(t)\|_{\mathcal{H}}+\|q(t)\|_{H^k(\mathbb{R}^{d-1})}+\|w(t)\|_{H^k(\mathbb{R}^{d-1})}\leq K  \}.
%$
\begin{corollary}\label{tleq1}
	%Assume Hypothesis \ref{MH}. 
	For any $K>0$, there exists %a small enough 
	$\delta_0<K$ such that for any $\gamma$ and  $\delta$ satisfying $0<\gamma<\delta<\delta_0$, the mild solution $V(t)=(v(t),q(t),w(t))$ of \eqref{sys2} satisfying $\|V(t)\|_{\cX}\leq \delta$ on the interval $t\in[0,T(\delta,\gamma))$ is continuous with respect to the initial data $V^0=(v^0,q^0,w^0)$ satisfying $\|V^0\|_{\mathcal{X}}\le\gamma$. Moreover, %if $T(\delta,\gamma)\leq 1$, then 
	\begin{equation}\label{estVt}
	\|V(t)\|_{\mathcal{X}}\leq C(K) \|V^0\|_{\mathcal{X}}\text{ for all  } t\in[0,\min\{1,T(\delta,\gamma)\}],
	\end{equation}
	where $C(K)$ is a constant that depends on $K$ but is independent of $\delta$ and $\gamma$.
\end{corollary}
\begin{proof} Since the estimate in Lemma \ref{lem3.4.4}  is integrable at $t=0$, the continuity with respect to initial data is a simple modification of the standard argument, see \cite[Theorem 64.2]{Sell},%, as in the proof of Proposition \ref{prop3.8.21} above.
	
%	To prove the inequality  \eqref{estVt}, 
	%we
	%recall that % in Proposition~\ref{prop3.8.21}
	 %we let $T(t)=S_{\mathcal{L}_{\mathcal{H}}}(t)\oplus S_{\Delta_y}(t)\oplus S_{\Delta_y}(t)$,   $B=I\oplus I\oplus \nabla_y$, and  $F(V(s))=(F_1(v,q,w),F_2(v,q,w),F_2(v,q,w))$. 
	 
	 By Lemmas~\ref{lbeta} and  \ref{lem3.4.4} the semigroup  $\{T(t)\}_{t\geq 0} =\{S_{\mathcal{L}_{\mathcal{H}}}(t)\oplus S_{\Delta_y}(t)\oplus S_{\Delta_y}(t)\}_{t\geq 0}$ is bounded. We  define $M=\max\{\sup\{ \|T(t)\|_{\cB(\cX)}: t\ge0 \},C\}$, where $C$ is the constant from Lemma \ref{lem3.4.4}(c). The variation of constant formula \eqref{eq3.4.4} and Proposition \ref{prop3.6.17} and Lemma \ref{lem3.4.7}  together with assumption $0\leq t<\min\{1,T(\delta,\gamma)\}\leq 1$, for all $t\in[0,\min\{1,T(\delta,\gamma)\}  )$,  yield
$$	\|V(t)\|_{\mathcal{X}}\leq M \|V^0\|_{\mathcal{X}}+M C_K\delta\int_0^t(t-s)^{-1/2}\|V(s)\|_{\mathcal{X}}ds%\\
	%&%\leq M \|V^0\|_{\mathcal{X}} + 2MC_K\delta \sup\limits_{0\leq t<T(\delta,\gamma)}\|V(t)\|_{\mathcal{X}}t^{1/2} %\,\, (\text{noting that }\, 0\leq t<T(\delta,\gamma)\leq 1)
	\leq M \|V^0\|_{\mathcal{X}} + 2MC_K\delta \sup\limits_{  T(\delta,\gamma)\geq 0}\|V(t)\|_{\mathcal{X}}.
$$
	 We then choose $\delta_0\leq \min\{ K, 1/4MC_K\}$ and  conclude that for any $0<\delta<\delta_0$ and $0<\gamma<\delta$, then $\|V(t)\|_{\mathcal{X}}\leq C(K) \|V^0\|_{\mathcal{X}}$ for some $C(K)$ depending on $K$ for all $t\in[0,\min\{1,T(\delta,\gamma)\} )$.
\end{proof}

\subsection{The algebraic decay of solutions in weighted norm}
In this subsection we show that the weighted norm of the solution $v(t)=(v_1,v_2)$ of \eqref{sys2} decays algebraically as $t\rightarrow\infty$, the ${H^k(\mathbb{R}^{d})}$-norm of  $v_2(t)$ also decays algebraically as $t\rightarrow \infty$, while the ${H^k(\mathbb{R}^{d})}$-norm of  $v_1(t)$  is bounded provided the initial value of the solution is sufficiently small.
For the initial data  
 $(v^0,q^0,w^0)\in \ran \cQ_{{\mathcal{H}}}\times H^k(\bbR^{d-1})\times H^k(\bbR^{d-1})$, we denote the size of the initial values  by
\begin{equation}\label{dfnek}
E_k=\|v^0\|_{\mathcal{H}}+\|q^0\|_{H^{k+1}(\mathbb{R}^{d-1})}+\|q^0\|_{W^{1,1}(\mathbb{R}^{d-1})}.
\end{equation}
We assume $q^0\in H^{k+1}(\mathbb{R}^{d-1}) \cap W^{1,1}(\mathbb{R}^{d-1})$ so that  when  \eqref{sys2} has a mild solution, $w(t)$ satisfies $w(t)=\nabla_y q(t)$ and $w(t)\in H^k(\mathbb{R}^{d-1})^{d-1}\cap L^1(\mathbb{R}^{d-1})^{d-1}$, thus \eqref{dfnek} contains the norm of   $\|w^0\|_{L^{1}(\mathbb{R}^{d-1})}$ in the last term. 

The following estimates  are proved by  direct computation  in \cite{Xin}. 
\begin{lemma}\label{lem3.4.12}
	Suppose $a,b,c>0$, then
	\begin{itemize}
		\item[(1)] $\int_0^{t/2} (1+t-s)^{-b}(1+s)^{-c}ds\leq (1+t)^{-a}$, if $a\leq b$, $a\leq b+c-1$, $c\neq 1$; or if $a<b$, $c=1$;
		\item[(2)] $\int_{t/2}^t (1+t-s)^{-b}(1+s)^{-c}ds\leq (1+t)^{-a}$, if $a\leq c$, $a\leq b+c-1$, $b\neq 1$; or if $a<c$, $b=1$;
		\item[(3)] $\int_{0}^{t}e^{-b(t-s)}(1+s)^{-c}ds\leq (1+t)^{-c}$.
	\end{itemize}
\end{lemma}
We now show that the 
weighted norm of $v(t)$ and the $H^k(\bbR^{d-1})$-norms of $q(t)$ and $w(t)$, in fact, decay to zero algebraically as long as $t$ grows but the ${\mathcal{H}}$-norm of $v(t)$ and the ${H^k(\bbR^{d-1})}$ norms of $q(t)$ and $w(t)$ remain small. %Recall notation \eqref{dfnek} and the definition of $T(\delta,\gamma)$ given in Definition~\ref{DFNtmax}.

\begin{proposition}\label{pro3.4.13}
	Assume  Hypotheses \ref{hypo3.2.1}, \ref{hypo3.2.2} and  \ref{hypo3.2.7}, and let $k\geq[\frac{d+1}{2}]$. Choose $\nu>0$ as in Lemma \ref{lem3.4.9}. There exist $\delta_1>0$  and $C_1>0$ such that for every $\delta\in(0,\delta_1)$ and every $\gamma$ with $0<\gamma<\delta$, if $E_k<\gamma$, then the solution $(v(t),q(t),w(t))$ of \eqref{sys2} with the initial data $(v^0,q^0,w^0)$, for $t\in[0,T(\delta,\gamma))$ satisfies the estimates
	\begin{align*}
	&\|v(t)\|_{H^k_{\alpha}(\bbR^d)}\leq C_1(1+t)^{-\frac{d+1}{2}}E_k,\\
	&\|q(t)\|_{H^k(\mathbb{R}^{d-1})}\leq C_1(1+t)^{-\frac{d-1}{4}}E_k,\\
	&\|w(t)\|_{H^k(\mathbb{R}^{d-1})}\leq C_1(1+t)^{-\frac{d+1}{4}}E_k.
	\end{align*}
\end{proposition}

\begin{proof}
	In Corollary~\ref{tleq1} we have discussed  the solution of \eqref{sys2} in a  time period  $[0,\min\{1,T(\delta,\gamma)\})$, therefore, without loss of generality, we may assume that $T(\delta,\gamma)>1$.% and $t>1$ is large in this proof.
	
	We recall that $\ran\mathcal{Q}_{\mathcal{H}}=\ran\mathcal{L}_{\alpha}\cap \mathcal{H}^n$, thus for $v\in\ran\mathcal{Q}_{\mathcal{H}}$ we can replace $\mathcal{L}_{\mathcal{H}}$ by $\mathcal{L}_{\alpha}$ in  \eqref{eq3.4.4}.
	Applying the semigroup estimates  from Lemmae~\ref{lem3.4.4} and  \ref{lem3.4.9} to equations \eqref{eq3.4.4} yields
	\begin{align}\label{decay1}
	&\|v(t)\|_{H^k_{\alpha}(\bbR^d)}\leq C\left(e^{-\nu t}\|v^0\|_{H^k_{\alpha}(\bbR^d)}+\int_0^t e^{-\nu (t-s)}\|F_1(s)
%	(v(s),q(s),w(s))
\|_{H^k_{\alpha}(\bbR^d)}ds\right),\\\no
	&\|q(t)\|_{H^k(\mathbb{R}^{d-1})}\leq C\left((1+t)^{-\frac{d-1}{4}}\|q^0\|_{L^1(\mathbb{R}^{d-1})}+e^{-\beta t}\|q^0\|_{H^k(\mathbb{R}^{d-1})}\right. \\\no
	&\qquad \left.+\int_0^te^{-\beta(t-s)}\|F_2(s)
	%(v(s),q(s),w(s))
	\|_{H^k(\mathbb{R}^{d-1})}ds
	%\\\no&
	+\int_0^t(1+t-s)^{-\frac{d-1}{4}}\|F_2(s)
	%(v(s),q(s),w(s))
	\|_{L^1(\mathbb{R}^{d-1})}ds\right),\\\no
	&\|w(t)\|_{H^k(\mathbb{R}^{d-1})}\leq C\left((1+t)^{-\frac{d+1}{4}}\|q^0\|_{L^1(\mathbb{R}^{d-1})}+t^{-1/2}e^{-\beta t}\|q^0\|_{H^k(\mathbb{R}^{d-1})}\right.
	\\\no
	&\qquad 
	+\int_0^t(t-s)^{-\frac{1}{2}}e^{-\beta(t-s)}\|F_2
	(s)
	%(v(s),q(s),w(s))
	\|_{H^k(\mathbb{R}^{d-1})}ds
	%\\\no& \qquad
	\left.
	+\int_0^t(1+t-s)^{-\frac{d+1}{4}}\|F_2
	(s)
	%(v(s),q(s),w(s))
	\|_{L^1(\mathbb{R}^{d-1})}ds\right).
	\end{align}
	where $F_{1,2}(s)=F_{1,2}(v(s),q(s),w(s))$.
	For $t>1$, there exist  $C$ such that $e^{-\nu t}\|v^0\|_{H^k_{\alpha}(\bbR^d)}\leq Ce^{-\nu t}E_k$ and
	\begin{eqnarray*}
	%&&e^{-\nu t}\|v^0\|_{\alpha}\leq Ce^{-\nu t}E_k,\\
	&&(1+t)^{-\frac{d-1}{4}}\|q^0\|_{L^1(\mathbb{R}^{d-1})}+e^{-\beta t}\|q^0\|_{H^k(\mathbb{R}^{d-1})}\leq  C(1+t)^{-\frac{d-1}{4}}E_k,\\
	&&(1+t)^{-\frac{d+1}{4}}\|q^0\|_{L^1(\mathbb{R}^{d-1})}+t^{-1/2}e^{-\beta t}\|q^0\|_{H^k(\mathbb{R}^{d-1})} \leq C(1+t)^{-\frac{d+1}{4}}E_k.
	\end{eqnarray*}
	For  any $\delta'$ and $\gamma$ such that $0< \gamma<\delta'$, if $E_k<\gamma$ then, by Lemma \ref{lem3.4.7},
	\begin{equation*}
	\|v(s)\|_{\mathcal{H}}+\|q(s)\|_{H^k(\mathbb{R}^{d-1})}+\|w(s)\|_{H^k(\mathbb{R}^{d-1})}<\delta'\, \text{ for all }\, s\in [0, T(\delta',\gamma) ).
	\end{equation*}
	Within this bounded set, Proposition~\ref{prop3.6.17} (b) and (c) states that  %the norms of  $F_1((v(s),q(s),w(s)))$ and $F_2((v(s),q(s),w(s)))$ can be estimated as follows:
	\begin{align*}
	&\|F_1(v(s),q(s),w(s))\|_{H^k_{\alpha}(\bbR^d)}
	\leq   C_{\delta'}(\|v(s)\|_{H^k(\mathbb{R}^d)}\|v(s)\|_{H^k_{\alpha}(\bbR^d)}\\&{\qquad\qquad\qquad\qquad\qquad\qquad\qquad\qquad\qquad}+\|q(s)\|_{H^k(\mathbb{R}^{d-1})}\|v(s)\|_{H^k_{\alpha}(\bbR^d)}+\|w(s)\|_{H^k(\mathbb{R}^{d-1})}^2),\\
	&\|F_2(v(s),q(s),w(s))\|_{H^k(\mathbb{R}^{d-1})}
	\leq  C_{\delta'}(\|v(s)\|_{H^k(\mathbb{R}^d)}\|v(s)\|_{H^k_{\alpha}(\bbR^d)}\\&\qquad \qquad\qquad \qquad\qquad\qquad\qquad\qquad\qquad+\|q(s)\|_{H^k(\mathbb{R}^{d-1})}\|v(s)\|_{H^k_{\alpha}(\bbR^d)}+\|w(s)\|_{H^k(\mathbb{R}^{d-1})}^2).
	\end{align*}
	%For brevity, we denote 
	%$$R(s) = \|v(s)\|_{H^k(\mathbb{R}^d)}\|v(s)\|_{\alpha}+\|q(s)\|_{H^k(\mathbb{R}^{d-1})}\|v(s)\|_{\alpha}+\|w(s)\|_{H^k(\mathbb{R}^{d-1})}^2,$$
	%then 
	Using Proposition~\ref{prop3.6.17} (d) the inequalities \eqref{decay1} can be rewritten as follows:
	\begin{align*}
	&\|v(t)\|_{H^k_{\alpha}(\bbR^d)}\leq Ce^{-\nu t}E_k\\&
	+CC_{\delta'}\int_0^t e^{-\nu (t-s)}
	(\|v(s)\|_{H^k(\mathbb{R}^d)}\|v(s)\|_{H^k_{\alpha}(\bbR^d)}+\|q(s)\|_{H^k(\mathbb{R}^{d-1})}\|v(s)\|_{H^k_{\alpha}(\bbR^d)}+\|w(s)\|_{H^k(\mathbb{R}^{d-1})}^2)
	%R(s)
	ds,\\
	&\|q(t)\|_{H^k(\mathbb{R}^{d-1})}\leq C(1+t)^{-\frac{d-1}{4}}E_k
	\\&
	+CC_{\delta'}\int_0^te^{-\beta(t-s)}
	(\|v(s)\|_{H^k(\mathbb{R}^d)}\|v(s)\|_{H^k_{\alpha}(\bbR^d)}+\|q(s)\|_{H^k(\mathbb{R}^{d-1})}\|v(s)\|_{H^k_{\alpha}(\bbR^d)}+\|w(s)\|_{H^k(\mathbb{R}^{d-1})}^2)
%	R(s)
ds
	\\&
	 +CC_{\delta'}\int_0^t(1+t-s)^{-\frac{d-1}{4}}
	 ((\|v(s)\|_{H^k(\mathbb{R}^d)}+\|q(s)\|_{H^k(\mathbb{R}^{d-1})}) \|v(s)\|_{H^k_{\alpha}(\bbR^d)}+\|w(s)\|_{H^k(\mathbb{R}^{d-1})}^2)
	 %R(s)
	 ds
	 	\end{align*}
	 and 
	 \begin{align*}
	&\|w(t)\|_{H^k(\mathbb{R}^{d-1})}\leq C(1+t)^{-\frac{d+1}{4}}E_k\\&
	 +CC_{\delta'}\int_0^t(t-s)^{-\frac{1}{2}}e^{-\beta(t-s)}
	%R(s)%
	((\|v(s)\|_{H^k(\mathbb{R}^d)}+\|q(s)\|_{H^k(\mathbb{R}^{d-1})})\|v(s)\|_{H^k_{\alpha}(\bbR^d)}+\|w(s)\|_{H^k(\mathbb{R}^{d-1})}^2)
ds\\&	 +CC_{\delta'}\int_0^t(1+t-s)^{-\frac{d+1}{4}}
	((\|v(s)\|_{H^k(\mathbb{R}^d)}+\|q(s)\|_{H^k(\mathbb{R}^{d-1})})\|v(s)\|_{H^k_{\alpha}(\bbR^d)}+\|w(s)\|_{H^k(\mathbb{R}^{d-1})}^2)
%	R(s)
ds.
	\end{align*}
	We denote 
	\begin{align*}
	&M_v(t)=\sup_{0<s\leq t}(1+s)^{\frac{d+1}{2}}\|v(s)\|_{H^k_{\alpha}(\bbR^d)},\\
	&M_q(t)=\sup_{0<s\leq t}(1+s)^{\frac{d-1}{4}}\|q(s)\|_{H^k(\mathbb{R}^{d-1})},\\
	&M_w(t)=\sup_{0<s\leq t}(1+s)^{\frac{d+1}{4}}\|w(s)\|_{H^k(\mathbb{R}^{d-1})},
	\end{align*}
and note that for each $\delta<\delta'$, and $0<\gamma<\delta$, if $E_k<\gamma$, by Lemma \ref{lem3.4.7}, for all  $ s\in(0, T(\delta,\gamma))$, we have
$\|v(s)\|_{H^k(\mathbb{R}^d)}\leq \|v(s)\|_{\mathcal{H}}<\delta$,
	therefore
%	\begin{align}\label{eq3.4.13}
%	\|v(t)&\|_{\alpha}\leq Ce^{-\nu t}E_k +CC_{\delta'}\int_0^t e^{-\nu (t-s)}(\delta\|v(s)\|_{\alpha}+||v(s)\|_{\alpha}||q(s)\|_{H^k(\mathbb{R}^{d-1})}+||w(s)\|_{H^k(\mathbb{R}^{d-1})}^2)ds,\\\no
%	||q(t)&\|_{H^k(\mathbb{R}^{d-1})}\leq C(1+t)^{-\frac{d-1}{4}}E_k+CC_{\delta'}\int_0^te^{-\beta(t-s)}\big(\delta\|v(s)\|_{\alpha}
%	+\|v(s)\|_{\alpha}\|q(s)\|_{H^k(\mathbb{R}^{d-1})}+\|w(s)\|_{H^k(\mathbb{R}^{d-1})}^2\big)ds\\&+CC_{\delta'}\int_0^t(1+t-s)^{-\frac{d-1}{4}} \big(\delta\|v(s)\|_{\alpha}+\|v(s)\|_{\alpha}\|q(s)\|_{H^k(\mathbb{R}^{d-1})}+\|w(s)\|_{H^k(\mathbb{R}^{d-1})}^2\big)ds,\\\no
%	\|w(t)&\|_{H^k(\mathbb{R}^{d-1})}\leq C(1+t)^{-\frac{d+1}{4}}E_k\no\\&+CC_{\delta'}\int_0^t(t-s)^{-1/2}e^{-\beta(t-s)}(\delta\|v(s)\|_{\alpha}+\|v(s)\|_{\alpha}\|q(s)\|_{H^k(\mathbb{R}^{d-1})}+\|w(s)\|_{H^k(\mathbb{R}^{d-1})}^2)ds\no\\&+CC_{\delta'}\int_0^t(1+t-s)^{-\frac{d+1}{4}}(\delta\|v(s)\|_{\alpha}+\|v(s)\|_{\alpha}\|q(s)\|_{H^k(\mathbb{R}^{d-1})}+\|w(s)\|_{H^k(\mathbb{R}^{d-1})}^2)ds.
%	\end{align}
%	We now denote
%	\begin{align*}
%	&M_v(t)=\sup_{0<s\leq t}(1+s)^{\frac{d+1}{2}}||v(s)\|_{\alpha},\\
%	&M_q(t)=\sup_{0<s\leq t}(1+s)^{\frac{d-1}{4}}||q(s)\|_{H^k(\mathbb{R}^{d-1})},\\
%	&M_w(t)=\sup_{0<s\leq t}(1+s)^{\frac{d+1}{4}}||w(s)\|_{H^k(\mathbb{R}^{d-1})}.
%	\end{align*}
%	Using the above definitions, \eqref{eq3.4.13} can be rewritten as follows:
	\begin{align*}
	&\|v(t)\|_{H^k_{\alpha}(\bbR^d)}\leq Ce^{-\nu t}E_k+CC_{\delta'}\left( \delta M_v(t)\int_0^t e^{-\nu (t-s)}(1+s)^{-\frac{d+1}{2}}ds\right.
	\\& \quad
	+\left.M_v(t)M_q(t)\int_0^t e^{-\nu (t-s)}(1+s)^{-\frac{3d+1}{4}}ds
	%\\&
	+M_w^2(t)\int_0^t e^{-\nu (t-s)}(1+s)^{-\frac{d+1}{2}}ds\right),\\
	&\|q(t)\|_{H^k(\mathbb{R}^{d-1})}\leq C(1+t)^{-\frac{d-1}{4}}E_k+CC_{\delta'}\left( \delta M_v(t)\int_0^t e^{-\beta(t-s)}(1+s)^{-\frac{d+1}{2}}ds\right.\\
	&  \quad + M_v(t)M_q(t)\int_0^t e^{-\beta (t-s)}(1+s)^{-\frac{3d+1}{4}}ds+ M_w^2(t)\int_0^t e^{-\beta (t-s)}(1+s)^{-\frac{d+1}{2}}ds
	\\	&   \quad
	 +\delta M_v(t)\int_0^t (1+t-s)^{-\frac{d-1}{4}}(1+s)^{-\frac{d+1}{2}}ds
	%\\&
	+M_v(t)M_q(t)\int_0^t (1+t-s)^{-\frac{d-1}{4}}(1+s)^{-\frac{3d+1}{4}}ds\\
	&  \quad  \left.+ M_w^2(t)\int_0^t (1+t-s)^{-\frac{d-1}{4}}(1+s)^{-\frac{d+1}{2}}ds \right),\\
	&\|w(t)\|_{H^k(\mathbb{R}^{d-1})}\leq  C(1+t)^{-\frac{d+1}{4}}E_k
+CC_{\delta'}\left(\delta M_v(t)\int_0^t (t-s)^{-\frac{1}{2}}e^{-\beta(t-s)}(1+s)^{-\frac{d+1}{2}}ds\right.\\
	& \quad +M_v(t) M_q(t)\int_0^t (t-s)^{-\frac{1}{2}}e^{-\beta (t-s)}(1+s)^{-\frac{3d+1}{4}}ds
	%\\&
	 + M_w^2(t)\int_0^t (t-s)^{-\frac{1}{2}}e^{-\beta (t-s)}(1+s)^{-\frac{d+1}{2}}ds\\
	& \quad +\delta M_v(t)\int_0^t (1+t-s)^{-\frac{d+1}{4}}(1+s)^{-\frac{d+1}{2}}ds
	%\\& 
	+M_v(t)M_q(t)\int_0^t (1+t-s)^{-\frac{d+1}{4}}(1+s)^{-\frac{3d+1}{4}}ds\\
	& \quad\left. +M_w^2(t)\int_0^t (1+t-s)^{-\frac{d+1}{4}}(1+s)^{-\frac{d+1}{2}}ds\right).
	\end{align*}
	By Lemma \ref{lem3.4.12} then
	\begin{align*}
	\|v(t)&\|_{H^k_{\alpha}(\bbR^d)}\leq  Ce^{-\nu t}E_k\\& \qquad\qquad +CC_{\delta'}\left(\delta M_v(t)(1+t)^{-\frac{d+1}{2}}
	%\\& \qquad
	+M_v(t) M_q(t)(1+t)^{-\frac{3d+1}{4}}+M_w^2(t)(1+t)^{-\frac{d+1}{2}}\right),\\
	\|q(t)&\|_{H^k(\mathbb{R}^{d-1})}\leq  C(1+t)^{-\frac{d-1}{4}}E_k+CC_{\delta'} \left( \delta M_v(t)
	+M_v(t) M_q(t)+M_w^2(t)\right)(1+t)^{-\frac{d-1}{4}},\\
	\|w(t)&\|_{H^k(\mathbb{R}^{d-1})} \leq  C(1+t)^{-\frac{d+1}{4}}E_k+CC_{\delta'}\left(
	\delta M_v(t)
	%& \qquad
	+%CC_{\delta'}
	M_v(t) M_q(t)+%CC_{\delta'}
	M_w^2(t)\right)(1+t)^{-\frac{d+1}{4}}.
	\end{align*}
	%Multiplying the above inequalities by $(1+t)^{\frac{d+1}{2}}$, $(1+t)^{\frac{d-1}{4}}$ and $(1+t)^{\frac{d+1}{4}}$, respectively, 
	One then has for some $C>0$,
	\begin{align*}
	(1+t)^{\frac{d+1}{2}}	\|v(t)\|_{H^k_{\alpha}(\bbR^d)}&\leq  C\left((1+t)^{\frac{d+1}{2}}e^{-\nu t}E_k+\delta M_v(t)+M_v(t) M_q(t)(1+t)^{-\frac{d-1}{4}}+M_w^2(t)\right),\\
	(1+t)^{\frac{d-1}{4}}\|q(t)\|_{H^k(\mathbb{R}^{d-1})}&\leq  C\left(E_k+\delta M_v(t)+M_v(t) M_q(t)+M_w^2(t)\right),\\
	(1+t)^{\frac{d+1}{4}}\|w(t)\|_{H^k(\mathbb{R}^{d-1})} &\leq  C\left(E_k+ \delta M_v(t)+ M_v(t) M_q(t)+ M_w^2(t)\right).
	\end{align*}
	Since $M_v(t)$, $M_q(t)$, $M_w(t)$ are increasing functions, it can be concluded that for $t\in[1,T(\gamma,\delta))$,
	\begin{align}
	&M_v(t)\leq CE_k+C(\delta M_v(t)+M_v(t)M_q(t)+M_w^2(t)),\nonumber\\
	&M_q(t)\leq CE_k+C(\delta M_v(t)+M_v(t)M_q(t)+M_w^2(t)),\label{eq3.4.17}\\
	&M_w(t)\leq CE_k+C(\delta M_v(t)+M_v(t)M_q(t)+M_w^2(t)).\nonumber
	\end{align}
	If we set  $M(t)=M_v(t)+M_q(t)+M_w(t)$, then by \eqref{eq3.4.17},  for all $t\in[1,T(\gamma,\delta))$ and for some $C>0$ that depends on $\delta'$, 
	\begin{equation*}
	M(t)\leq CE_k+C\delta M(t)+CM^2(t),
	\end{equation*}
	Note that by Corollary \ref{tleq1} $M(t)\leq C(\delta')E_k$ for $0\leq t\leq 1$ and some constant $C(\delta')>0$.  %which depends on $\delta'$. 
	Choose $\delta_1\leq \min\{ 1/2C,\delta' \}$ and $0<\gamma<\delta<\delta_1$, then absorbing the term $\frac{1}{2}M(t)$ into the left-hand side, we have 
	$
	M(t)\leq 2CE_k+2CM^2(t)$  for all $ t\in [0, T(\delta,\gamma))$.
	Since this inequality holds for all $t\in[0,T(\delta,\gamma))$, by continuity of $M(\cdot)$, the expression $M(t)$ can not \textquotedblleft jump'' over the first root of the respective quadratic equation. This root, in turn, can be controlled by $K_1E_k$ as long as $E_k$ is  sufficiently small. Indeed, let $M_1= {1-\sqrt{1-16C^2E_k}/{4C}}$  be the first root of the equation $2CM^2-M+2CE_k=0$. If $E_k < 1/16C^2$ then 
	\begin{align*}
	M_1=\dfrac{16C^2E_k}{4C(1+\sqrt{1-16C^2E_k})} <4CE_k.
	\end{align*}
	Since $M(t)$ is continuous in $t$ and $M(0)=E_k$, it follows that if $\delta_1\leq \min\{ \delta', 1/2C, 1/16C^2 \},$ then for all $\delta\in(0,\delta_1)$ and $0<E_k\leq \gamma<\delta$ (see Lemma \ref{lem3.4.7}) we have $M(t)\leq M_1\leq K_1 E_k$ for some  $C_1>0$ and all $t\in[0, T(\delta,\gamma))$.
\end{proof}

\subsection{The boundedness of solutions in ${H^k(\mathbb{R}^{d})}$-norm}
In this subsection, we show that  the ${H^k(\mathbb{R}^{d})}$-norm of the solution $v(t)$ remains bounded for  all $t$. Together with the decay of the weighted norm for large $t$ this implies smallness of ${\mathcal{H}}$-norm of the solution when the initial conditions are small, which is  the key step in the  bootstrap argument used in  Theorem \ref{main} proved below.

%The following heuristic comment is in order. Since $\alpha _+ \geq 0$, the weight function $\gamma_{\alpha}(z)$ is bounded away from $0$ for large $z$. Thus, to establish the decay of $\|v(t,\cdot,\cdot)\|_{H^k_{\alpha}(\bbR^d)}$, it is enough to prove the decay of the solution near $z=\infty$, but it is not enough to prove the decay at $z=-\infty$, i.e., it  is  possible that $\|v(t,\cdot,\cdot)\|_{0,-\infty}\rightarrow\infty$ near $z=-\infty$ even though $\|v(t,\cdot,\cdot)\|_{\alpha,-\infty}$ is algebraically decaying since $\gamma_{\alpha}(z)\rightarrow 0$ as $z\rightarrow-\infty$. It is  the \textquotedblleft product-triangular'' structure of the nonlinearity that allows us to show the boundedness of the perturbations in the norm without a weight.

We point out that the algebraic  decay of $\|v(t,\cdot,\cdot)\|_{H^k_{\alpha}(\bbR^d)}$ implies  convergence  of $v$  to $0$  for $z$ close to $\infty$, but it  does not provide any information about the  properties of the solution  at $z=-\infty$. Indeed,    since   the weight function $\gamma_{\alpha}(z)$  with $\alpha _+ \geq 0$ is either $1$ or  grows exponentially   as $z \to \infty$, the algebraic decay of the solution in the weighted norm may be achieved only if the solution decays to 0 for large positive $z$ faster then the weight  grows. On the other hand, 
since  $\gamma_{\alpha}(z)$ exponentially converges to $0$ as $z\rightarrow-\infty$,  it  is  possible that $v$ grows at $-\infty$  but that growth is compensated by the decay of the weight. 
%even though $\|v(t,\cdot,\cdot)\|_{\alpha,-\infty}$ is algebraically decaying.
  It is  the \textquotedblleft product-triangular'' structure of the nonlinearity that allows us to show the boundedness of the perturbations in the norm without a weight.

The following is  the analogue of Proposition \ref{pro3.4.13} for $H^k(\bbR^d)$-norm.
\begin{proposition}\label{prop3.5.2}
	Assume   Hypotheses \ref{hypo3.2.1},~\ref{hypo3.2.2},~and~\ref{hypo3.2.7} and let  $k\geq[\frac{d+1}{2}]$. Choose  $\rho>0$ to satisfy 
	$$\sup\{ \Re \lambda:\lambda\in \sgma(\cL^{(2)}_1)\}<-\rho,$$ and $\delta_1$ as indicated in Proposition \ref{pro3.4.13}. There exist $\delta_2\in(0,\delta_1)$ and $C_2>0$ such that for every $\delta\in(0,\delta_2)$ and every $\gamma$ with $0<\gamma<\delta$, the following is true: if $E_k\leq \gamma$, then the solution to \eqref{sys2} for $t\in[0,T(\delta,\gamma))$ satisfies the estimates
	\begin{align}\label{eq3.5.1}
	&	\|v_1(t)\|_{H^k(\bbR^d)}\leq C_2E_k;\\
	\label{eq3.5.2}
	&	\|v_2(t)\|_{H^k(\bbR^d)} \leq C_2(1+t)^{-\frac{d+1}{2}}E_k.
	\end{align}
\end{proposition}
\begin{proof}
	Using \eqref{eq3.2.13}, we write the first equation in \eqref{sys2} as follows,
	\begin{align}
	&\partial_tv_1=L^{(1)}v_1+d_{u_2}f(0,0)v_2+H_1(q,w,v_1,v_2),\label{eq3.5.4}\\
	&\partial_tv_2=L^{(2)}v_2+H_2(q,w,v_1,v_2),\label{eq3.5.5}
	\end{align}
	where
	% we split $f=(f_1,f_2)^T$ and introduce the nonlinearity $H=(H_1,H_2)^T$ by the formula 
	 $$\begin{pmatrix}H_1(v_1,v_2,q,w)\\H_2(v_1,v_2,q,w)\end{pmatrix}=F_1(v,q,w)+(df(\phi)-df(0))v.$$ 
	Since $(v_1,v_2,q,w)(t)$ is a fixed solution of \eqref{sys2} in $\mathcal{H}^n\times H^k(\mathbb{R}^{d-1})\times H^k(\mathbb{R}^{d-1})$, we may regard \eqref{eq3.5.4}-\eqref{eq3.5.5} as a nonautonomous linear system on $ H^k(\mathbb{R}^d)^n$. The mild solutions of \eqref{eq3.5.4} and \eqref{eq3.5.5} satisfy the  system of integral equations
	\begin{align}
	&v_1(t)=e^{t\cL^{(1)}}v_1^0+\int _0^te^{(t-s)\cL^{(1)}}\big(d_{u_2}f(0,0)v_2(s)+H_1(v(s),q(s),w(s))\big)\ ds,\label{intv1}\\
	&v_2(t)=e^{t\cL^{(2)}}v_2^0+\int _0^te^{(t-s)\cL^{(2)}}H_2(v(s),q(s),w(s))\ ds.\label{intv2}
	\end{align}
	As in the proof of Proposition~\ref{pro3.4.13}, we may assume that $t\in[1,T(\delta,\gamma))$.% and large.
	
From  Lemma~\ref{lem3.7.19} we know that $\|e^{t\mathcal{L}^{(2)}}\|_{\mathcal{B}( H^k(\mathbb{R}^d))}\leq Ke^{-\rho t}$. By the definition of $T(\delta,\gamma)$, for $0<\delta<\delta_1$, if $0<\gamma<\delta$ and $E_k<\gamma$, then for all $s\in[1,T(\gamma,\delta))$
	$$
	\|v(s)\|_{\mathcal{H}}+\|q(s)\|_{H^k(\mathbb{R}^{d-1})}+\|w(s)\|_{H^k(\mathbb{R}^{d-1})}<\delta<\delta_1.$$
	It follows from Lemmas~\ref{lem3.5.1} and part (b') of  Proposition~\ref{prop3.6.17} that there exists a constant $C_{\delta_1}>0$ such that
	\begin{align}\label{h1h2}
	&\| H_i(v_1(s),v_2(s),q(s),w(s)) \|_{H^k(\bbR^d)}\leq C_{\delta_1}(\|v(s)\|_{H^k_{\alpha}(\bbR^d)}+\|v(s)\|_{H^k(\bbR^d)}\|v_2(s)\|_{H^k(\bbR^d)}\\\no
	&\qquad\qquad\qquad +\|v(s)\|_{H^k(\bbR^d)}\|v(s)\|_{H^k_{\alpha}(\bbR^d)}+\|q(s)\|_{H^k(\mathbb{R}^{d-1})}\|v(s)\|_{H^k_{\alpha}(\bbR^d)}+\|w(s)\|_{H^k(\mathbb{R}^{d-1})}^2)
	\end{align}
	for $i=1,2$.
	Thus by Proposition~\ref{pro3.4.13}, Lemma~\ref{lem3.4.12} and \eqref{h1h2}, and also because $\|v(t)\|_{H^k(\bbR^d)}<\delta$, formula \eqref{intv2} yields,  for some $K>0$ and all $t\in[1,T(\gamma,\delta))$,% the estimates for $v_2(t)$ in $\mathcal{H}^{n_2}_0$ as follows:
	\begin{align*}
	%\label{eq3.9.73}
	\|v_2(t)\|_{H^k(\bbR^d)}&
	\leq Ke^{-\rho t}E_k+K\int_{0}^t e^{-\rho (t-s)}C_{\delta_1}(\|v(s)\|_{H^k_{\alpha}(\bbR^d)}+\|v(s)\|_{H^k(\bbR^d)}\|v_2(s)\|_{H^k(\bbR^d)}\notag \\&\qquad +\|v(s)\|_{H^k(\bbR^d)} \|v(s)\|_{H^k_{\alpha}(\bbR^d)}+\|q(s)\|_{H^k(\mathbb{R}^{d-1})}\|v(s)\|_{H^k_{\alpha}(\bbR^d)}+\|w(s)\|_{H^k(\mathbb{R}^{d-1})}^2)ds
	%&\leq  Ke^{-\rho t}E_k+KC_{\delta_1}\int_0^te^{-\rho (t-s)}(\delta\|v_2(s)\|_{H^k(\bbR^d)}+(1+\delta)\|v(s)\|_{H^k_{\alpha}(\bbR^d)}\\\no
	%&+\|q(s)\|_{H^k(\mathbb{R}^{d-1})}\|v(s)\|_{\alpha}+\|w(s)\|_{H^k(\mathbb{R}^{d-1})}^2)ds\\\no
	%&\leq K\left( 
	%e^{-\rho t}E_k+C_{\delta_1}\delta\int_0^t e^{-\rho (t-s)}\|v_2(s)\|_0ds+C_{\delta_1}E_k(1+\delta)\int_0^t e^{-\rho (t-s)}(1+s)^{-\frac{d+1}{2}}ds\right\\\no}
	%&+KC_{\delta_1}E_k\int_0^t e^{-\rho (t-s)}(1+s)^{-\frac{3d+1}{4}}ds+KC_{\delta_1}E_k\int_0^t e^{-\rho (t-s)}(1+s)^{-\frac{d+1}{2}}ds\\\no
	%&\leq  Ke^{-\rho t}E_k+KC_{\delta_1}\delta\int_0^t e^{-\rho (t-s)}\|v_2(s)\|_0ds\\\no
	%&+ KC_{\delta_1}E_k(1+\delta)(1+t)^{-\frac{d+1}{2}}+KC_{\delta_1}E_k(1+t)^{-\frac{3d+1}{4}}+KC_{\delta_1}E_k(1+t)^{-\frac{d+1}{2}}\\\no
	%&\leq KE_k\big(e^{-\rho t}+(1+t)^{-\frac{d+1}{2}}+(1+t)^{-\frac{3d+1}{4}}\big)+KC_{\delta_1}\delta\int_0^t e^{-\rho (t-s)}\|v_2(s)\|_{H^k(\bbR^d)}ds\notag \\&
%	\leq  K(1+t)^{-\frac{d+1}{2}}E_k +K\delta\int_0^t e^{-\rho (t-s)}\|v_2(s)\|_{H^k(\bbR^d)}ds,
	\end{align*}
	 We define
$
	M_{v_2}(t)=\sup_{0<s\leq t}(1+s)^{\frac{d+1}{2}}\|v_2(s)\|_{H^k(\bbR^d)} $ and use
 Lemma~\ref{lem3.4.12} to obtain
	\begin{align*}%\label{eq3.9.72}
	\|v_2(t)\|_{H^k(\bbR^d)}&\leq  K(1+t)^{-\frac{d+1}{2}}E_k+K\delta M_{v_2}(t)\int_0^t e^{-\rho (t-s)}(1+s)^{-\frac{d+1}{2}}ds%\no
	\\ &
	\leq  K(%(1+t)^{-\frac{d+1}{2}}
	E_k+\delta%(1+t)^{-\frac{d+1}{2}} 
	M_{v_2}(t))(1+t)^{-\frac{d+1}{2}},
	\end{align*}
	and thus,
	$
	(1+t)^{\frac{d+1}{2}}\|v_2(t)\|_{H^k(\bbR^d)}\leq K E_k+K\delta  M_{v_2}(t).
	$
	Because $M_{v_2}(t)$ is increasing, we  conclude that
	\begin{equation*}
	M_{v_2}(t)\leq KE_k+K\delta M_{v_2}(t).
	\end{equation*}
	Choosing $\delta_2<\min\{ \delta_1, 1/{2K} \}$, we obtain that if $0<\delta<\delta_2$ and $0<\gamma< \delta$ then
	\begin{equation*}
	\| v_2(t) \|_{H^k(\bbR^d)}\leq \tilde C_2(1+t)^{-\frac{d+1}{2}}E_k,
	\end{equation*}
	for some $\tilde C_2>0$ on the time interval $t\in[1,T(\gamma,\delta))$, thus finishes the proof of \eqref{eq3.5.2}.
	
	To prove \eqref{eq3.5.1}, we first use Lemma \ref{lem3.7.19} in \eqref{intv1} to infer $\|e^{t\mathcal{L}^{(1)} }\|_{\mathcal{B}( H^k(\mathbb{R}^d))}\leq K.$  We then  find an estimate for the solution to \eqref{intv1}  based on \eqref{h1h2},
	\begin{align*}
	\|v_1(t)\|_{H^k(\bbR^d)} \leq &KE_k+KC_{\delta_1}\int_0^t \big[C \|v_2(s)\|_{H^k(\bbR^d)}+\|v(s)\|_{H^k_{\alpha}(\bbR^d)}+\|v(s)\|_{H^k(\bbR^d)}\|v_2(s)\|_{H^k(\bbR^d)}\no\\
	&+\|v(s)\|_{H^k(\bbR^d)}\|v(s)\|_{H^k_{\alpha}(\bbR^d)}+\|q(s)\|_{H^k(\mathbb{R}^{d-1})}\|v(s)\|_{H^k_{\alpha}(\bbR^d)}+\|w(s)\|_{H^k(\mathbb{R}^{d-1})}^2 \big]ds,
	\end{align*}
	for some $C>0$.
	Since $0<E_k<\gamma<\delta<\delta_2$, for $t\in[1,T(\gamma,\delta))$ we have $\|v(t)\|_{H^k(\bbR^d)}<\delta$ by Lemma \ref{lem3.4.7} and the definition \ref{DFNtmax} of $T(\delta,\gamma)$. Therefore,
	\begin{align*}
	\|v_1(t)\|_{H^k(\bbR^d)}&\leq KE_k+KC_{\delta_1}\int_0^t\big( C\|v_2(s)\|_{H^k(\bbR^d)}+\|v(s)\|_{H^k_{\alpha}(\bbR^d)}+\delta\|v_2(s)\|_{H^k(\bbR^d)}
	\\\no
	&
	\qquad\quad \,\,\,+\delta\|v(s)\|_{H^k_{\alpha}(\bbR^d)}+\|q\|_{H^k(\mathbb{R}^{d-1})}\|v(s)\|_{H^k_{\alpha}(\bbR^d)}+\|w(s)\|^2_{H^k(\mathbb{R}^{d-1})} \big)ds\\\no
	&\leq KE_k+K(C+\delta)\int_0^t\|v_2(s)\|_{H^k(\bbR^d)} ds+K(1+\delta)\int_0^t \|v(s)\|_{H^k_{\alpha}(\bbR^d)}ds\\\no
	&\qquad\quad \,\,\, +K\int_0^t\|q\|_{H^k(\mathbb{R}^{d-1})}\|v(s)\|_{H^k_{\alpha}(\bbR^d)}ds+K\int_0^t\|w(s)\|^2_{H^k(\mathbb{R}^{d-1})}ds.
	\end{align*}
	Finally, we apply Proposition~\ref{pro3.4.13}, equation \eqref{eq3.5.2} and Lemma~\ref{lem3.4.12} to obtain that
	\begin{align*}
	\|v_1(t)\|_{H^k(\bbR^d)}\leq& K\left(E_k+C_{\delta_1}(C+\delta)E_k\int_0^t (1+s)^{-\frac{d+1}{2}}ds
	+C_1(1+\delta)E_k\int_0^t (1+s)^{-\frac{d+1}{2}}ds\right.\\\no
	&\left.+{C_1}^2E_k\int_0^t(1+s)^{-\frac{3d+1}{4}}ds
	+{C_1}^2E_k\int_0^t(1+s)^{-\frac{d+1}{2}}ds\right)
	\leq \tilde{\tilde C}_2E_k.
	\end{align*}
	The proposition then holds with $K_2=\max\{\tilde C_2, \tilde{\tilde C}_2\}$.
\end{proof}

\subsection{Global existence and bounds}

We  now  present the main result of this paper. It relies on a bootstrap argument based on Propositions~\ref{pro3.4.13} and \ref{prop3.5.2}. The constant $\delta_0$ in the next theorem can be taken to be $\delta_0=\delta_2$, where $\delta_2$ is chosen as in Proposition~\ref{prop3.5.2}.
\begin{theorem} \label{main}
	Assume   Hypotheses \ref{hypo3.2.1},  \ref{hypo3.2.2}, \ref{hypo3.2.3}, and \ref{hypo3.2.7} and 
	let   $k\geq [\frac{d+1}{2}]$. There exist positive constants  $\delta_0$ and $C$ such that for each $0<\delta<\delta_0$ there exists $0< \eta<\delta$ such that the following is true. Let $(v^0,q^0, w^0)\in \mathcal{H}^n \times H^k(\bbR^{d-1})\times H^k(\bbR^{d-1})^{d-1}$ be the initial condition satisfying $$E_k=\|v^0\|_{\mathcal{H}}+\|q^0\|_{H^{k}(\mathbb{R}^{d-1})}+\|q^0\|_{W^{1,1}(\mathbb{R}^{d-1})}\leq \eta$$ and let $(v(t),q(t),w(t))\in\mathcal{H}^n\times H^k(\bbR^{d-1})\times H^k(\bbR^{d-1})^{d-1}$ be the solution of the evolution equation \eqref{sys2} with the initial condition $(v^0,q^0,w^0)$. Then for all $t> 0$,
	\begin{itemize}
		\item[(1)] $(v(t),q(t),w(t))$ is defined in $\mathcal{H}^n \times H^k(\bbR^{d-1})\times H^k(\bbR^{d-1})^{d-1}$;
		\item[(2)] $\|v(t)\|_{\mathcal{H}}+\|q(t)\|_{H^k(\mathbb{R}^{d-1})}+\|w(t)\|_{H^k(\mathbb{R}^{d-1})}\leq \delta$;
		\item[(3)] $\|v(t)\|_{H^k_{\alpha}(\bbR^d)}\leq C(1+t)^{-\frac{d+1}{2}}E_k$;
		\item[(4)] $\|q(t)\|_{H^k(\mathbb{R}^{d-1})}\leq C(1+t)^{-\frac{d-1}{4}}E_k$;
		\item[(5)] $\|w(t)\|_{H^k(\mathbb{R}^{d-1})}\leq C(1+t)^{-\frac{d+1}{4}}E_k$;
		\item[(6)] $\|v_1(t)\|_{H^k(\mathbb{R}^{d})}\leq CE_k$;
		\item[(7)] $\|v_2(t)\|_{H^k(\mathbb{R}^{d})} \leq C(1+t)^{-\frac{d+1}{2}}E_k.$
	\end{itemize}
\end{theorem}
\begin{proof}
	We choose $\d_0=\d_2$, with $\d_2$ from Proposition~\ref{prop3.5.2}, and then we  fix $C>\max\{ 1,C_1,C_2 \}$ with $C_1$ and $C_2$ from Propositions~\ref{pro3.4.13} and \ref{prop3.5.2} respectively. We take $\gamma$ such that  $0<\g <\d<\d_0$ and set $\eta=C^{-1}\g/3$. Let $(v^0,q^0,w^0)\in\ran \cQ_{\mathcal{H}}\times H^k(\bbR^{d-1})\times H^k(\bbR^{d-1})^{d-1}$ be the initial value of the solution $(v(t),q(t),w(t))\in \ran \cQ_{\mathcal{H}}\times H^k(\bbR^{d-1})\times H^k(\bbR^{d-1})^{d-1}$ of equation $\eqref{sysu}$ such that $E_k\leq \eta$. Since $\eta<\g<\d$, we can apply Propositions~\ref{pro3.4.13} and \ref{prop3.5.2} with $\g$ replaced by $\eta$ and conclude that for all $t\in[0,T(\g,\eta))$ assertions (1)-(7) of the theorem hold.
	
	We claim that $T(\d,\eta)=\infty$; thus the theorem holds as soon as the claim is proved. To prove the claim,   we fix any $T\in(0,T(\d,\eta))$. At that $T$,   we note that by assertions (3),(6), and (7) and the definition of the $\cH$ norm, $\|v(T,v^0,q^0)\|_{\mathcal{H}} \leq \sqrt{2} CE_k$  and therefore
	\begin{align*}
	&\|v(T,v^0,q^0,w^0)\|_{\mathcal{H}}+\|q(T,v^0,q^0,w^0)\|_{H^k(\mathbb{R}^{d-1})}+\|w(T,v^0,q^0,w^0)\|_{H^k(\mathbb{R}^{d-1})}\\
	& \qquad\qquad\qquad\qquad\qquad\qquad\qquad\qquad\qquad\qquad\qquad\qquad\qquad
	 \leq (2+\sqrt{2}) CE_k\leq (2+\sqrt{2})C\eta=\gamma.
	\end{align*}
	We now apply Lemma \ref{lem3.4.7} to the solution with the initial data $(v(T),q(T),w(T))$. This lemma says that for all $t\in[0,T(\d,\g))$ we have the inequality
	\begin{equation*}
	\|v(t+T)\|_{\mathcal{H}}+\|q(t+T)\|_{H^k(\mathbb{R}^{d-1})}+\|w(t+T)\|_{H^k(\mathbb{R}^{d-1})}\leq \d,
	\end{equation*}
	so, if $E_k\leq \eta$ then $\|v(t)\|_{\mathcal{H}}+\|q(t)\|_{H^k(\mathbb{R}^{d-1})}+\|w(t)\|_{H^k(\mathbb{R}^{d-1})}\leq \d$ for all $t\in[0,T+T(\d,\g))$. By Definition \ref{DFNtmax} that means that $T(\d,\eta)\geq T+T(\d,\g)$ for each $T\in(0,T(\d,\g))$ %. Taking the sup over such $T$ yields $T(\d,\eta)\geq T(\d,\eta)+T(\d,\g)$ 
	and therefore $T(\d,\g)=\infty$ which completes the proof. % as claimed, completing the proof of the theorem.
\end{proof}

\appendix
%\addcontentsline{toc}{section}{Appendix}
\section{Lipschitz Properties of the Nemytskij Operator}\lb{app}
In this appendix we prove the Lipschitz properties of the Nemytskij operator \eqref{eq3.13A} induced by the nonlinear term in  system \eqref{sysu} that we  consider. In order to do so, we need the following lemma from \cite{RunstSickel} and a generalized H\"{o}lder's inequality (see, e.g., \cite{WZ}).
\begin{lemma}\label{lem2.3.2} 
	For Sobolev spaces $W^{k,p}(\mathbb{R}^d)$ and $W^{k_0,p_0}(\mathbb{R}^d)$, if $k>k_0$ and $k-\frac{d}{p} > k_0-\frac{d}{p_0} $, then the Sobolev embedding $W^{k,p}(\mathbb{R}^d)\hookrightarrow W^{k_0,p_0}(\mathbb{R}^d)$ holds.
\end{lemma}

\begin{lemma}\label{lem2.3.3}
	Assume that $r\in(0, \infty)$ and $p_1$, ..., $p_n \in (0, \infty]$ are such that
	$\sum_{k=1}^n \frac1{p_k}=\frac1r.$ Then for all $\mu$-measurable real or complex-valued functions $f_1$, ..., $f_n$,
	$$\Big\|\prod_{k=1}^n f_k\Big\|_{L^r(\mu)}\le \prod_{k=1}^n\|f_k\|_{L^{p_k}(\mu)}.$$
	In particular,
	$f_k\in L^{p_k}(\mu)$ for all $k\in\{1,\ldots,n\}$ implies that $\prod_{k=1}^n f_k \in L^r(\mu)$.
\end{lemma}

%Recall notation ${H^k(\mathbb{R}^d)}=H^k(\bbR^d)$ and $\mathcal{H}_{\a}=H^k_{\a}(\bbR^d)=H^k_{\a}(\bbR)\times H^k(\bbR^{d-1})$ and $\mathcal{H}={H^k(\mathbb{R}^d)}\cap \mathcal{H}_{\a}$, where $\g_{\a}$ is defined in \eqref{eq3.1.7}. 
Next we formulate an analogue of \cite[Proposition 7.2]{GLS2}. 
\begin{proposition}\label{prop2.3.2}
	Assume  that  $m:(q,u) \mapsto m(q,u)\in\mathbb{R}$ is a function from $C^{k+1}(\mathbb{R}^2)$ with $k\geq[\frac{d+1}{2}]$. Consider the formula
	\begin{equation}\label{eq3.13A}
	(q(x),u(x),v(x))\mapsto m(q(x),u(x))v(x),
	\end{equation}
	where $q(\cdot)$, $u(\cdot)$, $v(\cdot) :\mathbb{R}^{d}\mapsto \mathbb{R}$, and the variable $x=(x_1,...,x_d)\in\mathbb{R}^{d}$.
	\begin{itemize}
		\item[(1)] Formula \eqref{eq3.13A} defines a mapping from $H^k(\mathbb{R}^d)\times H^k(\mathbb{R}^d)^2$ to $ H^k(\mathbb{R}^d)$ that is locally Lipschitz on any set of the form $\{ (q,u,v):\|q\|_{H^k(\mathbb{R}^d)}+\|u\|_{H^k(\mathbb{R}^d)}+\|v\|_{H^k(\mathbb{R}^d)}\leq K \}$.
		\item[(2)] Formula \eqref{eq3.13A} defines a mapping from $H^k(\mathbb{R}^d)\times\mathcal{H}^2$ to $\mathcal{H}$ that is locally Lipschitz on any set of the form $\{ (q,u,v):\|q\|_{H^k(\mathbb{R}^d)}+\|u\|_{\mathcal{H}}+\|v\|_{\mathcal{H}}\leq K \}$.
	\end{itemize}
\end{proposition}
\begin{proof}
	In the following proof, we will abbreviate the norm of $q,u,v$, for instance, $H^k(\mathbb{R}^d)$ to $H^k$ since $q,u,v$ are all $\mathbb{R}^d\rightarrow \mathbb{R}$.  We shall  use the equivalent Sobolev norm (see, e.g., \cite{NS}%, page 316
	):
	\begin{equation*}%\label{eq2.3.2}
	\|f\|_{H^k}\sim\|f\|_{L^2}+\sum_{a_1+\cdots+a_d=k}\Big\|\frac{\partial^k}{\partial x_1^{a_1}\cdots \partial x_d^{a_d}}f\Big\|_{L^2},
	\end{equation*}
	where the sum is over all $d$-tuples $(a_1,...,a_d)$ of non-negative integers such that $\sum_{l=1}^d a_l=k$, and $\frac{\partial^{a_l}}{\partial x_l^{a_l}}$ is the $a_l$-th differentiation of functions with respect to $x_l$, $l=1$, ..., $d$. 
	
For variation in $q$, we write 
	$
	m(q+\bar{q},u)-m(q,u)=\bar{q} \big(\int_{0}^{1}m_q(q+t\bar{q},u)dt\Big),
	$
	and,  using  embedding of $H^k(\mathbb{R}^d)$  in $L^{\infty}(\mathbb{R}^d)$, obtain%herefore the estimate of $m(q+\bar{q},u)v-m(q,u)v$ on $L^2(\mathbb{R}^d)$ follows from
	$$
	\|m(q+\bar{q},u)v-m(q,u)v\|_{L^2}\leq\|m\|_{C^1(\mathbb{R}^2)}\|\bar{q}\|_{L^{\infty}}\|v\|_{L^2}\leq\|m\|_{C^1(\mathbb{R}^2)}\|\bar{q}\|_{H^k}\|v\|_{L^2}.
	$$
	The estimate of $m(q+\bar{q},u)v-m(q,u)v$ in $L^2_{\alpha}(\mathbb{R})\otimes L^2(\mathbb{R}^{d-1})$ follows from
	$$
	\|\gamma_{\alpha}\big(m(q+\bar{q},u)v-m(q,u)v\big)\|_{L^2}\leq\|m\|_{C^1(\mathbb{R}^2)}\|\bar{q}\|_{L^{\infty}}\|\gamma_{\alpha}v\|_{L^2}\leq\|m\|_{C^1(\mathbb{R}^2)}\|\bar{q}\|_{H^k}\|\gamma_{\alpha}v\|_{L^2}.
	$$
	We denote $m_1(q,\bar{q},u)=\int_{0}^{1}m_q(q+t\bar{q},u)dt$. To estimate the  derivatives of $m_1(q,\bar{q},u)\bar{q}v$   %$\dfrac{\partial^{k} \big(m_1(q,\bar{q},u)\bar{q}v\big)}{\partial x_1^{a_1}\cdots\partial x_d^{a_d}}$ 
	in $L^2(\mathbb{R}^d)$, we need the general Leibniz Rule  %see p.318 of 
	\cite{Olver}: if $f_1$, ..., $f_m$ are all $n$-times differentiable functions, then their product $f_1\cdots f_m$ is also $n$-times differentiable and its $n$th derivative is given by
	$$\left(f_1 f_2 \cdots f_m\right)^{(n)}=\sum_{k_1+k_2+\cdots+k_m=n} {n \choose k_1, k_2, \ldots, k_m} \prod_{1\leq t\leq m}f_{l}^{(k_{l})},$$
	where the sum extends over all m-tuples $(k_1,...,k_m)$ of non-negative integers such that  $\sum_{l=1}^m k_l=n$ and $\displaystyle{{n \choose k_1, k_2, \ldots, k_m}
	= \frac{n!}{k_1!\, k_2! \cdots k_m!}}$
	are the multinomial coefficients.
	%Fix any tuple $(a_1,...,a_d)$ and use the General Leibniz Rule. 
	We then have
	\begin{align*}%\label{eq2.3.4} \no
	&\frac{\partial^{k}}{\partial x_1^{a_1}\cdots \partial x_d^{a_d}}\big(m_1(q,\bar{q},u)
	\bar{q}v\big)
	%\\\no
	=\frac{\partial^{a_1}}{\partial x_1^{a_1}}\cdots\frac{\partial^{a_{d-1}}}{\partial x_{d-1}^{a_{d-1}}}\sum_{b_{d}+c_{d}+e_{d}=a_d}{a_d \choose b_{d},c_{d},e_{d}}\frac{\partial^{b_{d}}
m_1}{\partial x_d^{b_{d}}}%m_1
	%(q,\bar{q},u)
	\cdot \frac{\partial^{c_{d}}\bar{q}}{\partial x_d^{c_{d}} }%\bar{q}
	\cdot\frac{\partial^{e_d}v}{\partial x_d^{e_d} }
	%v
	\\\no
	=&\qquad \frac{\partial^{a_1+\cdots+a_{d-2}}}{\partial x_1^{a_1}\cdots\partial x_{d-2}^{a_{d-2}}}\sum_{b_{d}+c_{d}+e_{d}=a_d}{a_d \choose b_{d},c_{d},e_{d}}\\\no
	&
	\sum_{b_{d-1}+c_{d-1}+e_{d-1}=a_{d-1}}{a_{d-1} \choose b_{d-1}, c_{d-1},e_{d-1}}\frac{\partial^{b_{d-1}+b_{d}}m_1}{\partial x_{d-1}^{b_{d-1}}\partial x_d^{b_d}}
	%m_1(q,\bar{q},u)
	\cdot\frac{\partial^{c_{d-1}+c_{d}}\bar{q}}{\partial x_{d-1}^{c_{d-1}}\partial x_d^{c_{d}}}%\bar{q}
	\cdot\frac{\partial^{e_{d-1}+e_{d}}v}{\partial x_{d-1}^{e_{d-1}}\partial x_d^{e_{d}}}%v 
	%\\\no&
	=\cdots
	\\&=\sum_{b_{d}+c_{d}+e_{d}=a_d}{a_d \choose b_{d}, c_{d},e_{d}}\cdots%\\\no
	%&
	%\qquad
	 \sum_{b_{1}+c_{1}+e_{1}=a_1}{a_1 \choose b_{1}, c_{1},e_{1}}\frac{\partial^{b_{1}+\cdots+b_{d}}m_1}{\partial x_1^{b_{1}}\cdots\partial x_{d}^{b_{d}}}%m_1
	%(q,\bar{q},u)
	\cdot\frac{\partial^{c_{1}+\cdots+c_{d} }\bar{q}}{\partial x_1^{c_{1}}\cdots\partial x_{d}^{c_{d}}}%\bar{q}
	\cdot\frac{\partial^{e_{1}+\cdots+e_{d}}v}{\partial x_1^{e_{1}}\cdots\partial x_{d}^{e_{d}}}%v
	,
	\end{align*}
	where $a_1+\cdots+a_d=k$.

	We now refer to the Higher Chain Formula (see, e.g.,  \cite[Lemma 1]{TWM}). We consider a mapping $$M : x \in X \subset \mathbb{R}^{d}\underset{g}{\rightarrow}(q(x),\bar{q}(x),u(x)) \in G \subset \mathbb{R}^{3}\underset{h}{\rightarrow}m_1\in\mathbb{R},$$ where $X$, $G$
	are open subsets of $\mathbb{R}^{d}$ and $\mathbb{R}^{3}$ respectively, and $g$, $h$ are sufficiently smooth functions. We denote $(g_1(x), g_2(x),g_3(x))=(q(x),\bar{q}(x),u(x))$. For each $i$ in the set $J_s$ of integers $1$,  $2$, ..., $s$, where $s=b_1+\cdots+b_d$, let $t_i$ denote one of the independent variables $x_1$,  ..., $x_d$. A partition of $J_s$ is a family of pairwise disjoint nonempty subsets of $J_s$ whose union is $J_s$. Sets in a partition are called blocks. A block's function is to assign a label to each block of a partition. The set of all block functions from a partition $P$ of $J_s$ into $J_3$ is
	denoted by $P_3$.  The set of all partitions of $J_s$ is denoted by $P_s$. We  then have
	\begin{equation} \label{A4}
	\frac{\partial^s m_1(g(x))}{\partial {t_1}\cdots\partial {t_s}}
	=\sum_{P\in P_s}\sum_{\lambda\in P_{3}}\Big\{ \Big( \prod_{B\in P}\frac{\partial}{\partial{g_{\lambda(B)}}}\Big)m_1   \Big\}
	\Big\{   \prod_{B\in P} \Big[\Big( \prod_{b\in B} \frac{\partial}{\partial{t_b}}\Big) g_{\lambda(B)} \Big]  \Big\}.
	\end{equation}
	 $B\in P$  means that $B$ runs through the list of all of the blocks of the partition $P$. The number of blocks in the partition $P$ is denoted  by $|P|$.  The partition $P$ then can be written as $P=\{B_1,...,B_{|P|}\}$. Let $|B_i|$ be the size of the block $B_i$, $i=1$, $2$, ..., $|P|$. For fixed multinomial coefficients ${a_i\choose b_i,c_i,e_i}$ ($i=1$, ..., $d$), fixed $P\in P_s$ and $\lambda\in P_3$ we need to estimate the following term in both $L^2(\mathbb{R}^d)$ and $L^2_{\alpha}(\mathbb{R})\otimes L^2(\mathbb{R})^{d-1}$:
	$$\Big\{\Big( \prod_{B\in P}\frac{\partial}{\partial{g_{\lambda(B)}}}\Big) m_1\Big\} 
	\Big\{   \prod_{B\in P} \Big[ \Big( \prod_{b\in B} \frac{\partial}{\partial{t_b}}\Big) g_{\lambda(B)} \Big]  \Big\}\frac{\partial^{c_1+\cdots+c_d}\bar{q}}{\partial{x_1^{c_1}}\cdots\partial{x_d^{c_d}}}\frac{\partial^{e_1+\cdots+e_d}v}{\partial{x_1^{e_1}}\cdots\partial{x_d^{e_d}}},
	$$ 
	where $g_{\lambda(B_i)}$ is one of $(g_1,g_2,g_3)=(q,\bar{q},u)$. To obtain the estimates we distinguish several cases.

	\textbf{Case 1.1:}
		If $b_1+\cdots+b_d\neq0$, $c_{1}+\cdots+c_{d}\neq0$ and $e_{1}+\cdots+e_{d}\neq0$, we use Lemma \ref{lem2.3.3} with 		$
		\frac{1}{2}=\sum_{i=1}^{\left|P\right|+2}\frac{1}{p_i},
		$
		where $p_i$ will be chosen below. If we denote $P=\{B_1,B_2,...,B_{\left|P\right|}\}$ and $l=|P|+2$ (note that $3\leq l\leq k$), and  introduce the $l$-tuple $$(n_1,n_2,...,n_l)=(|B_1|,|B_2|,...,|B_{l-2}|,c_{1}+\cdots+c_{d},e_{1}+\cdots+e_{d}),$$ then
		\begin{align}\label{eq2.3.6} \no
		&\Big\|\Big\{\Big(
		 \prod_{B\in P}\frac{\partial}{\partial{g_{\lambda(B)}}}
		 \Big)m_1\Big\} 
		 \Big\{   \prod_{B\in P} \Big[ \Big(
		 \prod_{b\in B} \frac{\partial}{\partial {t_b}}\Big) g_{\lambda(B)} \Big]  \Big\}\frac{\partial^{c_1+\cdots+c_d}\bar{q}}{\partial{x_1^{c_1}}\cdots\partial{x_d^{c_d}}}\frac{\partial^{e_1+\cdots+e_d}v}{\partial{x_1^{e_1}}\cdots\partial{x_d^{e_d}}}\Big\|_{L^2}\\\no
		&\leq\left\|m_1^{(\left|P\right|)}\right\|_{L_{\infty}(\mathbb{R}^3)} \Big\|\frac{\partial^{c_1+\cdots+c_d}\bar{q}}{\partial{x_1^{c_1}}\cdots\partial{x_d^{c_d}}}\Big\|_{L^{p_{l-1}}} \Big\|\frac{\partial^{e_1+\cdots+e_d}v}{\partial{x_1^{e_1}}\cdots\partial{x_d^{e_d}}}\Big\|_{L^{p_{l}}}
	%	\\\no 
	%	& \prod_{B\in P} \left[ \left(\prod_{b\in B} \frac{\partial}{\partial {t_b}}\right) g_{\lambda(B)} \right]\\
	 \prod_{i=1}^{l-2}	 \Big\|\Big(\prod_{b\in B_i} \frac{\partial}{\partial{t_b}}\Big) g_{\lambda(B_i)}\Big\|_{L^{p_1}}%\cdots\Big\|\Big(\prod_{b\in B_{l-2}} \frac{\partial}{\partial{t_b}}\Big) g_{\lambda(B_{l-2})}\Big\|_{L^{p_{l-2}}}
		%\\\no &
	\\
		& \leq\|m_1^{(|P|)}\|_{L_{\infty}(\mathbb{R}^3)}\|g_{\lambda(B_1)}\|_{W^{n_1,p_1}}\cdots\|g_{\lambda(B_{l-2})}\|_{W^{n_{l-2},p_{l-2}}}\|\bar{q}\|_{W^{n_{l-1},p_{l-1}}}\|v\|_{W^{n_{l},p_{l}}},%\no
		\end{align}
		where $W^{k,p}$ are the Sobolev spaces of $k$ times differentiable functions from $L^p$. 
		
		In equation \eqref{eq2.3.6}, $\sum_{i=1}^{|P|}|B_i|=\sum_{j=1}^{d}b_{j}$  because  $P$ is a partition of the $b_{1}+\cdots+b_{d}$ indices of $$(\underbrace{x_1,...,x_1}_{b_{1}\, times},
		...,\underbrace{x_d,...,x_d}_{b_{d}\, times})$$ and $\{ B_1,...,B_{|P|} \}$ are all blocks in the partition $P$.
			Since all $b_1+\cdots+b_d$, $c_1+\cdots+c_d$, and $e_1+\cdots+e_d$ are nonzero, it is obvious that $k>n_i$. By Lemma~\ref{lem2.3.2} then, in order to prove that $W^{k,2}(\mathbb{R}^d)=H^k(\mathbb{R}^d)\hookrightarrow W^{n_i,p_i}(\mathbb{R}^d)$,  we must to show that $k-\frac{d}{2}>n_i-\frac{d}{p_i}$. If we choose $\frac{1}{p_i}=(\frac{1}{2}-\frac{k}{d})\frac{1}{l}+\frac{n_i}{d}$,  then 
					%\begin{equation}%\lb{sumpi}
		$\sum_{i=1}^{l}\frac{1}{p_i}%=(\frac{1}{2}-\frac{k}{d})+\frac{1}{d} \sum\limits_{i=1}^{l}{n_i}=\frac{1}{2}-\frac{k}{d}+\frac{k}{d}
		=\frac{1}{2},$
	%	\end{equation}
		and
		\begin{align*}%\lb{ni}
		n_i-\frac{d}{p_i}&=n_i-d\Big( \Big(\frac{1}{2}-\frac{k}{d}\Big)\frac{1}{l}+\frac{n_i}{d} \Big)
		%&=n_i-(\frac{d}{2l}-\frac{k}{l}+n_i)\\\no
		=\frac{k}{l}-\frac{d}{2l}. %=\frac{1}{l}(k-\frac{d}{2}).
		\end{align*}
		Since $k\geq[\frac{d+1}{2}]$ and $l> 2$, we can conclude that 
		\begin{equation*}%\label{eq2.3.12}
		k-\frac{d}{2}>\frac{1}{l}(k-\frac{d}{2})=n_i-\frac{d}{p_i},\, i=1,...,l.
		\end{equation*}
		Therefore $H^k(\mathbb{R}^d)$ can be embedded into $W^{n_i,p_i}(\mathbb{R}^d)$, $i=1,...,l$. Following \eqref{eq2.3.6}, we then  have
		\begin{align}\label{eq2.3.9}
		&\Big\|m_1^{(\left|P\right|)}\cdot\prod_{B\in P} \Big[ \Big(\prod_{b\in B} \frac{\partial}{\partial{t_b}}\Big) g_{\lambda(B)} \Big]\cdot\frac{\partial^{c_1+\cdots+c_d}\bar{q}}{\partial{x_1^{c_1}}\cdots\partial{x_d^{c_d}}}\cdot\frac{\partial^{e_1+\cdots+e_d}v}{\partial{x_1^{e_1}}\cdots\partial{x_d^{e_d}}}\Big\|_{L^2}\\\no
		&\qquad \leq\|m_1^{(|P|)}\|_{L_{\infty}}\|\bar{q}\|_{W^{n_{l-1},p_{l-1}}}\|v\|_{W^{n_{l},p_{l}}}
		\prod_{i=1}^{|P|}\|g_{\lambda(B_i)}\|_{W^{n_i,p_i}}%\cdots\|g_{\lambda(B_{l-2})}\|_{W^{n_{l-2},p_{l-2}}}
		\\\no
		&\qquad \leq C\|m_1^{(|P|)}\|_{L_{\infty}}   \|\bar{q}\|_{H^k}\|v\|_{H^k} \prod_{i=1}^{|P|} \|g_{\lambda(B_i)}\|_{H^k}%\cdots\|g_{\lambda(B_{l-2})}\|_{H^k(\mathbb R^{d-1})}
		.
		\end{align}
		
%	This proof required $|P|\neq 0$, $c_{1}+\cdots+c_{d}$ and $e_{1}+\cdots+e_{d}$ are all non-zero such that $l>1$ which will be needed in inequality \eqref{eq2.3.12}, if not, see following discussions:
	
	\textbf{Case 1.2:}
	If $c_{1}+\cdots+c_{d}=e_{1}+\cdots+e_{d}=0$ and $|P|\neq 0$, we use the Sobolev embedding $H^k(\mathbb{R}^d)\hookrightarrow L^{\infty}(\mathbb{R}^d)$ and Lemma \ref{soblev}(1), so that, similarly to Case 1.1, 
	\begin{align*}
	&\Big\|m_1^{(|P|)}\cdot\prod_{B\in P} \Big[\Big(\prod_{b\in B} \frac{\partial}{\partial{t_b}}\Big) g_{\lambda(B)} \Big]\cdot\frac{\partial^{c_1+\cdots+c_d} \bar{q}}{\partial{x_1^{c_1}}\cdots\partial{x_d^{c_d}}}\cdot\frac{\partial^{e_1+\cdots+e_d}v}{\partial{x_1^{e_1}}\cdots\partial{x_d^{e_d}}}\Big\|_{L^2}\\
	%&\qquad=\left\|m_1^{(|P|)}\cdot\prod_{B\in P} \left[\left( \prod_{b\in B} \frac{\partial}{\partial{t_b}}\right) g_{\lambda(B)} \right]\cdot \bar{q}v\right\|_{L^2}\\
	&
	\qquad\qquad\leq\|m_1^{(|P|)}\|_{L^{\infty}(\mathbb{R}^3)}\Big\|\prod_{B\in P} \Big[\Big( \prod_{b\in B}\frac{\partial}{\partial{t_b}}\Big) g_{\lambda(B)} \Big]\Big\|_{L^2}\|\bar{q}v\|_{L^{\infty}}\\
	%&\qquad\leq C\|m_1^{(|P|)}\|_{L^{\infty}}\left\|\prod_{B\in P} \left[ \left(\prod_{b\in B} \frac{\partial}{\partial{t_b}}\right) g_{\lambda(B)} \right]\right\|_{L^2}\|\bar{q}v\|_{H^k(\mathbb R^{d})}\\
	&\qquad\qquad\leq C\|m_1^{(|P|)}\|_{L^{\infty}(\mathbb{R}^3)}\Big\|\prod_{B\in P} \Big[ \Big(\prod_{b\in B} \frac{\partial}{\partial{t_b}}\Big) g_{\lambda(B)} \Big]\Big\|_{L^2}\|\bar{q}\|_{H^k}\|v\|_{H^k},
	\end{align*}
	We denote  $l=|P|$. When $|P|=1$, we use the inequality $\Big\|
	\dfrac{\partial^{k} } {\partial{x_1^{a_1}}\cdots\partial{x_d^{a_d}}}g_i\Big\|_{L^2}\leq \|u\|_{H^k} $, where $g_i$ is one of $(g_1,g_2,g_3)=(q,\bar{q},u)$, and thus obtain
	\begin{align*}
	\Big\|m_1^{(|P|)}\cdot
	\frac{\partial^{k} g_i }{\partial{x_1^{a_1}}\cdots\partial{x_d^{a_d}}}\cdot \bar{q}v\Big\|_{L^2}&\leq \|m_1^{(1)}\|_{L^{\infty}(\mathbb R^3)}\|g_i\|_{H^k}\|\bar{q}v\|_{H^k}\\ &\leq C\|m\|_{C^2(\mathbb{R}^2)}\|g_i\|_{H^k}\|\bar{q}\|_{H^k}\|v\|_{H^k}.
\end{align*}
   When $l\geq 2$, let $P=\{B_1,B_2,...,B_{l}\}$ and the $l$-tuple $(n_1,...,n_l)=(|B_1|,...,|B_l|)$. 
	We then use Lemmas~\ref{lem2.3.2} and \ref{lem2.3.3} with $\frac{1}{p_i}=(\frac{1}{2}-\frac{k}{d})\frac{1}{l}+\frac{n_i}{d}$, $i=1$ ,..., $l$, to obtain
	\begin{align*}
	\Big\|\prod_{B\in P} \Big[ \Big(\prod_{b\in B} \frac{\partial}{\partial{t_b}}\Big) g_{\lambda(B)} \Big]\Big\|_{L^2}&\leq
	\prod_{i\in 1}^{l} 
	\Big\|
	\Big(\prod_{b\in B_i} \frac{\partial}{\partial{t_b}}\Big) g_{\lambda(B_i)}\Big\|_{L^{p_i}}%\cdots\left\|
	%\left(\prod_{b\in B_l} \frac{\partial}{\partial{t_b}}\right) g_{\lambda(B_l)}\right\|_{L^{p_l}}
	\\&
	\leq \prod_{i\in 1}^{l} \|g_{\lambda(B_i)}\|_{W^{n_i,p_i}}\leq \prod_{i\in 1}^{l} \|g_{\lambda(B_i)}\|_{H^k},
	%\\&\qquad \leq\|g_{\lambda(B_1)}\|_{W^{n_1,p_1}}\cdots\|g_{\lambda(B_l)}\|_{W^{n_l,p_l}}\leq\|g_{\lambda(B_1)}\|_{H^k(\mathbb R^{d-1})}\cdots\|g_{\lambda(B_l)}\|_{H^k(\mathbb R^{d-1})},
	\end{align*}
	from which we conclude that
	\begin{align*}
	&\Big\|m_1^{(l)}(q,\bar{q},u)
	\prod_{B\in P} \Big[ \Big(\prod_{b\in B}\frac{\partial}{\partial{t_b}}\Big)
	 g_{\lambda(B)} \Big]\bar{q}v\Big\|_{L^2}\\
	& \qquad\qquad\qquad \qquad\qquad\qquad \leq \|m_1^{(l)}\|_{L^{\infty}(\mathbb R^{3})}
%	\|g_{\lambda(B_1)}\|_{H^k(\mathbb R^{d-1})}\cdots\|g_{\lambda(B_l)}\|_{H^k(\mathbb R^{d-1})}
	%\|\bar{q}v\|_{H^k(\mathbb R^{d})}
	\|\bar{q}\|_{H^k}\|v\|_{H^k}
	\prod_{i\in 1}^{l} \|g_{\lambda(B_i)}\|_{H^k}.
	%\\&\qquad \leq C\|m\|_{C^{l+1}}\|g_{\lambda(B_1)}\|_{H^k(\mathbb R^{d-1})}\cdots\|g_{\lambda(B_l)}\|_{H^k(\mathbb R^{d-1})}\|\bar{q}\|_{H^k(\mathbb R^{d-1})}\|v\|_{H^k(\mathbb R^{d})}.
	\end{align*}
	
	\textbf{Case 1.3:}
	If $b_{1}+\cdots+b_{d}=c_{1}+\cdots+c_{d}=0$ and $e_{1}+\cdots+e_{d}\neq 0$, we are evaluating the term
	$m_1(q,\bar{q},u)\bar{q}
	\dfrac{\partial^k v}{ \partial{x_1}^{a_1}\cdots\partial{x_d}^{a_d}}$ on $L^2(\mathbb{R}^d)$, based on  the %Sobolev 
	embedding $H^k\hookrightarrow L^{\infty}$:
	\begin{align*}
	\left\|m_1(q,\bar{q},u)\bar{q}
	\dfrac{\partial^kv}{\partial{x_1}^{a_1}\cdots\partial{x_d}^{a_d}}\right\|_{L^2
	}&\leq \|m_1\|_{L^{\infty}(\mathbb{R}^3)}\|\bar{q}\|_{L^{\infty}}\left\|\frac{\partial^k v}{ \partial{x_1}^{a_1}\cdots\partial{x_d}^{a_d}}\right\|_{L^2}\\& 
	\leq C\|m\|_{C^1(\mathbb{R}^2)}\|\bar{q}\|_{H^k}\|v\|_{H^k}.
	\end{align*}
	Similarly, if $b_{1}+\cdots+b_{d}=e_{1}+\cdots+e_{d}=0$ and $c_{1}+\cdots+c_{d}\neq 0$, we have
	\begin{align*}
	\left\|m_1(q,\bar{q},u)v
	\frac{\partial^k\bar{q}}{ \partial{x_1}^{a_1}\cdots\partial{x_d}^{a_d}}\right\|_{L^2}&%\leq \|m_1\|_{L^{\infty}}\|v\|_{L^{\infty}}\left\|\frac{\partial^k \bar{q}}{ \partial{x_1}^{a_1}\cdots\partial{x_d}^{a_d}}\right\|_{L^{2}}\\&
	\leq C\|m\|_{C^1(\mathbb{R}^2)}\|v\|_{H^k}\|\bar{q}\|_{H^k}.
	\end{align*}
	
	\textbf{Case 1.4:}
	If $b_{1}+\cdots+b_{d}=0$, but $c_{1}+\cdots+c_{d}\neq0$, and $e_{1}+\cdots+e_{d}\neq0$, then we set $$(n_1,n_2)=(b_{1}+\cdots+b_{d},e_{1}+\cdots+e_{d})$$ and apply Lemmas \ref{lem2.3.2} and \ref{lem2.3.3}  with  $\frac{1}{p_i}=(\frac{1}{2}-\frac{k}{d})\frac{1}{l}+\frac{n_i}{d}$, $i=1$, $2$,
	\begin{align*} &\left\|m_1(q,\bar{q},u)\cdot\frac{\partial^{c_{1}+\cdots+c_{d}}\bar{q} }{\partial{x_1^{c_{1}}}\cdots \partial{x_{d}^{c_{d}}}}\cdot\frac{\partial^{e_{1}+\cdots+e_{d}}v}{\partial{x_1^{e_{1}}}\cdots\partial{x_{d}^{e_{d}}}}
	\right\|_{L^2}
	\no\\
	&\qquad\qquad \qquad 
	\leq\|m_1\|_{L^{\infty}(\mathbb{R}^3)}\left\|\frac{\partial^{c_{1}+\cdots+c_{d}}\bar{q}}{\partial{x_1^{c_{1}}}\cdots\partial{x_{d}^{c_{d}}}}\right\|_{L^{p_1}}\left\|\frac{\partial^{e_{1}+\cdots+e_{d}}v}{\partial{x_1^{e_{1}}}\cdots\partial{x_{d}^{e_{d}}}}\right\|_{L^{p_2}}\no\\
	&\qquad\qquad \qquad \leq  C\|m\|_{C^1(\mathbb{R}^2)}\|\bar{q}\|_{W^{n_1,p_1}}\|v\|_{W^{n_2,p_2}}\leq  C\|m\|_{C^1(\mathbb{R}^2)}\|\bar{q}\|_{H^k}\|v\|_{H^k}.
	\end{align*}

	\textbf{Case 1.5:} 
	If $c_{1}+\cdots+c_{d}=0$, $|P|\neq 0$ and $e_{1}+\cdots+e_{d}\neq0$,
	we, similarly,  using Lemmas~\ref{lem2.3.2} and \ref{lem2.3.3} with  $\frac{1}{p_i}=(\frac{1}{2}-\frac{k}{d})\frac{1}{|P|+1}+\frac{n_i}{d}$, $i=1$, ..., $|P|+1$, obtain
	% using Lemmas~\ref{lem2.3.2} and \ref{lem2.3.3} with  $\frac{1}{p_i}=(\frac{1}{2}-\frac{k}{d})\frac{1}{l}+\frac{n_i}{d}$, $i=1$, ..., $l$, and the Sobolev Embedding $H^k\hookrightarrow L^{\infty}(\mathbb{R}^d)$, we obtain that
	%are evaluating  
	%$$m_1^{(|P|)}(q,\bar{q},u)\cdot\prod_{B\in P} \left[\left(\prod_{b\in B} \frac{\partial}{\partial{t_b}}\right) g_{\lambda(B)} \right]\cdot \bar{q}\cdot\frac{\partial^{e_{1}+\cdots+e_{d}}}{\partial{x_1^{e_{1}}}\cdots\partial{x_d^{e_{d}}}} v
	%$$ on $L^2(\mathbb{R}^{d})$. To do so we let $l=|P|+1$, $P=\{B_1,B_2,...,B_{l-1}\}$, the $l$-tuple $(n_1,...,n_l)=(|B_1|,...,|B_{l-1}|,e_{1}+\cdots+e_{d})$ and applying Lemma \ref{lem2.3.2} and Lemma \ref{lem2.3.3} with  $\frac{1}{p_i}=(\frac{1}{2}-\frac{k}{d})\frac{1}{l}+\frac{n_i}{d}$, $i=1$, ..., $l$, and the Sobolev Embedding $H^k(\mathbb{R}^d)\hookrightarrow L^{\infty}(\mathbb{R}^d)$, we obtain that
	\begin{align*}
&	\|m_1^{(|P|)}(q,\bar{q},u)\cdot\prod_{B\in P} \Big[ \Big(\prod_{b\in B}\frac{\partial}{\partial{t_b}}\Big) g_{\lambda(B)} \Big]\cdot\bar{q}\cdot\frac{\partial^{e_{1}+\cdots+e_{d}}}{\partial{x_1^{e_{1}}}\cdots\partial{x_d^{e_{d}}}} v\|_{L^2}\\\no
%	&\leq \|m_1^{(l-1)}\|_{L^{\infty}}\|\bar{q}\|_{L^{\infty}}\left\|\left(	\prod_{b\in B_1}\frac{\partial}{\partial{t_b}}\right) g_{\lambda(B_1)}\right\|_{L^{p_1}}\cdots%\\\no &\left\|\left(\prod_{b\in B_{l-1}} \frac{\partial}{\partial{t_b}}\right) g_{\lambda(B_{l-1})}\right\|_{L^{p_{l-1}}}\left\|\frac{\partial^{e_{1}+\cdots+e_{d}}v}{\partial{x_1^{e_{1}}}\cdots\partial{x_d^{e_{d}}}} \right\|_{L^{p_l}}\\\no
	%&\leq C\|m\|_{C^{l}}\|\bar{q}\|_{H^k(\mathbb R^{d-1})}\|g_{\lambda(B_1)}\|_{W^{n_1,p_1}}\cdots\|g_{\lambda(B_{l-1})}\|_{W^{n_{l-1},p_{l-1}}}\|v\|_{W^{n_l,p_l}}\\\no
	& \qquad
	\leq  C\|m\|_{C^{|P|+1}(\mathbb{R}^2)}\|\bar{q}\|_{H^k}\|v\|_{H^k}  \prod_{i=1}^{|P|}\|g_{\lambda(B_i)}\|_{H^k},%\cdots\|g_{\lambda(B_{l-1})}\|_{H^k(\mathbb R^{d-1})}.
	\end{align*}
and, 	if $e_{1}+\cdots+e_{d}=0$, $|P|\neq 0$ and $c_{1}+\cdots+c_{d}\neq0$, %we let $l=|P|+1$, $P=\{B_1,B_2,...,B_{l-1}\}$, the $l$-tuple $(n_1,...,n_l)=(|B_1|,...,|B_{l-1}|,c_{1}+\cdots+c_{d})$ and applying Lemma \ref{lem2.3.2} and Lemma \ref{lem2.3.3} according to $\frac{1}{p_i}=(\frac{1}{2}-\frac{k}{d})\frac{1}{l}+\frac{n_i}{d}$, $i=1$,...,$l$, and the Sobolev Embedding $H^k(\mathbb{R}^d)\hookrightarrow L^{\infty}(\mathbb{R}^d)$
we obtain
	\begin{align}\label{eq2.3.15}
	&\|m_1^{(|P|)}(q,\bar{q},u)\cdot\prod_{B\in P}\Big[ \Big(\prod_{b\in B}\frac{\partial}{\partial{t_b}}\Big) g_{\lambda(B)} \Big]\cdot\frac{\partial^{c_{1}+\cdots+c_{d}}}{\partial{x_1^{c_{1}}}\cdots\partial{x_d^{c_{d}}}}\bar{q}\cdot v\|_{L^2}\\\no
	%&\leq\|m_1^{(l-1)}\|_{L^{\infty}}\|v\|_{L^{\infty}}\left\|\left(
	%\prod_{b\in B_1} \frac{\partial}{\partial{t_b}}\right) g_{\lambda(B_1)}\right\|_{L^{p_1}}\cdots\left\|\left(\prod_{b\in B_{l-1}} \frac{\partial}{\partial{t_b}}\right) g_{\lambda(B_{l-1})}\right\|_{L^{p_{l-1}}}\left\|\frac{\partial^{c_{1}+\cdots+c_{d}}}{\partial{x_1^{c_{1}}}\cdots\partial{x_d^{c_{d}}}} \bar{q}\right\|_{L^{p_l}}\\
	%\no\leq& C\|m\|_{C^{l}}\|v\|_{H^k(\mathbb R^{d})}\|g_{\lambda(B_1)}\|_{W^{n_1,p_1}}\cdots\|g_{\lambda(B_{l-1})}\|_{W^{n_{l-1},p_{l-1}}}\|\bar{q}\|_{W^{n_l,p_l}}\\
	%\no
	&\leq C\|m\|_{C^{|P|+1}(\mathbb{R}^2)}\|v\|_{H^k}%\|g_{\lambda(B_1)}\|_{H^k(\mathbb R^{d-1})}\cdots\|g_{\lambda(B_{l-1})}\|_{H^k(\mathbb R^{d-1})}
	\|\bar{q}\|_{H^k}. 
	 \prod_{i=1}^{|P|}\|g_{\lambda(B_i)}\|_{H^k}.
	\end{align}
	Since $|P|\leq k$, the inequalities \eqref{eq2.3.9}-\eqref{eq2.3.15}  imply
	\begin{align*}
	&\Big\|m_1^{(|P|)}(q,\bar{q},u)\cdot\prod_{B\in P} \Big[\Big(\prod_{b\in B} \frac{\partial}{\partial{t_b}}\Big) g_{\lambda(B)} \Big]\cdot\frac{\partial^{c_{1}+\cdots+c_{d}}\bar{q}}{\partial{x_1^{c_{1}}}\cdots\partial{x_d^{c_{d}}}}\cdot\frac{\partial^{e_{1}+\cdots+e_{d}}v}{\partial{x_1^{e_{1}}}\cdots\partial{x_d^{e_{d}}}} \Big\|_{L^2}\\\no
%&	\leq\|m_1\|_{C^{|P|}}\|g_{\lambda(B_1)}\|_{H^k(\mathbb R^{d-1})}\cdots\|g_{\lambda(B_{l-1})}\|_{H^k(\mathbb R^{d-1})}\|\bar{q}\|_{H^k(\mathbb R^{d-1})}\|v\|_{H^k(\mathbb R^{d})}\\\no
&	\qquad \leq \|m\|_{C^{k+1}(\mathbb{R}^2)} \|\bar{q}\|_{H^k}\|v\|_{H^k}
\prod_{i=1}^{|P|}\|g_{\lambda(B_i)}\|_{H^k}.
%\cdots\|g_{\lambda(B_{l-1})}\|_{H^k(\mathbb R^{d-1})}
	\end{align*}
	Similarly, %let $l=|P|+2$, $P=\{B_1,...,B_{l-2}\}$, the $l$-tuple $(n_1,...,n_l)=(|B_1|,...,\\ |B_{l-2}|,c_{1}+\cdots+c_{d},e_{1}+\cdots+e_{d})$ and let $\gamma_{\alpha}=e^{\alpha  x_1}$, then
	\begin{align*}
	&\Big\|\gamma_{\alpha}\Big(m_1^{(|P|)}(q,\bar{q},u)\cdot\prod_{B\in P} \Big[ \Big(\prod_{b\in B}\frac{\partial}{\partial{t_b}}\Big) g_{\lambda(B)} \Big]\cdot\frac{\partial^{c_{1}+\cdots+c_{d}}\bar{q}}{\partial{x_1^{c_{1}}}\cdots\partial{x_d^{c_{d}}}} \frac{\partial^{e_{1}+\cdots+e_{d}}v}{\partial{x_1^{e_{1}}}\cdots\partial{x_d^{e_{d}}}}\Big)\Big\|_{L^2}\\\no
	&\qquad\qquad \qquad\qquad\leq\|m_1^{(|P|)}\|_{L^{\infty}(\mathbb R^{3})}\prod_{i=1}^{|P|}\Big\|\Big(\prod_{b\in B_i}\frac{\partial}{\partial{t_b}}\Big) g_{\lambda(B_i)}\Big\|_{L^{p_i}}%\cdots\left\|	\left(\prod_{b\in B_{l-2}} \frac{\partial}{\partial{t_b}}\right) g_{\lambda(B_{l-2})}\right\|_{L^{p_{|P|}}}
	\\\no&\qquad\qquad\qquad\qquad\qquad\qquad\qquad \times
	\Big\|\frac{\partial^{c_{1}+\cdots+c_{d}}\bar{q}}{\partial{x_1^{c_{1}}}\cdots\partial{x_d^{c_{d}}}}\Big\|_{L^{p_{|P|+1}}}\Big\|\gamma_{\alpha}\frac{\partial^{e_{1}+\cdots+e_{d}}v}{\partial{x_1^{e_{1}}}\cdots\partial{x_d^{e_{d}}}}\Big\|_{L^{p_{|P|+2}}}\\\no
	&\qquad\qquad\qquad\qquad\leq \|m_1\|_{C^{|P|}(\mathbb{R}^3)} \|\gamma_{\alpha}v\|_{W^{n_{|P|+2},p_{|P|+2}}} \prod_{i=1}^{|P|}  \|g_{\lambda(B_i)}\|_{W^{n_i,p_i}}
	%\cdots\|g_{\lambda(B_{l-2})}\|_{W^{n_{l-2},p_{l-2}}}\|\bar{q}\|_{W^{n_{l-1},p_{l-1}}}
	\\\no
	&\qquad\qquad\qquad\qquad \leq\|m\|_{C^{|P|+1}(\mathbb{R}^2)}   \|\bar{q}\|_{H^k}\|\gamma_{\alpha}v\|_{H^k} \prod_{i=1}^{|P|}  \|g_{\lambda(B_i)}\|_{H^k}%\cdots\|g_{\lambda(B_{l-2})}\|_{H^k(\mathbb R^{d-1})}.
	\end{align*}
	
	The case  $|P|=0$, $c_{1}+\cdots+c_{d}=0$ or $e_{1}+\cdots+e_{d}=0$ can be treated analogously.
	
	For variations in $u$,   the representation
	$m( q, u + \bar{u})-m(q, u) =\bar{u} \int_{0}^{1}m_u(q, u + t\bar{u}) dt$  yields
	$$\|m(q,u+\bar{u})v-m(q,u)v\|_{L^2}\leq\|m\|_{C^1(\mathbb{R}^2)}\|\bar{u}\|_{L^{\infty}}\|v\|_{L^2},$$
	which, by the Sobolev embedding $H^k(\mathbb{R}^d)\hookrightarrow L^{\infty}(\mathbb{R}^d)$, implies
	$$\|m(q,u+\bar{u})v-m(q,u)v\|_{L^2}\leq\|m\|_{C^1(\mathbb{R}^2)}\|\bar{u}\|_{H^k}\|v\|_{L^2},$$
	and 
	$$\|\gamma_{\alpha}\left(m(q,u+\bar{u})v-m(q,u)v\right)\|_{L^2}\leq\|m\|_{C^1(\mathbb{R}^2)}\|\bar{u}\|_{H^k}\|\gamma_{\alpha}v\|_{L^2}.$$
	Let $m_2(q,u,\bar{u})=\int_{0}^{1}m_u(q,u+t\bar{u}) dt$, $g_1(x)=q(x)$, $g_2(x)=u(x)$, and $g_3(x)=\bar{u}(x)$.  For $\sum\limits_{t=1}^{d}a_t=k$, we then have% Again, by generalizing the Leibniz rule  for $\sum\limits_{t=1}^{d}a_t=k$ we obtain that:
	\begin{align*}
	&\frac{\partial^k}{\partial{x_1^{a_1}}\cdots\partial{x_d^{a_d}}}m_2(g_1,g_2,g_3)\bar{u}v%\\\no&\qquad
	=\frac{\partial^{a_1+\cdots+a_{d-1}}}{\partial{x_1^{a_1}}\cdots\partial{x_{d-1}^{a_{d-1}}}}\sum_{b_d+c_d+e_d=a_d}{a_d\choose b_d,c_d,e_d}\frac{\partial^{b_d} m_2}{\partial{x_d^{b_d}}}
	%(g_1,g_2,g_3)
	\frac{\partial^{c_d}\bar{u}}{\partial{x_d^{c_d}}}\frac{\partial^{e_d}v}{\partial{x_d^{e_d}}}\\\no
	&=\cdots\\\no
	&=\sum_{b_d+c_d+e_d=a_d}{a_d\choose b_d,c_d,e_d}\cdots\sum_{b_1+c_1+e_1=a_1}{a_1\choose b_1,c_1,e_1}
	%\\\no&\hskip1.5cm
	\frac{\partial^{b_1+\cdots+b_d}m_2}{\partial{x_1^{b_1}}\cdots\partial{x_d^{b_d}}}%(g_1,g_2,g_3)
	\frac{\partial^{c_1+\cdots+c_d}\bar{u}}{\partial{x_1^{c_1}}\cdots\partial{x_d^{c_d}}}\frac{\partial^{e_1+\cdots+e_d}v}{\partial{x_1^{e_1}}\cdots\partial{x_d^{e_d}}}.
	\end{align*}
	The same  argument as the one  that lead to  \eqref{A4} implies % Let $s=b_1+\cdots+b_d$. We again employ the Higher Chain Formula.  For each $i$ in the set $J_s$ of integers $1$, $2$, ..., $s$, let $t_i$ denote one of the independent variables $x_1$, ..., $x_d$. A partition of $J_s$ is a family of pairwise disjoint nonempty subsets of $J_s$ whose union is $J_s$. Sets in a partition are called blocks. A block function is to assign a label to each block of a partition. The set of all block functions from a partition $P$ of $J_s$ into $J_3$ is	denoted by $P_3$. The set of all partitions of $J_s$ is denoted by $P_s$, then
	\begin{equation*}
	\frac{\partial^s m_2(g(x))}{\partial{t_1}\cdots\partial{t_s}}
	=\sum_{P\in P_s}\sum_{\lambda\in P_{3}}\Big\{ \Big( \prod_{B\in P}\frac{\partial}{\partial{g_{\lambda(B)}}}\Big)m_2   \Big\}
	\Big\{   \prod_{B\in P} \Big[\Big( \prod_{b\in B} \frac{\partial}{\partial{t_b}}\Big) g_{\lambda(B)} \Big] \Big\}.
	\end{equation*}
	
	For fixed multinomial coefficients ${a_i\choose b_i,c_i,e_i}$ ($i=1$, ...,$d$), $P\in P_s$, and $\lambda\in P_3$, we shall find a bound on estimate the term 
	$$\Big\{\Big( \prod_{B\in P}\frac{\partial}{\partial{g_{\lambda(B)}}}\Big)m_2\Big\} \Big\{   \prod_{B\in P} \Big[\Big(\prod_{b\in B} \frac{\partial}{\partial{t_b}}Big) g_{\lambda(B)} \Big]  \Big\}\frac{\partial^{c_1+\cdots+c_d}\bar{u}
	}{\partial{x_1^{c_1}}\cdots\partial{x_d^{c_d}}}\frac{\partial^{e_1+\cdots+e_d}v}{\partial{x_1^{e_1}}\cdots\partial{x_d^{e_d}}}$$ 
	in both $L^2(\mathbb{R}^d)$ and $L^2_{\alpha}(\mathbb{R})\otimes L^2(\mathbb{R}^{d-1})$. In order to do that we consider the following cases.
	
	\textbf{Case 2.1:}
	If $|P|$, $c_1+\cdots+c_d$, $e_1+\cdots+e_d>0$,  we denote  $l=|P|+2$ ($2<l\leq k$) and $P=\{B_1,...,B_{l-2}\}$,  and use Lemmas~\ref{lem2.3.2} and  \ref{lem2.3.3}  with  $$(n_1,...,n_l)=\{|B_1|,...,|B_{l-2}|, c_1+\cdots+c_d, e_1+\cdots+e_d\}$$ and $\frac{1}{p_i}=(\frac{1}{2}-\frac{k}{d})\frac{1}{l}+\frac{n_i}{d}$ for $i=1$, $2$,...,$l$, to obtain the inequality:
	\begin{align}\label{eq2.3.11}
	&\Big\|\Big\{
\Big( \prod_{B\in P}\frac{\partial}{\partial{g_{\lambda(B)}}}\Big)m_2\Big\}\Big\{   \prod_{B\in P} \Big[ \Big(\prod_{b\in B} \frac{\partial}{\partial{t_b}}\Big) y_{\lambda(B)} \Big]  \Big\}\frac{\partial^{c_1+\cdots+c_d}\bar{u}}{\partial{x_1^{c_1}}\cdots\partial{x_d^{c_d}}}\frac{\partial^{e_1+\cdots+e_d}v}{\partial{x_1^{e_1}}\cdots\partial{x_d^{e_d}}}\Big\|_{L^2}
	\\\no&
	\leq \|m_2\|_{C^{l-2}(\mathbb{R}^3)}\Big\|\frac{\partial^{c_1+\cdots+c_d}\bar{u}}{\partial{x_1^{c_1}}\cdots\partial{x_d^{c_d}}}\Big\|_{L^{p_{l-1}}}\Big\|\frac{\partial^{e_1+\cdots+e_d}v}{\partial{x_1^{e_1}}\cdots\partial{x_d^{e_d}}}\Big\|_{L^{p_l}}
	%\\\no&
	\prod_{i=1}^{ l-2}\Big\|\Big(\prod_{b\in B_i} \frac{\partial}{\partial{t_b}}\Big) g_{\lambda(B_i)}\Big\|_{L^{p_i}}
	%\cdots\left\|\left(\prod_{b\in B_{l-2}} \frac{\partial}{\partial{t_b}}\right) g_{\lambda(B_{l-2})}\Big\|_{L^{p_{l-2}}(\mathbb R^{d-1})}
	%%\\\no&\hskip1.5cm
	\\\no
	&\leq\|m\|_{C^{l-1}(\mathbb{R}^2)}\|g_{\lambda(B_1)}\|_{W^{n_1,p_1}}\cdots\|g_{\lambda(B_{l-2})}\|_{W^{n_{l-2},p_{l-2}}}\|\bar{u}\|_{W^{n_{l-1},p_{l-1}}}\|v\|_{W^{n_l,p_l}}\\\no
	&\leq C\|m\|_{C^{l-1}(\mathbb{R}^2)}\|g_{\lambda(B_1)}\|_{H^k}\cdots\|g_{\lambda(B_{l-2})}\|_{H^k}\|\bar{u}\|_{H^k}\|v\|_{H^k}.
	\end{align}
	When not all $|P|$, $c_1+\cdots+c_d$, $e_1+\cdots+e_d$ are positive, we have the following cases:
	
	\textbf{Case 2.2:}
	If $|P|=0$ and $c_1+\cdots+c_d,e_1+\cdots+e_d>0$, we apply Lemmas~\ref{lem2.3.2} and \ref{lem2.3.3} with  $l=2$, $(n_1,n_2)=(c_1+\cdots+c_d,e_1+\cdots+e_d)$ and $\frac{1}{p_i}=(\frac{1}{2}-\frac{k}{d})\frac{1}{l}+\frac{n_i}{d}$ so that $H^k(\mathbb{R}^d)\hookrightarrow W^{n_i,p_i}(\mathbb{R}^d)$ for $i=1$, $2$, and obtain
	\begin{align*}\no
	&\left\|m_2%(g_1,g_2,g_3)
	\frac{\partial^{c_1+\cdots+c_d}\bar{u}}{\partial{x_1^{c_1}}\cdots\partial{x_d^{c_d}}}\frac{\partial^{e_1+\cdots+e_d}v}{\partial{x_1^{e_1}}\cdots\partial{x_d^{e_d}}}\right\|_{L^2}
	\\\no&\qquad\qquad \qquad 
	%\leq\|m_2\|_{L^{\infty}}\left\|\frac{\partial^{c_1+\cdots+c_d}\bar{u}}{\partial{x_1^{c_1}}\cdots\partial{x_d^{c_d}}}\frac{\partial^{e_1+\cdots+e_d}v}{\partial{x_1^{e_1}}\cdots\partial{x_d^{e_d}}}\right\|_{L^2}\\\no&\qquad
	\leq\|m\|_{C^1(\mathbb{R}^2)}\left\|\frac{\partial^{c_1+\cdots+c_d}\bar{u}}{\partial{x_1^{c_1}}\cdots\partial{x_d^{c_d}}}\right\|_{L^{p_1}}\left\|\frac{\partial^{e_1+\cdots+e_d}v}{\partial{x_1^{e_1}}\cdots\partial{x_d^{e_d}}}\right\|_{L^{p_2}}\no
	\\&\qquad\qquad \qquad 
	\leq\|m\|_{C^1(\mathbb{R}^2)}\|\bar{u}\|_{W^{n_1,p_1}}\|v\|_{W^{n_2,p_2}}\leq\|m\|_{C^1(\mathbb{R}^2)}\|\bar{u}\|_{H^k}\|v\|_{H^k}.
	\end{align*}
	
	\textbf{Case 2.3:}
	If $e_1+\cdots+e_d=0$, and $|P|$, $c_1+\cdots+c_d>0$, let $l=|P|+1$ ($2\leq l\leq k$) and $P=\{ B_1,...,B_{l-1} \}$, then from Lemmas~\ref{lem2.3.2} and \ref{lem2.3.3} with  $(n_1,...,n_l)=(|B_1|,...,|B_{l-1}|,c_1+\cdots+c_d)$ and $\frac{1}{p_i}=(\frac{1}{2}-\frac{k}{d})\frac{1}{l}+\frac{n_i}{d}$, $i=1$, ..., $l$,  and the Sobolev embedding $H^k(\mathbb{R}^d)\hookrightarrow L^{\infty}(\mathbb{R}^d)$, we obtain
	\begin{align*}
	&\Big\|\Big\{\Big( \prod_{B\in P}\frac{\partial}{\partial{g_{\lambda(B)}}}\Big)m_2 \Big\}\Big\{   \prod_{B\in P} \Big[ \Big(\prod_{b\in B} \frac{\partial}{\partial{t_b}}\Big) g_{\lambda(B)} \Big] \Big\}
	\Big(\frac{\partial^{c_1+\cdots+c_d}}{\partial{x_1^{c_1}}\cdots\partial{x_d^{c_d}}}\bar{u}\Big)v\Big\|_{L^2}\\\no
	&\qquad \qquad \leq\|m_2\|_{C^{l-1}(\mathbb{R}^3)}\|v\|_{L^{\infty}}
	\Big\|\Big\{   \prod_{B\in P} \Big[\Big(\prod_{b\in B}\frac{\partial}{\partial{t_b}}\Big) g_{\lambda(B)} \Big] \Big\}
\frac{\partial^{c_1+\cdots+c_d}\bar{u}}{\partial{x_1^{c_1}}\cdots\partial{x_d^{c_d}}}\Big\|_{L^2}
	\\\no&\qquad \qquad 
	\leq C\|m_2\|_{C^{l-1}(\mathbb{R}^3)}\|v\|_{H^k}
	%\\\no& \qquad
	\prod_{i=1}^{|P|}
	\Big
	\|\Big(\prod_{b\in B_{i}} \frac{\partial}{\partial{t_b}}\Big) g_{\lambda(B_{i})}  \Big\|_{L^{p_{i}}}
	%\cdots\Big\|\Big(\prod_{b\in B_{|P|+1}} \frac{\partial}{\partial{t_b}}\Big) g_{\lambda(B_{|P|+1})}\Big\|_{L^{p_{l-1}}}
	%\\\no&\hskip1.5cm
	\Big\|\frac{\partial^{c_1+\cdots+c_d}\bar{u}}{\partial{x_1^{c_1}}\cdots\partial{x_d^{c_d}}}\Big\|_{L^{p_{l}}}
	\\
	\no
	&\qquad \qquad \leq C\|m_2\|_{C^{l-1}(\mathbb{R}^3)}  \|\bar{u}\|_{W^{n_{l},p_{l}}}\|v\|_{H^k} \prod_{i=1}^{|P|} \|g_{\lambda(B_i)}\|_{W^{n_i,p_i}}%\cdots\|g_{\lambda(B_{l-1})}\|_{W^{n_{l-1},p_{l-1}}}
	\\\no
	&\qquad \qquad 
	\leq C\|m\|_{C^{l}(\mathbb{R}^2)}\|v\|_{H^k}\prod_{i=1}^{|P|}\|g_{\lambda(B_i)}\|_{H^k}%\cdots\|g_{\lambda(B_{l-1})}\|_{H^k(\mathbb R^{d-1})}
	\|\bar{u}\|_{H^k}.
	\end{align*}

	Similarly, if $c_1+\cdots+c_d=0$ and $|P|$, $e_1+\cdots+e_d>0$, we have
	\begin{align*}
	&\Big\|\Big\{\Big( \prod_{B\in P}\frac{\partial}{\partial{g_{\lambda(B)}}}\Big)m_2 \Big\} \Big\{  \prod_{B\in P}\Big[ \Big(\prod_{b\in B} \frac{\partial}{\partial{t_b}}\Big) g_{\lambda(B)} \Big]  \Big\}\Big(\frac{\partial^{e_1+\cdots+e_d}}{\partial{x_1^{e_1}}\cdots \partial{x_d^{e_d}}}v\Big) \bar{u}\Big\|_{L^2}
	\\\no
	& \qquad \qquad  \leq\|m_2\|_{C^{l-1}(\mathbb{R}^3)}\|\bar{u}\|_{L^{\infty}}\Big\| \Big\{   \prod_{B\in P} \Big[\Big( \prod_{b\in B} \frac{\partial}{\partial{t_b}}\Big) g_{\lambda(B)} \Big] \Big\}\frac{\partial^{e_1+\cdots+e_d}}{\partial{x_1^{e_1}}\cdots\partial{x_d^{e_d}}}v\Big\|_{L^2}
	\\\no
	&\qquad \qquad  \leq  C\|m_2\|_{C^{l-1}(\mathbb{R}^3)} \|\bar{u}\|_{H^k} \Big\|\frac{\partial^{e_1+\cdots+e_d}v}{\partial{x_1^{e_1}}\cdots\partial{x_d^{e_d}}}\Big\|_{L^{p_{l}}}
	\prod_{i=1}^{l-1} \Big\|\Big(\prod_{b\in B_{i}} \frac{\partial}{\partial{t_b}}\Big) g_{\lambda(B_{i})}\Big\|_{L^{p_{i}}}
	%\cdots\left\|\Big(\prod_{b\in B_{l-1}}\frac{\partial}{\partial{t_b}}\Big) g_{\lambda(B_{l-1})}\Big\|_{L^{p_{l-1}}}
	%\Big\|\frac{\partial^{e_1+\cdots+e_d}v}{\partial{x_1^{e_1}}\cdots\partial{x_d^{e_d}}}\Big\|_{L^{p_{l}}}
	\\\no
	&\qquad \qquad  \leq  C\|m_2\|_{C^{l-1}(\mathbb{R}^3)}\|\bar{u}\|_{H^k} \|v\|_{W^{n_{l},p_{l}}} \prod_{i=1}^{l-1}\|g_{\lambda(B_i)}\|_{W^{n_i,p_i}}%\cdots\|g_{\lambda(B_{l-1})}\|_{W^{n_{l-1},p_{l-1}}}
	\\\no
	& \qquad \qquad \leq C\|m\|_{C^{l}(\mathbb{R}^2)}\|\bar{u}\|_{H^k}\|v\|_{H^k}\prod_{i=1}^{l-1} \|g_{\lambda(B_i)}\|_{H^k}.%\cdots\|g_{\lambda(B_{l-1})}\|_{H^k(\mathbb R^{d-1})}\|v\|_{H^k(\mathbb R^{d})}.
	\end{align*}

	\textbf{Case 2.4:}
	If $|P|=e_1+\cdots+e_d=0$ and $c_1+\cdots+c_d\neq0$, we use the Sobolev embedding $H^k(\mathbb{R}^d)\hookrightarrow L^{\infty}(\mathbb{R}^d)$, we obtain
	\begin{align*}
	\Big\|m_2(g_1,g_2,g_3)\Big(\frac{\partial^{a_1+\cdots+a_d}\bar{u}}{\partial{x_1^{a_1}}\cdots\partial{x_d^{a_d}}}\Big) v\Big\|_{L^2}&\leq  \|m_2\|_{L^{\infty}(\mathbb{R}^3)}\Big\|\frac{\partial^{a_1+\cdots+a_d}\bar{u}}{\partial{x_1^{a_1}}\cdots\partial{x_d^{a_d}}}\Big\|_{L^2}\|v\|_{L^{\infty}}
	\\\no&
	\leq C\|m\|_{C^1(\mathbb{R}^2)}\|\bar{u}\|_{H^k}\|v\|_{H^k},
	\end{align*}
	or,  for $|P|=c_1+\cdots+c_d=0$ and $e_1+\cdots+e_d\neq0$, 
	\begin{align*}
	\Big\|m_2(g_1,g_2,g_3)\bar{u}\frac{\partial^{a_1+\cdots+a_d}v}{\partial{x_1^{a_1}}\cdots\partial{x_d^{a_d}}}\Big\|_{L^2}&\leq \|m_2\|_{L^{\infty}(\mathbb{R}^3)}\|\bar{u}\|_{L^{\infty}}\Big\|\frac{\partial^{a_1+\cdots+a_d}v}{\partial{x_1^{a_1}}\cdots\partial{x_d^{a_d}}}\Big\|_{L^2}
	\\\no&
	\leq C\|m\|_{C^{1}(\mathbb{R}^2)}\|\bar{u}\|_{H^k}\|v\|_{H^k}.
	\end{align*}
	
	\textbf{Case 2.5:}
	If $c_1+\cdots+c_d=e_1+\cdots+e_d=0$ and $|P|\neq 0$, we  denote $l=|P|$ ($1\leq l\leq k$) and $P=\{ B_1,...,B_l \}$: when $l=1$, $g_{\lambda(B)}=g_i$, $i=1$, $2$, $3$. We then use the inequality 
	$$\Big\|\frac{\partial^{k}g_i}{\partial{x_1^{a_1}}\cdots\partial{x_d^{a_d}} }\Big\|_{L^2}\leq \|g_i\|_{H^k},$$
	  Sobolev embedding $H^k(\mathbb{R}^d)\hookrightarrow L^{\infty}(\mathbb{R}^d)$ and Lemma \ref{soblev}(1) to obtain for $i=1$, $2$, $3$,
	\begin{align*}
	&\left\|m_2^{(1)}\left(\frac{\partial^{k} }{\partial{x_1^{a_1}}\cdots\partial{x_d^{a_d}} }g_i\right) \bar{u} v\right\|_{L^2}\\&\qquad\qquad  \leq \|m\|_{C^1(\mathbb{R}^2)}\|g_i\|_{H^k}\|\bar{u} v\|_{L^{\infty}}\leq\|m\|_{C^1(\mathbb{R}^2)}\|g_i\|_{H^k}\|\bar{u} v\|_{H^k}\no\\
	&\qquad\qquad \leq\|m\|_{C^1(\mathbb{R}^2)}\|g_i\|_{H^k}\|\bar{u}\|_{H^k}\|v\|_{H^k}.
	\end{align*}
	When $l\geq2$, we use the Sobolev embedding $H^k(\mathbb{R})\hookrightarrow L^{\infty}(\mathbb{R})$, Lemma \ref{lem2.3.2} and Lemma \ref{lem2.3.3} with  $(n_1,...,n_l)=(|B_1|,...,|B_{l}|)$ and $\frac{1}{p_i}=(\frac{1}{2}-\frac{k}{d})\frac{1}{l}+\frac{n_i}{d}$ for $i=1,...,l$ and Lemma \ref{soblev}(1) to obtain
	\begin{align}\label{eq2.3.18}
	&\Big\|\Big\{\Big( \prod_{B\in P}\frac{\partial}{\partial{g_{\lambda(B)}}}\Big)m_2 \Big\}\Big\{   \prod_{B\in P} \Big[\Big( \prod_{b\in B} \frac{\partial}{\partial{t_b}}\Big) g_{\lambda(B)} \Big]  \Big\} \bar{u}v\Big\|_{L^2}
	\\\no&\qquad
	\leq\|m_2\|_{C^l(\mathbb{R}^3)}\Big\| \Big\{   \prod_{B\in P} \Big[ \Big(\prod_{b\in B} \frac{\partial}{\partial{t_b}}\Big) g_{\lambda(B)} \Big]  \Big\}\Big\|_{L^2}\|\bar{u}  v\|_{L^{\infty}}
	\\\no&\qquad
	\leq C\|m_2\|_{C^l(\mathbb{R}^3)} 	\|\bar{u}v\|_{H^k}
	\prod_{i=1}^{l} \Big\|\Big(\prod_{b\in B_i} \frac{\partial}{\partial{t_b}}\Big) g_{\lambda(B_i)}\Big\|_{L^{p_i}}%\cdots\Big\|\Big(\prod_{b\in B_l} \frac{\partial}{\partial{t_b}}\Big) g_{\lambda(B_l)}\Big\|_{L^{p_l}}
\\\no&\qquad
	\leq C\|m_2\|_{C^l(\mathbb{R}^3)} \|\bar{u}\|_{H^k}\|v\|_{H^k}
	\prod_{i=1}^{l}\|g_{\lambda(B_i)}\|_{W^{n_i,p_i}}%\cdots\|g_{\lambda(B_{l})}\|_{W^{n_{l},p_{l}}}
	\\\no&\qquad
	\leq C\|m\|_{C^{l+1}(\mathbb{R}^2)}
	|\bar{u}\|_{H^k}\|v\|_{H^k}
\prod_{i=1}^{l}\|g_{\lambda(B_1)}\|_{H^k}.%\cdots\|g_{\lambda(B_{l})}\|_{H^k(\mathbb R^{d-1})}
	%\|\bar{u}\|_{H^k(\mathbb R^{d-1})}\|v\|_{H^k(\mathbb R^{d-1})}
	\end{align}

	Combining above inequalities \eqref{eq2.3.11}-\eqref{eq2.3.18}, we can conclude that:
	\begin{align*}
	&\Big\|\Big\{\Big( \prod_{B\in P}\frac{\partial}{\partial{g_{\lambda(B)}}}\Big)m_2\Big\} \Big\{   \prod_{B\in P} \Big[ \Big(\prod_{b\in B} \frac{\partial}{\partial{t_b}}\Big) g_{\lambda(B)} \Big]  \Big\}\frac{\partial^{c_1+\cdots+c_d}\bar{u}}{\partial{x_1^{c_1}}\cdots\partial{x_d^{c_d}}}\frac{\partial^{e_1+\cdots+e_d}v}{\partial{x_1^{e_1}}\cdots\partial{x_d^{e_d}}}\Big\|_{L^2}
	\\\no
	&\qquad
	\leq\|m\|_{C^{k+1}(\mathbb{R}^2)}\|\bar{u}\|_{H^k}\|v\|_{H^k}
	\prod_{i=1}^{|P|}
	\|g_{\lambda(B_1)}\|_{H^k}%\cdots\|g_{\lambda(B_{|P|})}\|_{H^k(\mathbb R^{d-1})}
	.
	\end{align*}
	Similarly, we let $l=|P|+2$ ($2<l\leq k$), $P=\{B_1,...,B_{l-2}\}$ and $\gamma_{\alpha}=e^{\alpha x_1}$, and use Lemmas~\ref{lem2.3.2} and  \ref{lem2.3.3} with $(n_1,...,n_l)=\{|B_1|,...,|B_{l-2}|,c_1+\cdots+c_d,e_1+\cdots+e_d\}$ and $\frac{1}{p_i}=(\frac{1}{2}-\frac{k}{d})\frac{1}{l}+\frac{n_i}{d}$ for $i=1$, $2$,..., $l$, to obtain 	
	\begin{align*}
	&\Big\|\gamma_{\alpha}m_2^{(|P|)}\Big\{   \prod_{B\in P} \Big[ \Big(\prod_{b\in B} \frac{\partial}{\partial{t_b}}\Big) g_{\lambda(B)} \Big]  \Big\}\frac{\partial^{c_1+\cdots+c_d}\bar{u}}{\partial{x_1^{c_1}}\cdots\partial{x_d^{c_d}}}\frac{\partial^{e_1+\cdots+e_d}v}{\partial{x_1^{e_1}}\cdots\partial{x_d^{e_d}}}\Big\|_{L^2}\\\no
	%&\qquad \leq \|m\|_{C^{k+1}}\Big\|\frac{\partial^{c_1+\cdots+c_d}\bar{u}}{\partial{x_1^{c_1}}\cdots\partial{x_d^{c_d}}}\Big\|_{L^{p_{l-1}}(\mathbb R^{d-1})}
	%\\\no&
	%\Big\|\gamma_{\alpha}\frac{\partial^{e_1+\cdots+e_d}v}{\partial{x_1^{e_1}}\cdots\partial{x_d^{e_d}}}\Big\|_{L^{p_l}}
	%\prod_{i=1}^{|P|}\Big\|\Big(\prod_{b\in B_i} \frac{\partial}{\partial{t_b}}\Big) g_{\lambda(B_i)}\Big\|_{L^{p_i}(\mathbb R^{d-1})}
	%%\cdots\Big\|\Big(\prod_{b\in B_{l-2}}\frac{\partial}{\partial{t_b}}\Big) g_{\lambda(B_{l-2})}\Big\|_{L^{p_{l-2}}(\mathbb R^{d-1})}
	%\\\no&\qquad\qquad\qquad\qquad
	%\Big\|\frac{\partial^{c_1+\cdots+c_d}\bar{u}}{\partial{x_1^{c_1}}\cdots\partial{x_d^{c_d}}}\Big\|_{L^{p_{l-1}}(\mathbb R^{d-1})}
	%%\\\no&
	%\Big\|\gamma_{\alpha}\frac{\partial^{e_1+\cdots+e_d}v}{\partial{x_1^{e_1}}\cdots\partial{x_d^{e_d}}}\Big\|_{L^{p_l}}
	%\\\no
	%&\qquad\leq \|m\|_{C^{k+1}}\Big\|g_{\lambda(B_1)}\Big\|_{W^{n_1,p_1}}\cdots\Big\| g_{\lambda(B_{l-2})}\Big\|_{W^{n_{l-2},p_{l-2}}}\Big\|\bar{u}\Big\|_{W^{n_{l-1},p_{l-1}}}\Big\|\gamma_{\alpha}v\Big\|_{W^{n_l,p_l}}\\\no
	&\qquad\leq C\|m\|_{C^{k+1}(\mathbb{R}^2)}\|g_{\lambda(B_1)}\|_{H^k}\cdots\|y_{\lambda(B_{l-2})}\|_{H^k}\|\bar{u}\|_{H^k}\|\gamma_{\alpha}v\|_{H^k}.
	\end{align*}
	The cases when $|P|$, $c_1+\cdots+c_d=0$ or $e_1+\cdots+e_d=0$ can be treated analogously.

	Finally, for variations in $v$, we write
	$m( q, u)(v + \bar{v})- m( q, u)v = m( q, u)\bar{v}$, and, therefore, 
\begin{eqnarray*}
	\| m( q, u))\bar{v}\|_{L^2}&\leq &\|m\|_{L^{\infty}}\|\bar{v}\|_{L^2},\\
	\|\gamma_{\alpha} m( q, u)\bar{v}\|_{L^2} &\leq&\|m\|_{L^{\infty}}\|\gamma_{\alpha}\bar{v}\|_{L^2}.
	\end{eqnarray*}
	By the general Leibniz rule,
		\begin{align*}
	&\frac{\partial^{k}}{\partial{x_1^{a_1}}\cdots \partial{x_d^{a_d}}}\big(m(q,u)\bar{v}\big)=
	%\\\no&\qquad
	\sum_{b_d=0}^{a_d}\begin{pmatrix}
	a_d \\ b_d
	\end{pmatrix}\cdots\sum_{b_{1}=0}^{a_{1}}\begin{pmatrix}
	a_{1} \\ b_{1}
	\end{pmatrix}\frac{\partial^{b_1+\cdots+b_d}m}{\partial{x_1^{b_1}}\cdots\partial{x_d^{b_d}}} \frac{\partial^{a_1-b_1+\cdots+a_d-b_d} \bar{v}}{\partial{x_1^{a_1-b_1}}\cdots\partial{x_d^{a_d-b_d}}},
	\end{align*}
	where ${a_j \choose b_j}=\dfrac{a_j!}{b_j!(a_j-b_j)!}$.  Let $s=b_1+\cdots+b_d$, $(g_1(x),g_2(x))=(q(x),u(x))$. We use the Higher Chain Formula, for each $i$ in the set $J_s$ of integers $1$, $2$, ..., $s$. Let again $t_i$ denote one of the independent variables $x_1$, ..., $x_d$. We consider a partition of  $J_s$. % is a family of pairwise disjoint nonempty subsets of $J_s$ whose union is $J_s$. Sets in a partition are called blocks. A block function is to assign a label to each block of a partition. 
	The set of all block functions from a partition $P$ of $J_s$ into $J_2$ is
	 $P_2$ and  $P_s$ is the set of all partitions of $J_s$.  Then
	\begin{equation*}
	\frac{\partial^s m(q,u)}{\partial{t_1}\cdots\partial{t_s}}
	=\sum_{P\in P_s}\sum_{\lambda\in P_{2}}\Big\{ \Big( \prod_{B\in P}\frac{\partial}{\partial{g_{\lambda(B)}}}\Big)m   \Big\}
	\Big\{   \prod_{B\in P} \Big[ \Big(\prod_{b\in B} \frac{\partial}{\partial{t_b}}\Big) g_{\lambda(B)} \Big]  \Big\}.
	\end{equation*}
	Thus, for fixed binomial coefficients ${a_j\choose b_j}$, $j=1$, ...,$d$, for fixed partition $P\in P_s$ and  block function $\lambda\in P_2$ we need to estimate 
	$$\Big\{ \Big( \prod_{B\in P}\frac{\partial}{\partial{g_{\lambda(B)}}}\Big)m   \Big\}
	\Big\{   \prod_{B\in P} \Big[ \Big( \prod_{b\in B}\frac{\partial}{\partial{t_b}}\Big) g_{\lambda(B)} \Big]  \Big\}\frac{\partial^{a_1-b_1+\cdots+a_d-b_d}}{\partial{x_1^{a_1-b_1}}\cdots\partial{x_d^{a_d-b_d}}} \bar{v}$$
	in  $L^2(\mathbb{R}^d)$ and $L_{\alpha}^2(\mathbb{R})\otimes L^2(\mathbb{R}^{d-1})$- norms. To do so we consider the following cases.
	
	\textbf{Case 3.1:}
	If $b_1+\cdots+b_d=0$, then 
	\begin{align*}
	&\Big\|  \Big\{ \Big( \prod_{B\in P}\frac{\partial}{\partial{g_{\lambda(B)}}}\Big)m   \Big\}
	\Big\{   \prod_{B\in P} \Big[ \Big(\prod_{b\in B} \frac{\partial}{\partial{t_b}}\Big) g_{\lambda(B)} \Big]  \Big\}
	\frac{\partial^{a_1-b_1+\cdots+a_d-b_d}}{\partial{x_1^{a_1-b_1}}\cdots\partial{x_d^{a_d-b_d}}} \bar{v}
	\Big\|_{L^2}
	\\\no
	&\qquad =\Big\|m(q,u)\frac{\partial^{a_1+\cdots+a_d}}{\partial{x_1^{a_1}}\cdots\partial{x_d^{a_d}}} \bar{v}\Big\|_{L^2}\leq \|m\|_{L^{\infty}(\mathbb{R}^2)}\Big\| \frac{\partial^{a_1+\cdots+a_d}}{\partial{x_1^{a_1}}\cdots\partial{x_d^{a_d}}} \bar{v} \Big\|_{L^2}\leq \|m\|_{L^{\infty}(\mathbb{R}^2)}\|\bar{v}\|_{H^k}.
	\end{align*}
	
	\textbf{Case 3.2:}
	If $0<b_1+\cdots+b_d<a_1+\cdots+a_d$, let $P=\{B_1,B_2,...,B_{l-1}\}$,  and  $$(n_1,...,n_l)=(|B_1|,...,|B_{l-1}|,a_1-b_1+\cdots+a_d-b_d),$$ 
	then we use Lemmas~\ref{lem2.3.2} and  \ref{lem2.3.3}  with  $\frac{1}{p_i}=(\frac{1}{2}-\frac{k}{d})\frac{1}{l}+\frac{n_i}{d}$, $i=1$, ..., $l$, and obtain 
	\begin{align*}
	&\Big\|\Big\{ \Big( \prod_{B\in P}\frac{\partial}{\partial{g_{\lambda(B)}}}\Big)m   \Big\}
	\Big\{   \prod_{B\in P} \Big[ \Big(\prod_{b\in B} \frac{\partial}{\partial{t_b}}\Big) g_{\lambda(B)} \Big]  \Big\}\frac{\partial^{a_1-b_1+\cdots+a_d-b_d}}{\partial{x_1^{a_1-b_1}}\cdots\partial{x_d^{a_d-b_d}}} \bar{v}\Big\|_{L^2}\\\no
	&\qquad\leq \|m^{(l-1)}\|_{L^{\infty}(\mathbb{R}^2)}\Big\|\Big\{   \prod_{B\in P} \Big[ \Big(\prod_{b\in B} \frac{\partial}{\partial{t_b}}\Big) g_{\lambda(B)} \Big]  \Big\} \frac{\partial^{a_1-b_1+\cdots+a_d-b_d}}{\partial{x_1^{a_1-b_1}}\cdots\partial{x_d^{a_d-b_d}}} \bar{v} \Big\|_{L^2}\\\no
	&\qquad\leq \|m\|_{C^{l-1}(\mathbb{R}^2)}
	\Big\|\frac{\partial^{a_1-b_1+\cdots+a_d-b_d} \bar{v}}{\partial{x_1^{a_1-b_1}}\cdots\partial{x_d^{a_d-b_d}}} \Big\|_{L^{p_l}}
	\prod_{i=1}^{l-1}\Big\|\Big(\prod_{b\in B_i}\frac{\partial}{\partial{t_b}}\Big) g_{\lambda(B_i)}\Big\|_{L^{p_i}}%\cdots\Big\|\Big(\prod_{b\in B_{l-1}} \frac{\partial}{\partial{t_b}}\Big) g_{\lambda(B_{l-1})}\Big\|_{L^{p_{l-1}}}
	%&\hskip2cm
	%\Big\|\frac{\partial^{a_1-b_1+\cdots+a_d-b_d} \bar{v}}{\partial{x_1^{a_1-b_1}}\cdots\partial{x_d^{a_d-b_d}}} \Big\|_{L^{p_l}(\mathbb R^{d-1})}
	\\\no
	&\qquad\leq \|m\|_{C^{l-1}(\mathbb{R}^2)}\|\bar{v}\|_{W^{n_l,p_l}}
	\prod_{i=1}^{l-1} \|g_{\lambda(B_i)}\|_{W^{n_i,p_i}}%\cdots\|g_{\lambda(B_{l-1})}\|_{W^{n_{l-1},p_{l-1}}}
	\\\no
	&\qquad\leq \|m\|_{C^{l-1}(\mathbb{R}^2)}
	\prod_{i=1}^{l-1}
	\|g_{\lambda(B_i)}\|_{H^k}%\cdots\|g_{\lambda(B_{l-1})}\|_{H^k(\mathbb R^{d-1})}
	\|\bar{v}\|_{H^k}.
	\end{align*}
	
	\textbf{Case 3.3:}
	If $b_1+\cdots+b_d=a_1+\cdots+a_d$, when $|P|=1$, $g_{\lambda(B)}=g_i$, $i=1$, $2$, or $3$, we use Sobolev embedding $H^k(\mathbb{R}^d)\hookrightarrow L^{\infty}(\mathbb{R}^d)$  and obtain
	\begin{align*}
	&\Big\|\Big\{ \Big( \prod_{B\in P}\frac{\partial}{\partial{g_{\lambda(B)}}}\Big)m   \Big\}
	\Big\{   \prod_{B\in P} \Big[ \Big(\prod_{b\in B} \frac{\partial}{\partial{t_b}}\Big) g_{\lambda(B)} \Big]  \Big\}\frac{\partial^{a_1-b_1+\cdots+a_d-b_d}\bar{v}}{\partial{x_1^{a_1-b_1}}\cdots\partial{x_d^{a_d-b_d}}} \Big\|_{L^2}\\\no
	&\qquad\leq \|m^{(1)}\|_{L^{\infty}(\mathbb{R}^2)}\|\bar{v}\|_{L^{\infty}}\Big\|\frac{\partial^{a_1+\cdots+a_d}g_i}{\partial{x_1^{a_1}}\cdots\partial{x_d^{a_d}}} \Big\|_{L^2}\leq \|m\|_{C^{1}(\mathbb{R}^2)}\|\bar{v}\|_{H^k}\|g_i\|_{H^k}.
	\end{align*}
	When $|P|=l>1$, let $P=\{B_1,B_2,...,B_{l}\}$ and  $(n_1,...,n_l)=(|B_1|,...,|B_{l}|)$. 
	Lemmas~\ref{lem2.3.2} and \ref{lem2.3.3} with $\frac{1}{p_i}=(\frac{1}{2}-\frac{k}{d})\frac{1}{l}+\frac{n_i}{d}$, $i=1$, ...,$l$, and the Sobolev embedding $H^k(\mathbb{R}^d)\hookrightarrow L^{\infty}(\mathbb{R}^d)$ imply
	\begin{align*}
	&\Big\|\Big\{ \Big( \prod_{B\in P}\frac{\partial}{\partial{g_{\lambda(B)}}}\Big)m   \Big\}
	\Big\{   \prod_{B\in P} \Big[\Big( \prod_{b\in B} \frac{\partial}{\partial{t_b}}\Big) g_{\lambda(B)} \Big]  \Big\}\frac{\partial^{a_1-b_1+\cdots+a_d-b_d}}{\partial{x_1^{a_1-b_1}}\cdots\partial{x_d^{a_d-b_d}}} \bar{v}\Big\|_{L^2}
	\\\no
	&\qquad\leq \|m^{(l)}\|_{L^{\infty}(\mathbb{R}^2)}\|\bar{v}\|_{L^{\infty}}\Big\|\Big\{   \prod_{B\in P} \Big[ \Big(\prod_{b\in B} \frac{\partial}{\partial{t_b}}\Big) g_{\lambda(B)} \Big]  \Big\}  \Big\|_{L^2}\\\no
	&\qquad\leq \|m\|_{C^{l}(\mathbb{R}^2)}\|\bar{v}\|_{H^k}
	\prod_{i=1}^{l}
	\Big\|\Big(\prod_{b\in B_i} \frac{\partial}{\partial{t_b}}\Big) g_{\lambda(B_i)}\Big\|_{L^{p_i}}%\cdots\Big\|\Big(\prod_{b\in B_{l}} \frac{\partial}{\partial{t_b}}\Big) g_{\lambda(B_{l})}\Big\|_{L^{p_{l}}(\mathbb R^{d-1})}
	%\\\no &
	\leq \|m\|_{C^{l}(\mathbb{R}^2)}\|\bar{v}\|_{H^k}\prod_{i=1}^{l}\|g_{\lambda(B_i)}\|_{W^{n_i,p_i}}%\cdots\|g_{\lambda(B_{l})}\|_{W^{n_{l},p_{l}}}
	\\\no&
	\qquad \leq \|m\|_{C^{l}(\mathbb{R}^2)}\|\bar{v}\|_{H^k}
	\prod_{i=1}^{l}\ \|g_{\lambda(B_i)}\|_{H^k}%\cdots\|g_{\lambda(B_{l})}\|_{H^k(\mathbb R^{d-1})}
	.
	\end{align*}
	
	Similarly, let $l=|P|+1$, $P=\{B_1,...,B_{l-1}\}$,  $(n_1,...,n_l)=(|B_1|,...,
	|B_{l-1}|,a_1-b_{1}+\cdots+a_{d}-b_d)$ and  $\gamma_{\alpha}=e^{\alpha  x_1}$, then
	\begin{align*}
	&\Big\|\gamma_{\alpha}\Big\{ \Big( \prod_{B\in P}\frac{\partial}{\partial{g_{\lambda(B)}}}\Big)m   \Big\}
	\Big\{   \prod_{B\in P} \Big[\Big( \prod_{b\in B} \frac{\partial}{\partial{t_b}}\Big) g_{\lambda(B)} \Big]  \Big\}
	\frac{\partial^{a_1-b_1+\cdots+a_d-b_d}\bar{v}}{\partial{x_1^{a_1-b_1}}\cdots\partial{x_d^{a_d-b_d}}} \Big\|_{L^2}
	\\
	\no
	&\leq\|m^{(|P|)}\|_{L^{\infty}(\mathbb{R}^2)}
	\Big\|\gamma_{\alpha}\frac{\partial^{a_1-b_1+\cdots+a_d-b_d} \bar{v}}{\partial{x_1^{a_1-b_1}}\cdots\partial{x_d^{a_d-b_d}}}\Big\|_{L^{p_l}}
	\prod_{i=1}^{l-1}\Big\|\Big(\prod_{b\in B_i} \frac{\partial}{\partial{t_b}}\Big) g_{\lambda(B_i)}\Big\|_{L^{p_i}}
%	\cdots\Big\|	\Big(\prod_{b\in B_{l-1}}\frac{\partial}{\partial{t_b}}\Big) g_{\lambda(B_{l-1})}\Big\|_{L^{p_{l-1}}(\mathbb R^{d-1})}
	%\\\no&\hskip2cm
	%\Big\|\gamma_{\alpha}\frac{\partial^{a_1-b_1+\cdots+a_d-b_d} \bar{v}}{\partial{x_1^{a_1-b_1}}\cdots\partial{x_d^{a_d-b_d}}}\Big\|_{L^{p_l}(\mathbb R^{d-1})}
	\\\no
	&\leq \|m\|_{C^{l-1}(\mathbb{R}^2)}\|\gamma_{\alpha}\bar{v}\|_{W^{n_l,p_l}}
	\prod_{i=1}^{l-1} \|g_{\lambda(B_1)}\|_{W^{n_1,p_1}}
	%\cdots\|g_{\lambda(B_{l-1})}\|_{W^{n_{l-1},p_{l-1}}}
%	\\\no&
	\leq\|m\|_{C^{l-1}(\mathbb{R}^2)}\|\gamma_{\alpha}\bar{v}\|_{H^k}
	\prod_{i=1}^{l-1} \|g_{\lambda(B_i)}\|_{H^k}.
	%\cdots\|g_{\lambda(B_{l-1})}\|_{H^k(\mathbb R^{d-1})}\|\gamma_{\alpha}\bar{v}\|_{H^k(\mathbb R^{d-1})}.
	\end{align*}
	The case of weighted norm when $b_1+\cdots+b_d=0$, or $b_1+\cdots+b_d=a_1+\cdots+a_d$  can be considered similarly.
	
	Using  Lipschitz estimates for variations in $q$, $u$, and $v$, one can easily
	show that the mappings are locally Lipschitz on the given sets in $H^k(\mathbb R^d)$ and  in $H_{\alpha}^k(\mathbb R^d)$, therefore 
	%On $\mathcal{H}=H^k(\mathbb R^d)\bigcap H_{\alpha}^k(\mathbb R^d)$, by using the fact that $\|\cdot\|_0\leq \|\cdot\|_{\mathcal{H}}$ and $\|\cdot\|_{\alpha}\leq \|\cdot\|_{\mathcal{H}}$, the estimates can also be achieved, 
	 the mappings are also locally Lipschitz on the given sets $\mathcal{H}=H^k(\mathbb R^d)\bigcap H_{\alpha}^k(\mathbb R^d)$.
\end{proof}

\end{document}